\documentclass[a4paper,11pt]{article}
\usepackage[utf8]{inputenc}
\usepackage{amsfonts, amssymb, amsmath, amsthm, fdsymbol}
\usepackage{color}
\usepackage{makeidx}

\usepackage{srcltx}

\usepackage{epsfig,graphicx}
\usepackage[margin=1.25in]{geometry}

\newtheorem{theorem}{Theorem}
\renewcommand{\thetheorem}{\Alph{theorem}}

\newtheorem{lemma}{Lemma}
\newtheorem{claim}[lemma]{Claim}
\newtheorem{question}[lemma]{Question}
\newtheorem{conjecture}[lemma]{Conjecture}
\newtheorem{corollary}[lemma]{Corollary}
\newtheorem{proposition}[lemma]{Proposition}
\newtheorem{remark}[lemma]{Remark}
\newtheorem{remarks}[lemma]{Remarks}
\newtheorem*{main-conjecture}{Main conjecture}

\theoremstyle{definition}
\newtheorem{definition}[lemma]{Definition}

\def\DD{{\mathbb D}}  
\def\ZZ{{\mathbb Z}}

\def\RR{{\mathbb R}}

\def\NN{{\mathbb N}}

\def\cC{\mathcal{C}}
\def\cO{\mathcal{O}}
\def\cI{\mathcal{I}}
\def\cJ{\mathcal{J}}
\def\cD{\mathcal{D}}

\def\la{\lambda}

\def\eps{\epsilon}

\def\supp{{\rm {supp}}}

\def\interior{\operatorname{Interior}}

\makeatletter
\renewcommand{\l@section}{\@dottedtocline{2}{3.8em}{3.2em}}
\renewcommand{\l@subsection}{\@dottedtocline{3}{3.8em}{3.2em}}
\newcommand{\subsectionruninhead}{\@startsection{subsection}{2}{0mm}{-\baselineskip}{-0mm}{\bf\large}}
\newcommand{\subsubsectionruninhead}{\@startsection{subsubsection}{3}{0mm}{-\baselineskip}{-0mm}{\bf\normalsize}}
\makeatother

\makeindex

\begin{document}

\title{Mildly dissipative diffeomorphisms of the disk\\ with zero entropy}
\author{Sylvain Crovisier\footnote{S.C. was partially supported by the ERC project 692925 -- NUHGD; S.C. and E.P. were partially supported
by the Balzan Research Project of J. Palis. E.P. received support of NSF via grant DMS-1956022.}, Enrique Pujals, Charles Tresser}

\maketitle

\begin{abstract}
We discuss the dynamics of smooth diffeomorphisms of the disc with vanishing topological entropy which satisfy the mild dissipation property introduced in~\cite{CP}. This class contains the H\'enon maps with Jacobian up to $1/4$.
We prove that these systems are either (generalized) Morse Smale
or infinitely renormalizable. In particular we prove a conjecture of Tresser is this class: any diffeomorphism in the interface between
the sets of systems with zero and positive entropy admits doubling cascades.
This generalizes a well known consequence
of Sharkovskii's theorem for interval maps, to mild dissipative diffeomorphisms of the disk with zero entropy.
\end{abstract}

\tableofcontents

\section{Introduction}

The space of diffeomorphisms splits into two classes:
those with zero entropy and those with positive entropy (by which we always mean topological entropy). The former contains Morse-Smale diffeomorphisms: their nonwandering set is formed by finitely many periodic points.
The latter contains the systems exhibiting a transverse homoclinic orbit, \emph{i.e.,} an orbit which accumulates on the past and on the future
on a same periodic orbit and which persists under small perturbations: the nonwandering set is uncountable. In particular, both classes contain $C^1$-open sets. It has been proved that Morse-Smale systems and those having a transverse homoclinic intersection define a $C^1$-dense open set~\cite{PuSa,C}.
However, even in the $C^1$ context, the dynamics of systems belonging to the interface
of these two classes is not well understood, while in higher topologies almost nothing is known.
One goal would be to characterize the systems in the boundary of the zero entropy class  and to try to identify, if it exists, the universal phenomenon that generates entropy.

In a more general context
our central question here is the transition between simple and complicated dynamics as seen from two different angles:
{\color{black} the theoretical perspective and 
the applied one.
The transition to chaos has been observed in a variety of natural and engineering contexts, which are modeled by dissipative flows,
and that can be reduced to discrete-time systems by considering their trace on Poincar{\'e} sections.}
This happens both in some forced damped oscillators for which one observes the formation of horseshoes
and in autonomous flows where the chaos is linked to a Shil'nikov bifurcation in $C^\omega$ regularity~\cite{shilnikov}, or even $C^{1+\text{Lip}}$ regularity~\cite{Tr}.

We can think about two related problems when considering this central question:  
\smallskip

-- \emph{the transition to chaos} (\emph{i.e.,} the transition from zero to positive entropy),\footnote{In dimensions $1$ and $2$, the topological entropy is continuous with respect to the $C^{\infty}$-topology by~\cite{Mi,Ka,Yo}. In particular, there is no jump in entropy at the transition. For $C^1$ families on the interval, the transition to positive entropy requires infinitely many period doubling bifurcations \cite{BlHa}.}
\smallskip

-- \emph{the transition from finitely to infinitely many periods of hyperbolic periodic orbits.}
\hspace{-2cm}\mbox{} 
\smallskip

In the one-dimensional context, the natural ordering on the interval allows the development of a ``combinatorial theory", which describes properties of orbits related to this ordering. An example of this is Sharkovskii's hierarchy of periodic orbits~\cite{sharkovskii}; it implies in particular that any system with zero entropy only admits periodic points of period $2^n$. One paradigmatic example is the case of unimodal maps: Coullet-Tresser and independently Feigenbaum conjectured~\cite{CT,F} that the ones in the boundary of the zero entropy class are a limit of a period doubling cascade with universal metric property under rather mild smoothness assumptions and are infinitely renormalizable (see also~\cite{Chandra} and~\cite{CP2}).

In those papers a renormalization operator  was introduced\footnote {In \cite{CT}, Coullet and Tresser recognized that operator as similar to the renormalization operator introduced in Statistical Mechanics by Kennet Wilson
following a prehistory in the context of high energy physics.} and it was shown that the numerical observations could be explained if this operator, defined on an appropriate space of functions, would have a hyperbolic fixed point. The central results of the universality theory for unimodal maps have been proved by Lyubich~\cite{L} for analytic unimodal maps and extended to lower regularity in~\cite{FMP}.  Partial results about multimodal maps and the associated transition to chaos have been obtained  by many authors (see \emph{e.g.,} \cite{MiTr} and references cited or citing). 

\paragraph{a -- Mildly dissipative diffeomorphisms of the disc.}
The first step towards these universal goals in higher dimension, is to consider embeddings of the disc $\mathbb{D}$. 
These embeddings can be extended as diffeomorphisms of the two-dimensional sphere by gluing a repelling disc, 
as detailed in~\cite{BoFr}.\footnote{Notice that \cite{BoFr} is the first paper studying cascades of period doubling in dimensions 1 and 2.}
Therefore, to avoid notations we will call \emph{dissipative diffeomorphisms of the discs}
the $C^r$ embeddings $f\colon \DD\to f(\DD)\subset \text{Interior}(\DD)$ with $r>1$,
such that $|\det(Df(x))|<1$ for any $x\in \DD$.
Observe that any $f$-invariant ergodic probability measure $\mu$
which is not supported on a hyperbolic sink
has one negative Lyapunov exponent
and  one which is non-negative.
In particular,
for $\mu$-almost every point $x$, there exists a well-defined one-dimensional
stable manifold $W^s(x)$.
We denote $W^s_\DD(x)$ the connected component of $W^s(x)\cap \DD$
containing $x$.
We strengthen the notion of dissipation:

\begin{definition}\label{SD defi}
A dissipative diffeomorphism of the disc is \emph{mildly dissipative}\index{dissipation, mild and $\gamma$-dissipation} if
for any ergodic measure $\mu$ not supported on a hyperbolic sink, and for
$\mu$-almost every $x$, the curve $W^s_\DD(x)$ separates $\DD$.
\end{definition}

This notion was introduced for any type of surface\footnote {In \cite{CP} these systems are called \emph{strongly dissipative diffeomorphisms} since many results were only applied for systems with very small Jacobian; in the  context of the disc we call them \emph{mildly dissipative}, since there are classes of diffeomorphisms with not such small Jacobian, as the H\'enon maps, that satisfy the main property of the definition.} in \cite{CP}, where
it is shown that mild dissipation is satisfied for large classes of systems, 
For instance it holds for $C^2$ open sets of diffeomorphisms of the disc,
and for polynomial automorphisms of $\mathbb{R}^2$ whose Jacobian is sufficiently close to $0$,
including the diffeomorphisms from the H\'enon family with Jacobian of modulus less than $1/4$
(up to restricting to an appropriate trapped disc).
This class captures certain properties of one-dimensional maps but keeps two-dimensional features showing all the well known complexity of dissipative surface diffeomorphisms. The dynamics of the new class, in some sense, is intermediate between one-dimensional dynamics and general surface diffeomorphisms.

\paragraph{b -- Renormalization.}
As  mentioned before, the essential mechanism for interval endomorphisms in the transition to chaos are the period doubling cascades; the main universal feature of systems in the boundary of zero entropy is that they are  infinitely renormalizable.
A similar result can be proved for mildly dissipative diffeomorphisms of the disc that belong to the boundary of the zero entropy class.

A diffeomorphism $f$ of the disc is \emph{renormalizable}\index{renormalization, infinite renormalization} if
there exists a compact set $D\subset \mathbb{D}$ homeomorphic to the unit disc
and an integer $k>1$ such that
$f^i(D)\cap D=\emptyset$ for each $1\leq i<k$ and $f^k(D)\subset D$.
Moreover $f$ is \emph{infinitely renormalizable} if there exists an infinite nested sequence of renormalizable attracting periodic domains with arbitrarily large periods.
For instance \cite{GvST} built a $C^\infty$-diffeomorphism which has vanishing entropy
and is infinitely renormalizable
(see also Figure~\ref{f.odometer}).

\begin{theorem}\label{t.theoremA}
For any mildly dissipative diffeomorphism $f$ of the disc whose topological entropy vanishes,
\begin{itemize}
\item[--] either $f$ is renormalizable,
\item[--] or any forward orbit of $f$ converges to a fixed point.
\end{itemize}
\end{theorem}

Morse-Smale diffeomorphisms (whose non-wandering set is a finite set of hyperbolic periodic points)
are certainly not infinitely renormalizable.
It is natural to generalize this class of diffeomorphisms
in order to allow bifurcations of periodic orbits.

\begin{definition}\label{d.generalizedMS}
A diffeomorphism is \emph{generalized Morse-Smale}\index{Morse-Smale and generalized Morse-Smale diffeomorphism} if:
\begin{itemize}
\item[--] the $\omega$-limit set of any forward orbit is a periodic orbit,
\item[--] the $\alpha$-limit set of any backward orbit in $\mathbb{D}$ is a periodic orbit,
\item[--] the period of all the periodic orbits is bounded by some $K>0$.
\end{itemize}
\end{definition}
Clearly these diffeomorphisms have zero entropy.
We will see in Section~\ref{s.gms} that the set of mildly dissipative generalized Morse-Smale diffeomorphisms of the disc
is $C^1$ open. A stronger version of Theorem~\ref{t.theoremA}, proved in Section~\ref{s.renormalize} (see Theorem \ref{t.renormalize-prime}),  states that in the renormalizable case there exist finitely many renormalizable domains  such  that the limit set in their  complement consists of fixed points. That version  implies:

\begin{corollary}\label{c.dichotomy0}
A mildly dissipative diffeomorphism of the disc with zero entropy is
\begin{itemize}
\item[--] either infinitely renormalizable,
\item[--] or generalized Morse-Smale.
\end{itemize}
\end{corollary}

\paragraph{c -- Boundary of zero entropy.}
The set of $C^r$ diffeomorphisms, $r>1$, with positive entropy is $C^1$ open (see~\cite{Ka}).
One may thus consider how positive entropy appears: a diffeomorphism belongs to the boundary of
zero entropy if its topological entropy vanishes, but it is the $C^1$ limit of diffeomorphisms with
positive entropy.
The previous results immediately give:

\begin{corollary}\label{c.infinitely-renormalizable}
A mildly dissipative diffeomorphism of the disc in the boundary of zero entropy
is infinitely renormalizable.
\end{corollary}

We may ask if the converse also holds:

\begin{question}\label{q.approximate} In the space of mildly dissipative $C^r$ diffeomorphisms of the disc, $r>1$,
can one approximate any diffeomorphism exhibiting periodic orbits of arbitrary large period
by diffeomorphisms with positive entropy?
\end{question}

This would imply that generalized Morse-Smale
diffeomorphisms are the mildly dissipative diffeomorphisms of the disc  with robustly vanishing entropy.
Question~\ref{q.approximate} has a positive answer if one considers $C^1$-approximations of $C^2$-diffeomorphisms
(this is essentially Corollary 2 in~\cite{PuSa2}).
In a similar spirit, it is unknown (even in the $C^1$-topology) if diffeomorphisms with zero entropy are limit of generalized Morse-Smale diffeomorphisms.

\begin{question}\label{q.approximate2} In the space of mildly dissipative $C^r$ diffeomorphisms of the disc, $r>1$,
can one approximate any diffeomorphism with zero entropy by a generalized Morse-Smale diffeomorphism?\footnote{For issues related to the two last questions in the context of interval maps, see \emph{e.g.,}~\cite{HuTr} and references therein.}
\end{question}

\paragraph{d -- Decomposition of the dynamics with zero entropy.}
Let us recall that Conley's theorem (see~\cite[Chapter 9.1]{robinson}) decomposes the dynamics of homeomorphisms:
the chain-recurrent set splits into disjoint invariant compact sets called
\emph{chain-recurrence classes}.
We now describe the dynamics inside the chain-recurrence classes
of mildly dissipative diffeomorphisms with zero entropy.

Let $h$ be a homeomorphism of the Cantor set $\mathcal{K}$.
One considers partitions of the form
$\mathcal{K}=K\cup h(K)\cup\dots\cup h^{p-1}(K)$ into clopen sets that
are cyclically permuted by $h$.
We say that $h$ is an \emph{odometer}\index{odometer, generalized odometer} if there exist such partitions into clopen sets
with arbitrarily small diameters.
The set of the periods $p$ is a multiplicative semi-group which uniquely determines the odometer.
Each odometer is minimal and preserves a unique probability measure
(this allows to talk about almost every point of the odometer $x\in \mathcal{K}$).
Figure~\ref{f.odometer} represents a diffeomorphism of the disc which induces
an odometer on an invariant Cantor set.

\begin{corollary}\label{c.structure}
Let $f$ be a mildly dissipative diffeomorphism of the disc with zero entropy.
Then any chain-recurrence class $\mathcal{C}$ of $f$ is:
\begin{itemize}
\item[--] either \emph{periodic}: there exists a compact connected set $C$
and an integer $n\geq 1$ such that $\mathcal{C}=C\cup\dots\cup f^{n-1}(C)$
and any point in $\color{black} C$ is fixed under $f^n$,
\item[--] or a \emph{generalized odometer}:\index{odometer, generalized odometer}
there exists an odometer $h$ on the Cantor set $\mathcal{K}$ and a
continuous subjective map $\pi\colon \mathcal{C}\to \mathcal{K}$ such that
$\pi\circ f=h\circ \pi$ on $\mathcal{K}$. Moreover almost every point $z\in \mathcal{K}$
has at most one preimage under $\pi$.
\end{itemize}
In addition:
\begin{itemize}
\item[--] Each generalized odometer is a quasi-attractor, \emph{i.e.,} admits a basis of open neighborhoods $U$
satisfying $f(\overline U)\subset U$.
\item[--] The union of the generalized odometers is an invariant compact set $\Lambda$.
Outside any neighborhood of $\Lambda$ the set of periods of the periodic orbits is finite.
\end{itemize}
\end{corollary}

\begin{figure}
\begin{center}
\includegraphics[width=7cm,angle=0]{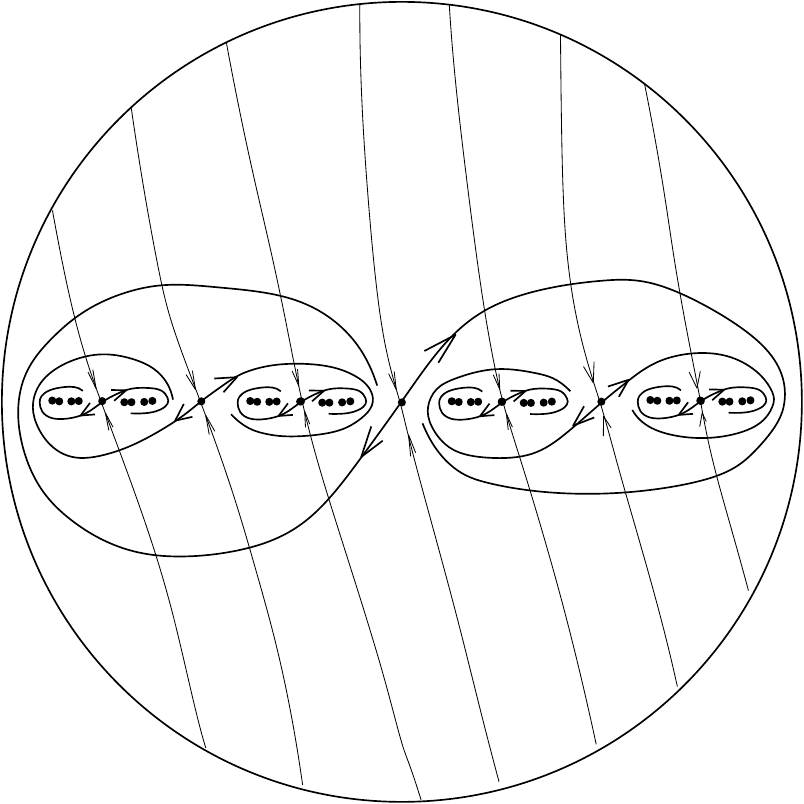}
\end{center}
\caption{Dynamics exhibiting saddle orbits of each period $2^n$, $n\in \mathbb{N}$,
and one odometer.\label{f.odometer}}
\end{figure}

Corollary~\ref{c.structure} can be compared to a recent result by Le Calvez and Tal~\cite{lecalvez-tal} about transitive sets of homeomorphisms
of the $2$-sphere with zero entropy. The methods there are quite different from ours.
Note that the dissipation hypothesis is essential: for conservative systems with zero entropy, the dynamics is modeled on integrable systems,
see~\cite{franks-handel}.
\medskip

We do not know if there exist examples of systems exhibiting generalized odometers
which are not conjugate to odometers (\emph{i.e.,} such that the map $\pi$ is not injective).
Another problem concerns the cardinality of these classes:

\begin{question}
Does there exist a mildly dissipative diffeomorphism of the disc with zero entropy
and infinitely many generalized odometers?\footnote{A degenerate $C^1$ example can be extracted from the Denjoy-like example in \cite{BoGamLiTr}.}
\end{question}

The answer to this question is not known for general one-dimensional $C^r$-endo\-mor\-phism. However for multimodal endomorphisms of the interval, the {\color{black} number of} nested sequences of infinitely renormalizable domains is bounded by the number of critical points. In particular, generically the number of nested renormalizable domains is finite. This type of result is not known for surface diffeomorphisms.

\paragraph{e -- Periods of renormalizable domains.}

For one-dimensional multimodal maps with zero entropy, Sharkovskii's Theorem~\cite{sharkovskii} implies that the period of the renormalizable domains are powers of $2.$ In the context of mildly dissipative diffeomorphisms this cannot be true, but a similar result holds when one considers renormalizable domains with ``large period'':

\begin{theorem}\label{t.period}
Let $f$ be a mildly dissipative diffeomorphism of $\DD$ with zero topological entropy
and infinitely renormalizable.
There exist an open set $W$ and $m\geq 1$ such that:
\begin{itemize}
\item[--] $W$ is a finite disjoint union of topological discs that are trapped by $f^m$,
\item[--] the periodic points in $\DD\setminus W$ have period bounded by $m$,
\item[--] any renormalizable domain $D\subset W$ of $f^m$ has period
of the form $2^k$;
$D$ is associated to a sequence of renormalizable domains
$D=D_k\subset\dots \subset D_1\subset  W$ of $f^m$ with period $2^{k},\dots, 2$.
\end{itemize}
\end{theorem}
In other words, the period of a renormalizable domain is eventually a power of $2$, meaning that, after replacing $f$ by an iterate, the period of all the renormalizable domains are powers of $2$. As explained in the paragraph {\em summary of the proof} below, the proof of Theorem \ref{t.period} uses some rigidity argument. 

This implies an  analogue of Sharkovskii's theorem for surface diffeomorphisms:

\begin{corollary}\label{c.period}
Let $f$ be a mildly dissipative diffeomorphism of the disc with zero topological entropy.
{\color{black} If the set of periods $\operatorname{Per}(f)$\index{period, set of periods $\operatorname{Per}(f)$} of the periodic orbits of $f$ is infinite,
there exist two finite families of integers $\{n_1,\dots,n_k\}_{k\geq 1}$
and $\{m_1,\dots,m_\ell\}_{\ell \geq 1}$}
\begin{equation}\label{e.period}
\operatorname{Per}(f)=\{n_1,\dots,n_k\}\cup \left\{m_i.2^j, \; 1\leq i\leq \ell \text{ and } j\in \mathbb{N}\right\}.
\end{equation}
\end{corollary}

In particular, in the setting of mildly dissipative diffeomorphisms, we get an affirmative answer to the following conjecture that was formulated by one of us in 1983, and mentioned verbally since then, but appeared in a text (see~\cite{GT}) only a few years after.

\begin{conjecture}[Tresser]
In the space of $C^k$ orientation preserving embeddings of the $2$-disk, with $k>1$,
which are area contracting, generically,
maps which belong to the boundary of positive topological entropy have a set of periodic orbits which, except for a finite subset, is made of an infinite number of periodic orbits with periods, $m.2^k$ for a given $m$ and all $k \geq 0.$
\end{conjecture}

We note that it is possible to realize any set of the form~\eqref{e.period}
as the set of periods a mildly dissipative diffeomorphism of the disc having zero entropy,
whereas a diffeomorphism with positive entropy has a different set of periods
(it always contains a set of the form $k.\NN^*$).
In a more general framework, Theorem~\ref{t.period} is  false if the dynamics is
conservative (an integrable twist in the disc may admit all the periods and {\color{black} have}
vanishing entropy).

Previous works in the direction to develop a forcing theory as it follows from Shar\-kov\-skii's theorem  (see \cite{GST}) used the ideas and language of braids.  For surface diffeomorphisms, a periodic orbit defines a braid type that in turns can or not, force the positivity of topological entropy
(the complement of an orbit of period three or larger in the disc
can be equipped with a hyperbolic structure from where Nielsen-Thurston theory can be developed).
In that sense, permutations are replaced by braids, but the discussion in braid terms cannot be reduced to a discussion in terms of periods as the conjecture formulates.

\paragraph{f -- H\'enon family.}
Given any $C^r$ endomorphisms $h$ of an interval $I\subset \RR$ and $b_0>0$, there exists a disc $\DD=I\times (-\eps,\eps)$
such that the maps defined by
\begin{equation}\label{e.extension}
f_{b}(x,y)=(h(x)+y, -bx),\quad
\text{for}\quad 0<|b|<b_0
\end{equation}
are dissipative diffeomorphisms of $\DD$. 
The (real) H\'enon family is a particular case where $h$ is a quadratic polynomial.\footnote{After studying the Lorenz model for large values of the ``Rayleigh number" $r$ on the advice of David Ruelle, Yves Pomeau presented this joint work at the observatory of Nice where Michel H\'enon  was working. He showed in particular that the time-$t$ map, for $t$ varying from 0 to 1, transforms a well chosen rectangle to an incomplete horseshoe. That night, the legend tells, H\'enon extracted a model of that from his former studies of the conservative case while Pomeau and Ibanez had preferred to focus to a full double covering for which the mathematics are much simpler. ``The most recognition for the least work" 
H\'enon told to Tresser. Later, Coullet and Tresser realized that the H{\'e}non map appears to be in the same universality class for period doubling than the one -dimensional quadratic map: this led them to conjecture in 1976 (see~\cite{CT}) that universal period doubling should be observed in fluids, since H{\'e}non map was built to imitate a Poincar{\'e} map of the Lorenz flow in some parameters ranges, and the quadratic map is the limit as the dissipation goes to infinity, of the H{\'e}non map.}
As mentioned before, the H\'enon family  is  mildly dissipative~\cite{CP} for $0<|b|<1/4$ in restriction to a trapped disc. One can easily analyse the dynamics outside this trapped disc,
therefore all the theorems mentioned above can be applied to the these parameters of the H\'enon family and in particular one gets the following corollary which describes the dynamics on the whole plane $\RR^2$.
Note that this contrasts with the usual global dynamical descriptions of the H\'enon family which suppose $|b|\ll 1$ (see~\cite{BC,dCLM,LM}).

\begin{corollary}\label{c.henon}
Let $f_{b,c}\colon (x,y)\mapsto(x^2+c+y, -bx)$  be a H\'enon map with $|b|\in(0,1/4)$ and $c\in \RR$.
If the topological entropy vanishes, then:
\begin{itemize}
\item[--] for any forward (resp. backward) orbit one of the following cases occurs:
\begin{enumerate}
\item it escapes at infinity, \emph{i.e.,} it leaves any compact set,
\item it converges to a periodic orbit,
\item it accumulates to (a subset of) a generalized odometer;
\end{enumerate}
\item[--] the set of periods has the form described in~\eqref{e.period}.
\end{itemize}
\end{corollary}

\paragraph{g -- Small Jacobian.}

For diffeomorphisms of the disc sufficiently close to an endomorphism of the interval
and whose entropy vanishes, Section \ref{ss.close-endo} proves that the periods of all renormalizable domains (and so the periods of all periodic orbits) are powers of two. 

More precisely, given a  $C^r$ endomorphism of the  interval $f_0$, there exists $b_0>0$ such that for any $0<|b|<b_0$ the diffeomorphism $f_{b}$ is mildly dissipative.  In particular
all the theorems mentioned before can be applied.
Assuming the Jacobian sufficiently small, a stronger property holds:

\begin{theorem}\label{t.small jacobian} Given a
family $(f_b)$ associated to a $C^2$ endomorphism of the  interval
as in~\eqref{e.extension}, there exists $b_0>0$ such that,
for any $b\in (0,b_0)$ and for any diffeomorphism $g$ with zero entropy in a $C^2$-neighborhood of $f_b$,
there exists $n_0\in\NN\cup\{\infty\}$ satisfying
\begin{equation}
\operatorname{Per}(g)=\{2^n, \; n<n_0\}.
\end{equation}
\end{theorem}
In particular, the previous theorem can be applied to the H\'enon family and one recovers one of the results in~\cite{dCLM,LM}.

\paragraph{h -- Some differences with the one-dimensional approach.} In the context of one-dimensional dynamics of the interval, and in particular  for unimodal maps, the  renormalization intervals are built using the dynamics around the turning point: the boundary of the interval contains the closest iterate to the turning point of a repelling orbit (whose period is a power of two) and a preimage of that iterate. For H\'enon maps with small Jacobian,  although there is no notion of turning point, renormalization domains are built  in \cite{dCLM,LM}  using the local stable manifold of a saddle periodic point  of index $1$
and its preimages (those points are the analytic continuations of the repelling points of the one-dimensional map). 

Our approach can rely neither of the notion of turning point nor on being close to well understood one-dimensional dynamics. Our construction is different and uses the structure of the set of periodic points. Following the unstable branches, {\color{black} we build} a skeleton
{\color{black} for} the dynamics that allows {\color{black} one} to construct the trapping regions and the renormalization domains.

\paragraph{i -- The renormalization operator.} In \cite{dCLM,LM} {\color{black} it} is  proved that infinitely renormalizable real H\'enon-like maps with sufficiently small Jacobian 
admit an appropriately defined renormalization operator  {\color{black} (see also~\cite{H} for renormalization with other combinatorics).}
After proper affine rescaling, the dynamics (at the period) on the renormalizable attracting domain converge to a smooth quadratic unimodal map which is nothing {\color{black} but} the hyperbolic fixed point of the renormalization operator for the one-dimensional dynamics.

It is not difficult to construct mildly dissipative diffeomorphisms with zero entropy which are not a priori close to a unimodal map on the interval (for instance, when the first renormalization domain has period larger than two) and in this case the renormalization scheme developed for H\'enon-like maps with small Jacobian would need to be {\color{black} recast}.  Although the present paper does not provide a well-defined renormalization operator
for mildly dissipative diffeomorphism of the disk, it gives the existence of nested renormalization domains and deep renormalizations seem to drive the
system towards the one-dimensional model.
Indeed the renormalization domains eventually have (relative) period two;
moreover the return dynamics on these domains recover
certain smooth properties that are satisfied by diffeomorphism close to the one-dimensional endomorphisms (see Section \ref{ss.uniform stable});
\cite{CP} associates a quotient dynamic which, on these ``deep domains'', induces  an endomorphism of a real tree. That raises the following question:

\begin{question}
Given a sequence of nested renormalizable domains, 
is it true that (after proper rescalings)   
the sequence of return maps generically converges to an unimodal map?
\rm One does not expect to replace ``generic" by ``general" 
because of the expected possible alternate convergence of the renormalizations to more that one fixed point: this happens in dimension 1, see \emph{e.g.,} \cite{MiTr} and also 
\cite{OETr}.
\end{question}

When $f$ is mildly dissipative, the larger Lyapunov exponent
of each generalized odometer $\mathcal{C}$ vanishes, hence the iterates of the derivative of $f$ on $\mathcal{C}$
do not grow exponentially; but one can ask if a stronger property holds: {\em
given a nested sequence of renormalization domains $(D_n)$ and their induced maps $(f_n)$,
are the derivatives $\|Df_n|_{D_n}\|$  uniformly bounded?}

\paragraph{j -- New general tools.} Some of the new results obtained in the present paper hold for any mildly dissipative diffeomorphism of the disk. 
\begin{description}
\item[\it Closing lemma.]
One of them is a new version of the closing lemma proved in \cite{CP} which states that for mildly dissipative diffeomorphisms of the disk, the support of any measure is contained in the closure of periodic points. Our improvement (Theorem \ref{t.measure local}) localizes the periodic points: given  an invariant cellular connected compact set $\Lambda$, the support of any invariant probability on the set is contained in the closure of the periodic points in $\Lambda$. In that sense, Theorem \ref{t.measure local} is an extension of a well known result by Cartwright and Littlewood about the existence of fixed points for invariant cellular sets (see Proposition \ref{p.CL}).
\item[\it No cycle.]
Another one is a generalization of the result proved by Pixton~\cite{pixton} (improving a previous work by Robinson~\cite{robinson1}:
it states that for $C^\infty$-generic diffeomorphisms of the sphere,
a cyclic accumulation between stable and unstable branches of periodic points
can be perturbed to produce a homoclinic connection and positive entropy.
Theorem \ref{t.cycle} shows that the generic hypothesis is not needed for mildly dissipative diffeomorphisms of the disc: there is no finite sequence of fixed points such that the unstable manifold of each one accumulates on the next point and the unstable manifold of the last one accumulates on the first point (Theorems \ref{t.cycle} and  \ref{t.cycle2} in Section \ref{no cycle section}).
This is clear when the intersections between unstable and stable manifolds are transversal but when they just accumulate, it is more difficult.
The strategy consists in building special Jordan domains (that we call \emph{Pixton discs}) from the accumulation of unstable branches on stable manifolds.
\end{description}

\paragraph{k -- Summary of the proof.} 
In order to present the envisioned proof strategy, we first present a class of examples of infinitely renormalizable dissipative homeomorphisms of the disc (inspired by the examples in~\cite{GvST}) and we explain their main dynamical features. We use them as a prototype model for maps with zero entropy. The proofs below will show that these features (essentially) apply also for infinitely renormalizable mildly dissipative diffeomorphisms.

\paragraph{\it Prototype models.}
Let $f_0, f_1$ be two Morse-Smale dissipative diffeomorphisms of the disc. The limit set of $f_0$ is given by a fixed saddle whose unstable branches are interchanged and an attracting orbit of period two that revolves around the fixed point: the fixed point is then said to be {\it stabilized} and the attracting orbit is analogous to a period doubling sink for interval maps. The limit set of $f_1$  is given by a fixed attracting periodic point, a saddle of period three (also said to be {\it stabilized})  that revolves around the fixed point which anchors one of the unstable {\color{black} branches} of the saddle periodic points,  and an attracting periodic orbit (also of period three) that attracts the other unstable branch of the saddles. Both diffeomorphisms are depicted in Figure~\ref{MSpdf}.
\begin{figure}
\begin{center}
\includegraphics[width=15cm,angle=0]{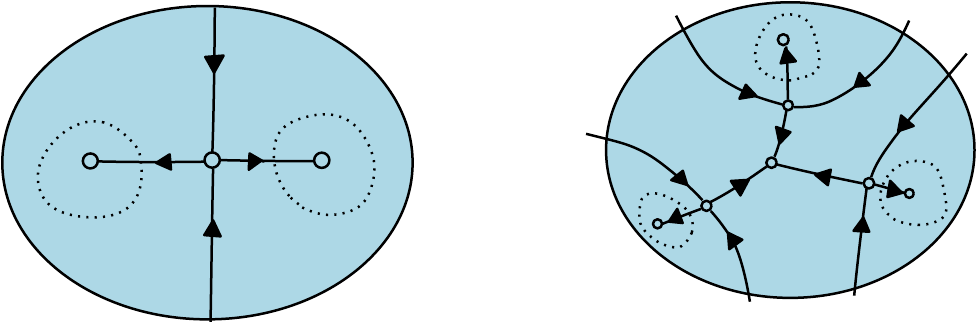}
\end{center}
\caption{The diffeomorphisms $f_0$ (left) and $f_1$ (right). The attracting domains are depicted with a dash boundary.\label{MSpdf}}
\end{figure}
Observe that $f_0$ has an attracting disc of period $2,$ whose iterates belong to two different regions bounded by the local stable manifold of the saddle; $f_1$ has an attracting disc of period three contained inside the disjoint regions bounded by the local stable manifolds of the saddle of period three (these regions, in both cases, are called {\it decorated regions}).

Given a sequence $(k_i)\in\{0,1\}^\NN$, one can build a sequence of dissipative diffeomorphisms  $g_i= f_{k_i}\sqcupplus f_{k_{i-1}}\sqcupplus\dots\sqcupplus f_{k_0}$
with a sink of period $\tau_i:=\Pi_{j=1}^i(2+k_j)$. The symbol $\sqcupplus$ means that the diffeomorphism $f_{k_j}$ is pasted in the basin of the sink of $ g_{j-1}$ (by writing $f_{k_j}$ as the composition of $\tau_{j-1}$ diffeomorphisms).
In that way, $g_i$ has a nested sequence of attracting discs $D_0\supset D_1\supset\dots\supset D_j$ of periods $\tau_0,\dots,\tau_i$.
Each diffeomorphism $g_i$ is Morse-Smale and the sequence $(g_i)$ converges to a homeomorphism whose limit set is made of
periodic points and of an odometer supported on a Cantor set (the intersection of the nested sequence of attracting domain).
We make some remarks: (i) The construction shows that there exist diffeomorphisms with vanishing entropy and with periodic points whose period is not $2^n$.
(ii) The sequence can converge to a mildly dissipative diffeomorphism if $k_i=0$ for $i$ large
(the convergence towards a diffeomorphism is more difficult, see \cite{GvST}).
(iii) The previous construction can be performed with more pasted diffeomorphisms:
the period of the saddle and the non-fixed sink may be larger; one can also consider more complicated Morse-Smale systems.

\paragraph{\it Pixton discs.}
The unstable branches connect the periodic points of $g_i$ and form a {\it chain} with a tree structure, see Figure \ref{treepdf}.
The tree branches land at points that are:
\begin{itemize}
\item[--] either attracting and may anchor unstable manifolds of points of larger period,
\item[--] or saddles whose unstable branches are exchanged at the period.
\end{itemize}
\begin{figure}[h]
\begin{center}
\includegraphics[width=12cm, angle=0]{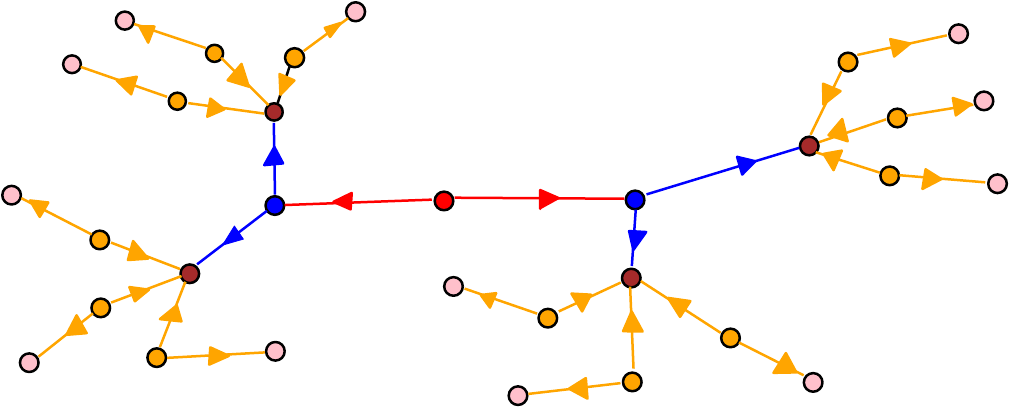}
\end{center}
\caption{Chain of periodic points associated to $f_1\sqcupplus f_0\sqcupplus f_0$: there is
one saddle fixed point, a saddle of period two (at the period its unstable branches are exchanged), a sink of period four, a saddle of period twelve, and a sink of period twelve.
The arrows indicate if the periodic points are saddles or sinks (on the one-dimensional structure a saddle appears as a sink).
\label{treepdf}}
\end{figure}
This observation will allow us to reconstruct the attracting discs, see Figure~\ref{Pixtpdf}.
In the first case (left of the figure), the unstable manifold of a fixed point $p$ accumulates on a fixed sink which anchors a stabilized revolving saddle with larger period: the unstable branch of $p$ has to cross the stable manifolds of the iterates of the saddle; this defines an attracting disc which contains all the periodic points attached to the sink. In the second case (right of the figure), the unstable manifold of the fixed point $p$ accumulates on a fixed saddle whose unstable branches are exchanged by the dynamics and accumulate on a sink of period $2$: the unstable branch of $p$ has to cross the stable manifold of the fixed saddle; this also defines an attracting disc which contains all the attached periodic points. We call the domains built in this way, {\em Pixton discs}.
\begin{figure}[h]
\begin{center}
\includegraphics[width=15cm,angle=0]{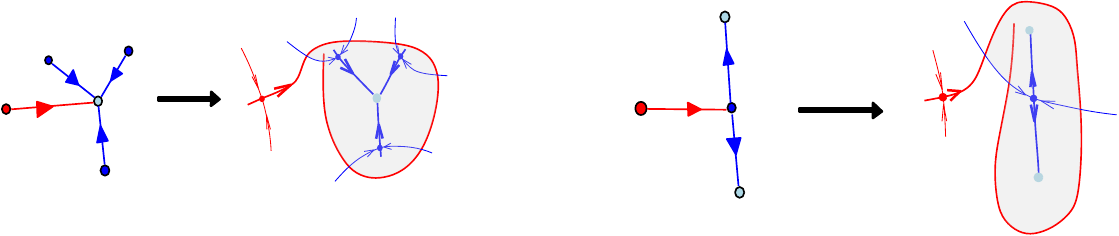}
\put(-63,44){\small $p$}
\put(-323,44){\small $p$}
\end{center}
\caption{Attracting discs obtained from an unstable branch and stable manifolds.  \label{Pixtpdf}}
\end{figure}
\medskip
 
\paragraph{\it When all the periodic points are fixed.}
We now explain how to handle a general mildly dissipative diffeomorphism with zero entropy.
In order to prove Theorem \ref{t.theoremA}, one first has to show that if all the periodic points are fixed,
then the limit set of the dynamics consists of only fixed points. The ``no-cycle property" is crucial.
Another ingredient is to prove that the $\omega-$limit set of any orbit contains a fixed point: this follows from our closing lemma (Theorem \ref{t.measure local}). With these tools, one builds a filtration associated to the fixed points and conclude that the limit set of the dynamics is reduced to the set of fixed points.

\paragraph{\it Periodic structure.}
When there are periodic points which are not fixed, we prove that the unstable branches induce a structure as in the previous examples: they form {\it chains} (see Definition \ref{d.chain})  that branch at points of low period to which are attached saddles of larger or equal period.
A special role is played by \emph{stabilized points}: these are saddles that either are fixed and whose unstable branches are exchanged,
or are not fixed but whose unstable manifold is anchored by a fixed point (see Definition \ref{d.stabilization} and Propositions \ref{p.chain} and \ref{p.stab-decorate}).
The local stable manifolds of the stabilized points bound domains called {\it decorated regions} (see Definition \ref{d.decorated region}) which are pair-wise disjoints: indeed if two such regions intersect, the unstable manifold of a stabilized point
has to cross the stable manifold of another iterate in order to accumulate on the anchoring fixed point, contradicting the fact that the entropy
vanishes. The decorating regions contain all the periodic point of larger period (see Definition \ref{d.decreasing-chain}, Proposition \ref{p.chain-decreasing} and \ref{p.decreasing-chain}).

\paragraph{\it Construction of trapping discs.}
For each unstable branch $\Gamma$, fixed by an iterate $f^n$, we build a disc that is trapped by $f^n$ and contains all of the accumulation set of the branch $\Gamma$ (Theorems \ref{t.renormalize} and  \ref{t.renormalize2}). To each saddle accumulated by $\Gamma$ one associates a {\it Pixton disc} which is a candidate to be trapped. These discs are bounded by arcs in $\Gamma$ and stable manifolds of saddles in the accumulation set,
as in the previous examples (see Lemma \ref{l.highperiod}).
A finite number of these Pixton discs is enough to cover the accumulation set, implying the trapping property.
The closing lemma mentioned above (Theorem \ref{t.measure local}) is a key point for proving the finiteness.

\paragraph{\it Finiteness of the renormalization domains.}
A stronger version of the renormalization (Theorem \ref{t.renormalize-prime}) implies corollaries \ref{c.infinitely-renormalizable} and \ref{c.structure}. It asserts that the number of renormalization domains required to cover the dynamics is finite.
Since the renormalization discs are related to decorated regions, we have to show that the periods of the stabilized saddles is bounded (see Theorem \ref{t.finite}).

\paragraph{\it Bound on the renormalization period.}
In order to  show that after several renormalization steps, the renormalization periods eventually equal two (Theorem \ref{t.period}), we develop a {\em rigidity argument}: the limit attractors (the generalized odometers obtained as intersection of nested renormalizable domains of an infinitely renormalizable diffeomorphism) induce a stable lamination whose leaves vary continuously in the $C^1$-topology over sets with measure  arbitrary close to one. This property follows from a \emph{$\gamma-$dissipation property} (see Section \ref{ss.gamma-dissipation}).
In particular, for a large proportion of points, the leaves of the lamination by local stable manifolds are ``parallel". Since the renormalization domains (inside a renormalization disc obtained previously) are contained in a (relative) decorated region, and since the measure is equidistributed between the different renormalization components, a relative renormalization period larger than two would contradict that a large proportion of local stable manifolds are parallel. A simple heuristic of the argument is the following: at small scale, the quotient by the local stable manifolds provides an interval that contains a large proportion of the points of the odometer, which  is enough to recover the period doubling mantra that permeates the renormalization scheme for zero entropy maps of the interval.

\paragraph{l -- Other attracting domains.} One can wonder about the transition to chaos for dissipative diffeomormorphisms on other attracting domains, such as  the annulus.  The transition to chaos is already much more complicated on the circle than on the interval (see \emph{e.g.,} \cite{FrTr} and references therein), as a result in particular of the non-triviality of the circle at the homotopy level. A prototype family that plays the role of the H\'enon maps for the circle, is an annulus  version of the Arnold family. Results related to the transition to chaos in that context can be found in \cite{CKKP} and \cite{GY}.

\paragraph{m -- Organization of the paper.} The  next three sections present preliminary results: Section~\ref{preliminaries} describes how fixed points may be rearranged
inside finitely many fixed curves, recalls the Lefschetz formula and a fixed point criterion due to Cartwright and Littlewood; in Section \ref{s.quantitative} we revisit the notion of $\gamma-$dissipation introduced in \cite{CP} and we present a few results that allow to improve the lower bound on $\gamma$; in Section \ref{ss.closing} we state a new closing lemma. 

Section~\ref{no cycle section} proves that (under the hypothesis of zero entropy and mild dissipation) there is no cycle between periodic points. This is essential to show in Section \ref{decoration section} that periodic points are organized in chains; also in that section we introduce the notions of decoration and stabilization that provides a hierarchical organization of the chains.
Section~\ref{s.gms} discusses the notion of generalized Morse-Smale diffeomorphisms.

In Section \ref{ss.trapping} we prove that the accumulation set of an unstable branch of a fixed point is contained in an arbitrarily small attracting domain and in Section \ref{ss.local renormalization} we conclude the proof of the local renormalization (Theorem \ref{t.theoremA}). A global version of that theorem (Theorem \ref{t.renormalize-prime}) is obtained in Section \ref{s.renormalize}; this requires to first show that the periods of the stabilized points are bounded (this is proved in Section \ref{ss.finitness}).

The proof of Theorem \ref{t.period} is provided in Section \ref{s.period}: it uses the description of the chain-recurrent set (Corollary \ref{c.structure}) which is proved in Section \ref{s.odometers}.

In the last two sections, we prove the results about dynamics close to interval maps and about the H\'enon maps
(Corollary~\ref{c.henon} and Theorem~\ref{t.small jacobian}).

{\color{black} We have included an index at the end of the paper in order to help to navigate between the definitions.}

\paragraph{Acknowledgements.}
We are indebted to Mikhail Lyubich for the discussions we exchanged during the preparation of this work, to Damien Thomine for his remarks on a first version of the paper, {\color{black}  and to the referee for his comments.
We also thank Bryce Gollobit and Axel Kodat for their careful reading which helped to improve the text.}

\section{Periodic orbits}
\label{preliminaries}

In  Section  \ref{ss.periodic}, we analyze the different types of periodic points that could exist
for a dissipative diffeomorphism of the disc. When there exist infinitely many periodic points of a given period, we rearrange them inside finitely many periodic arcs. In Section  \ref{ss.index arc} we recall the Lefschetz formula. In Section  \ref{ss.accumulation of unstable} we present a kind of topological $\lambda-$lemma that is useful to describe the accumulation set of unstable manifolds of periodic points. In Section \ref{ss.fixedcriterion} we recall a classical result by Cartwright and Littlewood about the existence of fixed points.

\subsection{Dynamics near a periodic point}\label{ss.periodic}
We describe the dynamics in the neighborhood of a periodic orbit.
Note that up to {\color{black} replacing} $f$ by an iterate, {\color{black} it is enough to consider the dynamics in a neighborhood of a fixed point}.
When $p$ is fixed, the eigenvalues $\la^-_p, \la^+_p$ of $D_pf$ verify
$|\la^-_p|\leq |\la^+_p|$ and $|\la^-_p\la_p^+|<1$.

\paragraph{\it Hyperbolic sink.} When $|\la^+_p|<1$, the point $p$ is a hyperbolic sink.
This covers in particular all the cases where $|\la^-_p|= |\la^+_p|$.
We now describe the other cases.

\paragraph{\it Stable curve, \color{black} stable branch.}
When $|\lambda^-_p|<|\lambda^+_p|$, there exists a well defined (strong) stable manifold which is a $C^1$ curve.
The connected component containing $p$ is denoted by $W^s_\mathbb{D}(p)$.
For orbits with higher period $\cO$, we denote by $W^s_\mathbb{D}(\cO)$ the union of the curves $W^s_\mathbb{D}(p)$, $p\in \cO$.
{\color{black} A connected component $\mathcal{L}$ of $W^s_\mathbb{D}(p)\setminus \{p\}$ is called a \emph{stable branch}.}
\index{branch (stable or unstable)}

\paragraph{\it Local stable set.} The local stable set of $p$, \emph{i.e.,} the set of points whose forward orbit converges to $p$
and remains in a small neighborhood of $p$, is either a neighborhood of $p$ (a sink),
a subset of $W^s_\mathbb{d}(p)$, or a half neighborhood of $p$ bounded by $W^s_\mathbb{D}(p)$.

\paragraph{\it Center manifold.}
When $|\lambda^-_p|<|\lambda^+_p|$,
the center manifold theorem asserts that there exists a $C^1$ curve $\gamma$ which contains $p$, is tangent to the eigendirection of $D_pf$
associated to $\lambda^+_p$ and is locally invariant:
there exists $\varepsilon>0$ such that $f(\gamma\cap B(p,\varepsilon))\subset \gamma$.
The two components of $\gamma\setminus \{p\}$ are either preserved or exchanged (depending if the eigenvalue $\lambda^+_p$ is positive or negative).
Along each component $\Gamma$ of $\gamma\setminus \{p\}$, the dynamics (under $f$ or $f^2$) is either attracting, repelling, or neutral
(in which case $p$ is accumulated by periodic points inside $\Gamma$).
{\color{black} The center manifold $\gamma$ is a priori not unique, but the type of dynamics in each component of
$\gamma\setminus \{p\}$ does not depend on the curve $\gamma$ that has been taken.}

\paragraph{\it Unstable branches.}
The \emph{unstable set} $W^u(p)$ of $p$ is the set of points $x$ such that
the distance $d(f^{-n}(x),f^{-n}(p))$ decreases to $0$ as $n\to +\infty$.
When it is not reduced to $p$, it is a $C^1$ curve which contains $p$. The local unstable set is defined as  the set of points whose backward orbit converges to $p$
and remains in a small neighborhood of $p$ and observe that they are  contained in the center manifold $\gamma$.
Each connected component $\Gamma$ of $W^u(p)\setminus \{p\}$ is called an unstable branch of $p$. 
\index{branch (stable or unstable)}

\paragraph{\it Hyperbolic saddle.} When $|\la^-_p|<1< |\la^+_p|$, the point $p$ is a hyperbolic saddle.
It admits two unstable branches.

\paragraph{\it Indifferent fixed point.} When $|\la^-_p|<1= |\la^+_p|$, the point $p$ is indifferent.
We then consider the dynamics (under $f$ or $f^2$) on each side of a center manifold.
When $p$ is isolated among points of period $1$ and $2$, it is either a sink (both components are attracting),
a saddle (both components are repelling) or a saddle-node (the components are fixed, one is attracting, one is repelling):
the type does not depend on the choice of the center curve $\gamma$.

\paragraph{\it Saddle with {\color{black} reflection}.} When the unstable branches of a fixed saddle $p$ are exchanged by the dynamics,
we say that $p$ is a \emph{(fixed) saddle with {\color{black} reflection}}.\index{saddle with (or with no) reflection}
{\color{black} Some authors also call them \emph{flip saddles}.}

\paragraph{\it Index.}\index{index, Lefschetz formula} For an isolated fixed point, one can define the index of that fixed point as the winding number of the vector field $f(x)-x$ around the fixed point. For dissipative diffeomorphisms the index of an isolated fixed point is:
\begin{itemize}
\item $1$ for a sink or a saddle with {\color{black} reflection},
\item $0$ for a saddle-node,
\item $-1$ for a saddle with no {\color{black} reflection}.
\end{itemize}
By the classical Lefschetz formula, when the number of fixed points is finite
the sum of the index of the fixed points in the disc is equal to $1.$

\begin{remark}\label{r.degenerated}
A saddle-node can be considered as the degenerated case of a sink and a saddle of index $-1$ that have collided.
Similarly, a fixed point with an eigenvalue less or equal to $-1$ can be considered as the collision of a fixed sink with the points
of a $2$-periodic orbit with positive eigenvalues. In particular, fixed saddles of index $1$ may be considered as the union
of a fixed sink with a saddle of period $2$.
\end{remark}

\subsection{Normally hyperbolic periodic arcs}\label{ss.arc}
When the number of fixed points is infinite, they appear inside normally hyperbolic arcs.
\begin{definition}\label{d.fixed-arc}
A \emph{fixed arc}\index{arc, fixed arc, half arc} is a compact $f$-invariant $C^1$ curve $I$ whose endpoints are fixed
and which admits an invariant splitting $T_x\mathbb{D}|_{x\in I}=E^s\oplus F$ satisfying:
\begin{itemize}
\item[--] $T_xI\subset F_x$ for each $x\in I$,
\item[--] there is $k\geq 1$ such that ${|Df^k_{E^s_x}|} < |Df^k_{F_x}|$ and $|Df^k_{E^s_x}|<1$ for each $x\in I$,

\end{itemize} 
It is \emph{isolated} if all the fixed points in a neighborhood are contained in $I$.

 \end{definition}

A fixed point is a fixed arc: for a hyperbolic sink, the splitting is trivial $F=\{0\}$. When $I$ has two distinct endpoints $p_1,p_2$,
the forward orbit of any point in the strip $W^s_{\mathbb D}(I)$ bounded by $W^s_{\mathbb D}(p_1)$ and $W^s_{\mathbb D}(p_2)$
converges to a fixed point in $I$. 
When $I$ is not reduced to a sink, $\mathbb{D}\setminus W^s_{\mathbb D}(I)$ has two connected components.
\medskip

{\color{black} An \emph{$f$-invariant unstable branch} of the fixed arc $I$ is an unstable branch $\Gamma$ of an endpoint of $I$,
that is fixed by $f$. Note that} the unstable set of $I$ is contained in the {\color{black} union of the} unstable branches of $I$.
\medskip

\begin{definition}\label{d.type-arc}
Four cases may occur for an isolated fixed arc $I$. It has the \emph{type of}:
\begin{itemize}
\item \emph{a sink}, if the orbit of any point in a neighborhood converges to a fixed point in $I$,
\item \emph{a saddle with {\color{black} reflection}}, if $I$ is a single fixed point $p$ with an eigenvalue $\lambda^+_p\leq -1$,\index{saddle with (or with no) reflection}
\item \emph{a saddle-node}, if the arc has one $f$-invariant unstable {\color{black} branch},
\item \emph{a saddle with no {\color{black} reflection}}, if the arc has two $f$-invariant unstable branches.\index{saddle with (or with no) reflection}
\end{itemize}
See Figure~\ref{f.type}.
\end{definition}

\begin{figure}
\begin{center}
\includegraphics[width=11cm,angle=0]{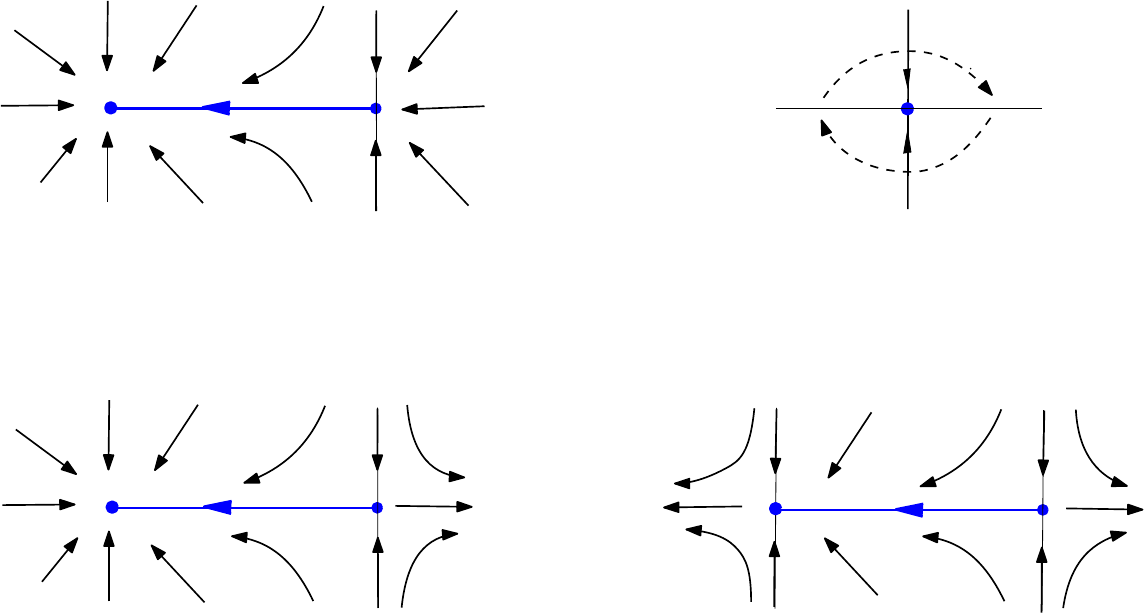}
\put(-260,90){sink}
\put(-270,-15){saddle-node}
\put(-110,90){saddle with reflection}
\put(-120,-15){saddle with no reflection}
\end{center}
\caption{\color{black} The four types of an isolated fixed arc.\label{f.type}}
\end{figure}

\begin{remark}
Note that if an isolated fixed arc $I$ contains a fixed point $p$ with an eigenvalue less or equal to $-1$,
then $I=\{p\}$ (since the endpoints of $I$ are fixed points).
This is the only case where there may exists periodic orbits in arbitrarily small neighborhoods of $I$.
The arc $I$ is isolated since $p$ may be accumulated only by points with period $2$.
\end{remark}
\medskip

\begin{proposition}\label{p.group}
If $f$ is a dissipative diffeomorphism of the disc, there exists a finite collection $\cI$
of disjoint isolated fixed arcs whose union contains all the fixed points of $f$.
\end{proposition}

\begin{proof}
By the implicit function theorem,
the set of fixed points of $f$ is the union of a finite set of isolated fixed points
and of a compact set $K$ of fixed points $p$ having one eigenvalue $\lambda^+_p\geq 1/2$.
Each isolated fixed point is an isolated fixed arc and it remains to cover $K$
by finitely many pairwise disjoint isolated fixed arcs.

To each fixed point $p$ having an eigenvalue $\lambda^+_p\geq 1/2$,
the center manifold theorem (see~\cite{BoCr})
associates a $C^1$ curve $\gamma$ which contains $p$,
is tangent to the eigenspace $F_p$ associated to the eigenvalue $\lambda^+_p$,
and is locally invariant: $f(\gamma)\cap \gamma$ contains a neighborhood of $p$ in $\gamma$;
moreover, any periodic point  in a neighborhood of $p$ is contained in $\gamma$.
One can thus build an arc $I\subset \gamma$ bounded by two fixed points,
which is invariant by $f$, normally contracted and which contains all the fixed points in a neighborhood of $p$:
it is a fixed arc, as in Definition~\ref{d.fixed-arc}.

By compactness, there exists a finite family of such fixed arcs.
Let us choose $\varepsilon>0$ small.
By decomposing the arcs, one can assume that each such arc $I$
has diameter smaller than $\varepsilon$,
is contained in a $C^1$ curve $J$ such that $J\setminus I$ has two connected components,
both of diameter larger than $2\varepsilon$ and such that any fixed point in the $2\varepsilon$-neighborhood of $I$
is contained in $J$.

If there exists two arcs $I,I'$ which intersect,
one {\color{black} considers} the larger curve associated to $J$.
We note that all the fixed points of $I'$ are contained in $J$.
One can thus reduce $I'$ as an arc $\widetilde I'$ such that
all the fixed points of $K\cap I'\cup I$ are contained in $I\cup \widetilde I'$
and $I\cup \widetilde I'$ is a $C^1$ curve.
One repeats this argument for all pairs of fixed intervals.
This ensures that the union of all the fixed intervals $I$
contains $K$ and is a union of disjoint $C^1$ curves.
By construction, each of these curves is an isolated fixed arc.
\end{proof}

The choice of the collection $\cI$ is in general not unique.
Note that for any distinct $I,I'\in \cI$ which are not sinks, the strips $W^s_{\mathbb D}(I)$, $W^s_{\mathbb D}(I')$
are disjoint.

\paragraph {\it Partial order:} The finite collection of fixed arcs $\cI$ can be partially ordered in such a way that at least one {\color{black} of} the unstable branches of the extremal points of $I_j$ {\color{black} accumulates} on $I_{j+1}.$

\subsection{Lefschetz formula and arcs}
\label{ss.index arc}
In the present section we recall the definition of index for isolated invariant arcs and ``half'' arcs.\index{arc, fixed arc, half arc}
\paragraph{\it Index of an arc.}\index{index, Lefschetz formula}  To any simple closed curve $\sigma\subset \mathbb{D}\setminus \operatorname{Fix}(f)$,
one associates an index $i(\sigma,f)$, which is the winding number of the 
vector field $f(x)-x$ along the curve.
This defines for any isolated fixed arc $I$ an index $index(I,f)$: this is the index $i(\sigma,f)$
associated to any simple closed curve contained in a small neighborhood of $I$ and surrounding $I$.
For arcs as described in Definition~\ref{d.type-arc}, the index takes a value in $\{-1,0,1\}$, equal to $1$ for a sink and a saddle with {\color{black} reflection},  $0$ for a saddle-node, 
and $-1$ for a saddle with no {\color{black} reflection}.
In particular, the index is $1$ exactly when the arc has no $f$-invariant unstable branch.
When $I$ is reduced to an isolated fixed point, $index(I,f)$ coincides with the usual index.

\paragraph{\it Index of a half arc.}
Let us consider a fixed arc $I$ which contains a fixed point $p$ having an eigenvalue $\lambda^+_p\geq 1$
and a connected component $V$ of $\mathbb{D}\setminus W^s_\mathbb{D}(p)$.
Let us assume that $I$ is \emph{isolated in $V$}, \emph{i.e.,} that any fixed point in a neighborhood of $I$ that belongs to $V$,
also belongs to $I$. Then, one can associate an index $index(I,V,f)$:
it has the value $-1/2$ if $I$ has a (local) unstable branch in $V$ that is fixed by $f$ and the value $1/2$
otherwise (in which case $I$ is semi-attracting in $V$). {\color{black} See Figure~\ref{f.half-index}.}
When $I$ is isolated {\color{black} in $\mathbb{D}$}, the index of $I$ is equal to the sum of the two indices associated to the two connected components of $\mathbb{D}\setminus W^s_\mathbb{D}(p)$.

\begin{figure}
\begin{center}
\includegraphics[width=9cm,angle=0]{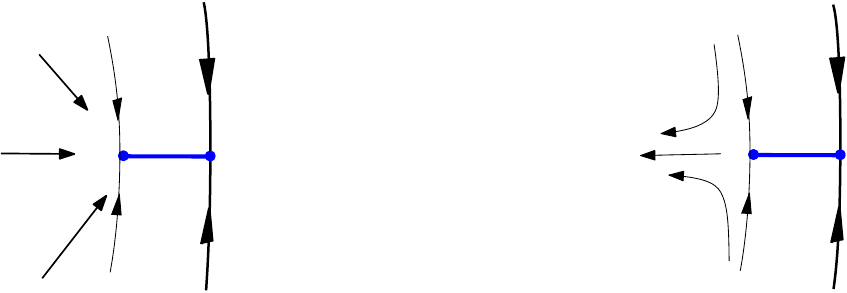}
\put(5,90){$p$}
\put(-185,90){$p$}
\end{center}
\caption{\color{black} Half arcs with index $1/2$ and $-1/2$.\label{f.half-index}}
\end{figure}

The next proposition restates the Lefschetz formula\index{index, Lefschetz formula} in our setting.
\begin{proposition}\label{p.lefschetz}
Let $f$ be a dissipative diffeomorphism of the disc
and $\cI$ a set of isolated fixed arcs as in Proposition~\ref{p.group}.
Then the sum of the indices $index(I,f)$ of the arcs $I\in \cI$ is $1$.
\end{proposition}
\begin{proof}
One can {\color{black} perform a local modification of} the dynamics {\color{black} near the set of indifferent fixed points} and
obtain in this way a diffeomorphism $g$ satisfying:
\begin{itemize}
\item[--] each $I\in \mathcal{I}$ is still a fixed arc for $g$,
\item[--] the fixed points of $g$ are all hyperbolic and contained in the union of the arcs $I$.
\end{itemize}
In particular $g$ has only finitely many fixed points.
In an arc $I$, two saddles are separated by a sink.
This proves that the index of any arc $I\in \mathcal{I}$ for $f$ coincides with the
sum of the indices of the fixed points of $g$ that are contained in $I$.

Consequently, the sum of the indices of the arcs $I\in \cI$ for $f$
is equal to the sum of the indices of the fixed points of $g$.
From~\cite[Proposition VII.6.6]{dold}, this sums equal $1$ since $g$ is a map homotopic to a constant.
\end{proof}

\subsection{The accumulation set of unstable branches}
\label{ss.accumulation of unstable}

Let $\Gamma$ be a $f$-invariant unstable branch of a fixed point $p$
and $\gamma\subset \Gamma$ be a curve which is a fundamental domain
{\color{black} (\emph{i.e.,} which meets any orbit contained in $\Gamma$).}
The \emph{accumulation set} of $\Gamma$ is
the limit set of the iterates of $f^n(\gamma)$ as $n\to +\infty$.
We say that $\Gamma$ accumulates on a set $X$ if $X$ intersects the accumulation set of $\Gamma$.
These definitions naturally extend to unstable branches of periodic points. The next proposition, is  a kind of topological version of the  classical $\la-$lemma for mildly dissipative diffeomorphisms without assuming homoclinic intersections.

\begin{proposition} \label{p.transitive}
Let $p$, $q$ be two fixed points of a mildly dissipative diffeomorphisms and let
$\Gamma_p$, $\Gamma_q$ be two $f$-invariant unstable branches such that
$\Gamma_p$ accumulates on a point of $\Gamma_q$.
Then {\color{black} $\Gamma_q$ is included in} the accumulation set of $\Gamma_p$.
\end{proposition}
\begin{proof} Let $U$ be a small simply connected neighborhood of $q$.
There are points $y_k\in\Gamma_p$ arbitrarily close to a point $y\in \Gamma_q\cap U$
having iterates $f^{-1}(y_k)$, $f^{-2}(y_k)$,\dots , $f^{-m_k}(y_k)$ in $U$ such that
$f^{-m_k}(y_k)$ converge to a point $x\in U\cap W^s_\mathbb{D}(q)$.
Let $\gamma^s\subset U\cap W^s_\mathbb{D}(q)\setminus \{q\}$ be a compact curve
containing $x$ and such that both connected components of $\gamma^s\setminus\{x\}$ properly contain a fundamental domain and its iterate.
Let also $\gamma^u\subset \Gamma^u(q)$ be a compact curve which properly contains a fundamental domain and its iterate.
Now we take two curves $l^u_1, l^u_2$ transversal to $W^s_{\mathbb{D}}(p)$ through the extremal points of $\gamma^s$ and two curves $l^s_1, l^s_2$ transversal to $\Gamma_q$ through the extremal points of $\gamma^u$. We construct the rectangles $R_n$ bounded by $l^u_1, l^u_2$ and the connected components of $f^{-n}(l^s_1), f^{-n}(l^s_2)$ inside $U\cap f^{-1}(U)\cap\dots\cap f^{-n}(U)$ that intersect $l^u_1$ and  $l^u_2.$ Observe that those rectangles converge to $\gamma^s.$ 

Let $y_n\in W^u(p)\cap R_n$ converging to $x.$ Let $l_n$ be a connected arc inside $W^u(p)$ that joins $y_n$ and $y_{n+2}$. It follows that either 
\begin{enumerate}
\item there is a connected subsegment $l_n'$ inside  $l_n\cap (R_{n}\cup R_{n+1}\cup R_{n+2}\cup\dots  )$ that contains $y_{n+2}$ and intersects either $l^u_1$ or $l^u_2,$ or
\item there is a subsegment $l_n'$ inside  $l_n$ that crosses $R_{n+1}$ and is disjoint from $l^u_1\cup l^u_2.$
\end{enumerate}
In the first case, the accumulation set of $\Gamma_p$ contains a fundamental domain of a stable branch of $q$: this is a contradiction since each stable branch
of $q$ contains a point in $\mathbb{D}\setminus f(\mathbb{D})$.
In the second case,  $f^n(l_n')$ converge to $\gamma^u$ in the Hausdorff topology and from there, the proposition is concluded.
\end{proof}

\subsection{Decoration}\label{ss.decoration}

The geometry described in the next definition is essential in this work.

\begin{definition}\label{d.decoration}
 Let $f$ be a mildly dissipative diffeomorphism of the disc. A periodic orbit  $\cO$ which is not a sink is \emph{decorated}\index{decoration} if for each $p\in \cO$, one connected component of $\mathbb{D}\setminus W^s_\mathbb{D}(p)$
does not intersects $\cO$ (see Figure~\ref{f.decoration}).
 
\end{definition}

\begin{figure}
\begin{center}
\includegraphics[width=5cm,angle=0]{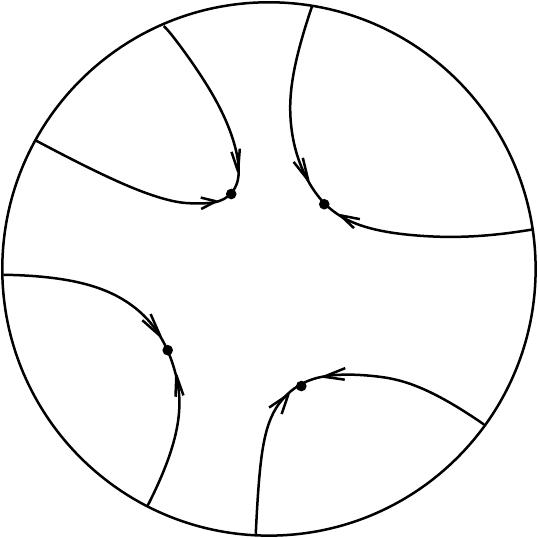}
\put(-62,30){$f^2(p)$}
\put(-53,92){$f(p)$}
\put(-90,95){$p$}
\put(-124,43){$f^4(p)$}
\end{center}
\caption{A decorated periodic orbit.\label{f.decoration}}
\end{figure}

\subsection{A fixed point criterion}
\label{ss.fixedcriterion}

The following result refines Brouwer fixed point theorem inside the disc.

\begin{proposition}[Cartwright-Littlewood~\cite{CL}]\label{p.CL}
Let $f$ be an orientation-preserving homeomorphism of the plane $\mathbb{R}^2$
and let $C$ be an invariant compact set whose complement $\mathbb{R}^2\setminus C$ is connected.
Then $f$ has a fixed point in $C$.
\end{proposition}

\section{Quantitative dissipation}
\label{s.quantitative}

We recall  a quantitative version of the dissipation that was introduced in \cite{CP}:
\begin{definition}
Let $f$ be a dissipative diffeomorphism of the disc and
$K$ be a $f$-invariant compact set.
For $\gamma\in (0,1)$, we say that
the diffeomorphism $f$ is \emph{$\gamma$-dissipative on $K$}\index{dissipation, mild and $\gamma$-dissipation}
if there is $n\geq 1$ such that for any $x,y\in K$ and any unit vector $u\in T_x\mathbb{D}$,
$$|\det Df^n(y)|< \|Df^n(x).u\|^\gamma.$$
\end{definition}
With this definition in mind, it is possible to get a uniform geometry of stable manifolds for all ergodic measures as  presented in Section ~\ref{ss.uniform stable}. In Section ~\ref{ss.gamma-dissipation} it is shown that $\gamma-$dissipation holds for uniquely ergodic aperiodic compact invariant sets.

\subsection{Criterion for $\gamma$-dissipation}
\label{ss.gamma-dissipation}

The next proposition provides sufficient conditions for $\gamma-$dissipation. Observe that the hypothesis are satisfied by odometers.

\begin{proposition}\label{p.gamma-strong}
Let $f$ be a dissipative $C^r$ diffeomorphism, $r>1$. Let $K$ be an invariant compact set which does not contain
any periodic point, is uniquely ergodic, and does not intersect any transitive compact set with positive entropy.
Then $f$ is $\gamma$-dissipative on $K$ for all $\gamma\in (0,1)$.
\end{proposition}
\begin{proof} 

We first claim that if a $C^r$ diffeomorphism $f$ of a surface (with $r>1$)  admits a hyperbolic ergodic measure $\mu$ with no atom,
then $\supp(\mu)$ is contained in a transitive set with positive topological entropy.
A theorem of Katok~\cite{Ka} asserts the existence of periodic points $p_n$ whose orbits equidistribute towards $\mu$.
Note that the points are distinct and their period goes to $+\infty$ since $\mu$ has no atom.
Moreover, from the proof one can check that all the points $p_n$ belong to a compact set
of points having uniform local stable and local unstable manifolds
which vary continuously with the points
(this is not necessarily true for all the iterates of the points $p_n$).
One may thus find two distinct points $p_n$ and $p_m$ close, so that their stable and unstable manifolds
have transverse intersections. This implies that $f$ has a horseshoe $\Lambda$. Taking $n$ larger, one gets an increasing sequence
of horseshoes whose union accumulates on $\supp(\mu)$.
\medskip

We now turn to the proof itself. Since $f$ is dissipative, there exists $b\in (0,1)$ such that
$|\det Df(x)|<b$ for any $x\in K$. Let us fix $\gamma\in (0,1)$ and any $\varepsilon\in (0,-\frac 1 4(1-\gamma).\log b)$.

Let $\mu$ be the ergodic probability on $K$.
Note that its upper Lyapunov exponent is non positive:
otherwise, $\mu$ would be hyperbolic with no atom and the claim above would imply
that $K$ intersects a transitive set with positive entropy, contradicting the assumptions of the proposition.
Consequently, there exists $\ell\geq 1$ such that
$\frac 1 \ell \int \log \|Df^\ell\| d\mu<\varepsilon/4$.
For $n\geq 1$ large enough and any $x\in K$,
the distribution of the iterates $x,\dots, f^n(x)$ is close to $\mu$, implying that for any $x,y\in K$,
$$\frac 1 n \log \|Df^n(x)\|\leq \varepsilon/4+\frac 1 n \sum_{j=0}^{n-1} \frac 1 \ell\log\| Df^\ell(f^j(x))\|
\leq \frac 1 2 \varepsilon + \frac 1 \ell \int \log \|Df^\ell\| d\mu \leq\frac 3 4 \varepsilon,$$
$$\text{and }\quad e^{-\frac{n\varepsilon}{4}}|\det Df^n(y)|\leq |\det Df^n(x)|.$$
For any $x,y\in K$ and any unit vector $u\in T_x\mathbb{D}$, we thus have:
$$e^{-\frac{n\varepsilon}{4}}|\det Df^n(y)|\leq |\det Df^n(x)|\leq \|Df^n(x)\| . \|Df^n(x).u\|\leq e^{n.\frac{3\varepsilon}4}\|Df^n(x).u\|.$$
If $u_0$ is the unit vector in $T_x\mathbb{D}$
which is the most contracted under $Df^n(x)$, we also have
$$\|Df^n(x).u_0\|^2\leq |\det Df^n(x)|\leq b^n.$$
Hence
$$|\det Df^n(y)|\leq e^{n.\varepsilon}\|Df^n(x).u_0\|\leq e^{n.\varepsilon}b^{n.(1-\gamma)/2}.\|Df^n(x).u_0\|^\gamma<
 \|Df^n(x).u_0\|^\gamma.$$
This gives as required
$$|\det Df^n(y)|< \|Df^n(x).u_0\|^\gamma\leq \|Df^n(x).u\|^\gamma.$$
So $f^n$ is $\gamma$-dissipative on $K$,
provided $n$ is chosen large enough.
\end{proof}

\subsection{Uniform geometry of the strong stable leaves}
\label{ss.uniform stable}

The next theorem was essentially proved in~\cite{CP}. However, it has to be recasted to get a precise quantitative estimate.

\begin{theorem}\label{t.stable}
For any $\alpha\in(0,1]$ and $\varepsilon>0$, there exists $\gamma\in (0,1)$ with the following property.
If $f$ is a $C^{1+\alpha}$ diffeomorphism which is $\gamma$-dissipative on an invariant
compact set $K$ which does not contain any sink, then there exists a compact set $A\subset K$ such that:
\begin{itemize}
\item[--] For any ergodic measure $\mu$ supported on $K$, we have $\mu(A)>1-\varepsilon$.
\item[--] Each point $x\in A$ has a stable manifold $W^s_\mathbb{D}(x)$ which varies continuously with $x$
in the $C^1$ topology.
\end{itemize}
\end{theorem}
 
For the proof, we refer to  \cite{CP} and the following slight changes in the results therein  that we need here.
Let $\tilde \sigma, \sigma, \tilde \rho, \rho \in (0,1)$ satisfying
\begin{equation}\label{e.pesin}
{\textstyle \frac{\tilde \rho\tilde \sigma}{\rho \sigma}} >\sigma^\alpha,
\end{equation}
and $A_{\tilde \sigma\sigma\tilde \rho\rho }(f)$ be the set of points $x$ having a direction $E\subset T_x{\color{black} \DD}$
such that for each $n\geq 0$
\begin{equation}\label{e.stable}
 \tilde \sigma^n\leq \|Df^n(x)_{|E}\|\leq  \sigma^{n},\;
\text{ and }\;  \tilde \rho^n\leq \frac{\|Df^n(x)_{|E}\|^2}{|\det Df^n(x)|}\leq \rho^{n}.
\end{equation}
We recall Theorem 5 and Remark 2.1 in \cite{CP}:

\begin{theorem}[Stable manifold at non-uniformly hyperbolic points]\label{t.stablecp}
Consider a $C^{1+\alpha}$ diffeomorphism $f$ with $\alpha\in (0,1]$.
Provided~\eqref{e.pesin} holds, the points in  $A_{\tilde \sigma\sigma\tilde \rho\rho }(f)$ 
have a one-dimensional stable manifold
which varies continuously for the $C^1$ topology with $x\in {\color{black} \DD}$.
\end{theorem}

To conclude Theorem \ref{t.stable}, observe that it is enough to prove following proposition:

\begin{proposition}\label{p.uniform-gamma} Given $\varepsilon>0$ and $\alpha\in (0,1]$, there is $\gamma>0$ with the following property.

Let us consider a diffeomorphism $f$ and an invariant compact set $K$ which does not contain any sink and where
$f$ is $\gamma-$dissipative.
Then there exist $\tilde \sigma, \sigma, \tilde \rho, \rho \in (0,1)$ satisfying~\eqref{e.pesin} such that
for any ergodic measure $\mu$ supported on $K$,
$\mu(A_{\tilde \sigma\sigma\tilde \rho\rho }(f))>1-\varepsilon.$
\end{proposition}   
\begin{proof} 
This is proved in \cite[Proposition 3.2]{CP} in the case $\varepsilon= 5/6$, $r=2$ and $\gamma=9/10$.
We explain how to adapt the proof by modifying the constants.
Let us take
$$D = \sup_{x\in K} |det Df(x)|,\,\,\, m = \inf_{x\in K} \|Df^{-1}(x)\|^{-1},$$
$$\tilde \sigma = m,\,\,\,  \tilde \rho = m^2/D,\,\,\,  \sigma  = D^{1-\alpha/3},\,\,\,  \rho = D^{1-\alpha/3}.$$
Since $f$ is $\gamma-$dissipative, $D< m^\gamma$ and the condition~\eqref{e.pesin} is satisfied provided $(1+\alpha/9)\gamma>1$.
Using Pliss lemma (as stated in \cite[Lemma 3.1]{CP}), the first condition in~\eqref{e.stable}, holds on a set with $\mu-$measure larger than $\frac{(1-\alpha/3)\log(D)-\log(D)}{(1-\alpha/3)\log(D)-\log(m)}>\frac{-\alpha/3}{(1-\alpha/3-\frac{1}{\gamma})}$.
Similarly, the second condition in~\eqref{e.stable} holds on a set with $\mu-$measure larger than
$\frac{(1-\alpha/3)\log(D)-\log(D)}{(1-\alpha/3)\log(D)-2\log(m)+\log(D)}>\frac{-\alpha/3}{(2-\alpha/3-\frac 2 \gamma)}$.
Hence~\eqref{e.stable} holds on a set with measure larger than $1-\varepsilon$ provided $\gamma$ is chosen close to $1$ so that
$\frac{-\alpha/3}{(2-\alpha/3-\frac 2 \gamma)}>1-\varepsilon/2$.
\end{proof}

\section{Closing lemmas}
\label{ss.closing}\index{closing lemma}

The following theorem is proved in~\cite{CP}.
\begin{theorem}\label{t.measure revisited}
For any mildly dissipative diffeomorphism of the disc, the support of any $f$-invariant probability {\color{black} measure} is contained in the closure of the set of periodic points.
\end{theorem}

We state now a local version of that result. 
Let us recall that a compact connected set of the plane is \emph{cellular}
if its complement is connected. Equivalently it is the decreasing intersection
of sets homeomorphic to the unit disc.

\addtocounter{theorem}{-1}
\renewcommand{\thetheorem}{\Alph{theorem}'}
\begin{theorem}[Local version]\label{t.measure local}
Let $f$ be a mildly dissipative diffeomorphism of the disc,
and $\Lambda$ an invariant cellular connected compact set.
Then the support of any $f$-invariant probability {\color{black} measure} on $\Lambda$ is contained in the closure of the periodic points in $\Lambda$.
\end{theorem}
\renewcommand{\thetheorem}{\Alph{theorem}}

This section is devoted to the proof of Theorem~\ref{t.measure local}.

We may assume that $\mu$ is ergodic and that $\mu$ is not supported on a finite
set since otherwise the conclusion of the theorem holds trivially.
We have to find a periodic point in $\Lambda$ arbitrarily close to $x$.
Note that one can replace $f$ by $f^2$ and reduce to the case where $f$ preserves the orientation. Also, by a slight modification of the boundary of the disk, it can be assumed that for almost every point the complement of the local stable manifold in the disc has two connected components.

\begin{definition}\label{d.crossing}
For $\mu$-almost every point $x$, the connected components of $W^s_\mathbb{D}(x)\setminus \{x\}$
are called \emph{stable branches of $x$}.
We say that the connected compact set $\Lambda$ \emph{crosses} a stable branch $\sigma$ of $x$
if there exists a connected compact set $C\subset \Lambda$ which intersects both connected components
of $\mathbb{D}\setminus W^s_\mathbb{D}(x)$ and is disjoint from
$W^s_\mathbb{D}(x)\setminus \sigma$.
\end{definition}

\begin{remark}\label{r.small}
One can build connected compact sets $C'\subset C$ satisfying the definition and contained in arbitrarily small
neighborhoods of $W^s_\mathbb{D}(x)$.
{\rm If this were not the case there would exist a small neighborhood $U$ of $W^s_\mathbb{D}(x)$
such that each connected component of $C\cap U$ is disjoint from one of the
connected components of $\mathbb{D}\setminus W^s_\DD(x)$.
Hence the points in $W^s_\DD(x)\cap C$ would not be accumulated by points of $C$
from both components of $\mathbb{D}\setminus W^s_\DD(x)$:
there would be a continuous partition of $C$ as points to the ``left" or to the ``right" of $W^s_\DD(x)$,
contradicting the connectedness.
}
\end{remark}

\begin{lemma}\label{l.trichotomy}
Three cases occur.
\begin{itemize}
\item[--] for $\mu$-almost every point $x$, the set $\Lambda$ crosses both stable branches of $x$,
\item[--] for $\mu$-almost every point $x$, the set $\Lambda$ crosses one stable branch of $x$
and is disjoint from the other one,
\item[--] for $\mu$-almost every point $x$, the set $\Lambda$ is disjoint from both stable branches of $x$.
\end{itemize}
\end{lemma}
\begin{proof}
We first note that the set of points such that both stable branches $W^s_\mathbb{D}(x)$ are crossed by $\Lambda$
is forward invariant, hence is $f$-invariant on a set with full $\mu$-measure.
Similarly for the set of points having only one stable branch crossed by $\Lambda$.
By ergodicity, three cases occur on a set $X$ with full measure:
$\Lambda$ crosses both branches of each point, or exactly one branch, or none of them.

Pesin theory gives the continuity of $W^s_\mathbb{D}(x)$ for the $C^1$ topology on a set
with positive $\mu$-measure.
Up to removing from $X$ a set with zero measure, one can thus assume that each point $x\in X$,
is accumulated by points $y$ of $X$ in each component of $\mathbb{D}\setminus W^s_\DD(x)$
such that $W^s_\mathbb{D}(x)$ and $W^s_\mathbb{D}(y)$ are arbitrarily close for the $C^1$ topology.

Let us consider a stable branch $\sigma$ of $x\in X$ and assume that there exists $z\in \sigma\cap \Lambda$.
Since $\Lambda$ is compact and invariant, there exists $z_1,z_2\in \sigma\setminus \Lambda$
such that $z$ belongs to the subarc $[z_1,z_2]$ of $\sigma$ connecting $z_1$ to $z_2$.
Since $\Lambda$ is connected and is not contained in $W^s_\mathbb{D}(x)$, there exists a compact connected set $C\subset \Lambda$
that intersects $[z_1,z_2]$, that is contained in a small neighborhood of $[z_1,z_2]$,
and that contains a point $\zeta\in \Lambda\setminus W^s_\mathbb{D}(x)$.
Considering a point $y$ as above in the same component of $\mathbb{D}\setminus W^s_\DD(x)$ as $\zeta$,
one deduces that the stable branch $\sigma_y$ of $y$ close to $\sigma$ is crossed by $\Lambda$.
Not also that if the other stable branch $\sigma'$ of $x$ is crossed by $\Lambda$, then
the stable branch $\sigma'_y$ of $y$ that is close to $\sigma_y$ is also crossed by $\Lambda$.
Since $x$ and $y$ have the same number of stable branches crossed by $\Lambda$,
one deduces that $\sigma$ is crossed by $\Lambda$.
\end{proof}

For proving Theorem~\ref{t.measure local},
the three cases of Lemma~\ref{l.trichotomy} have to be addressed. 

In the first case, using that both branches of local stable manifolds intersects $\Lambda$,  for a point $x$ in a hyperbolic block of the measure $\mu$, we build a rectangle that contains $x$ in its interior and the boundary of the rectangle are given by two local stable manifolds of generic points of the measure and two connected arcs contained in $\Lambda$ (this is done in Claim \ref{neighborhood R}). By Theorem~\ref{t.measure revisited} there is a periodic point close to $x$ and so in the interior of the rectangle; on the other  hand, by the construction of the rectangle, forward iterates of it converge to $\Lambda$;  therefore, the periodic point in the interior of the rectangle, has to be in the intersection of $\Lambda$ with the rectangle (this is explained immediately after the proof of Claim \ref{neighborhood R}).

In the second case, a similar rectangle (with boundaries given by two local stable manifolds of generic points of the measure and two connected arcs contained in $\Lambda$) can be built. However, that rectangle  does not contain points of $\Lambda$ in its interior and so Theorem \ref{t.measure revisited} does not guarantee the existence of a periodic point {\color{black} inside} (the periodic points provided by that theorem accumulate on the boundary of the rectangle). So a different strategy has to be formulated, which is described after the preparatory Claim \ref{c.reduction}. 

For the third case, we use a slight variation of the strategy developed for the second case.

\paragraph{\it First case:
$\Lambda$ crosses both stable branches of $x$.}
We select a neighborhood of $x$ verifying:

\begin{claim}\label{neighborhood R}
There is a neighborhood $R$ of $x$
whose boundary is contained in $\Lambda\cup W^s_\mathbb{D}(x')\cup W^s_\mathbb{D}(x'')$
where $x',x''$ are iterates of $x$.
\end{claim}
\begin{proof}
From Pesin theory, there exists a set $X$ with positive measure for $\mu$
such that $W^s_\mathbb{D}(z)$ exists and varies continuously with $z\in X$ in the $C^1$ topology.
Since $\mu$ has no atom,
one can furthermore require that any point $z\in X$ is accumulated from  both components
of $\mathbb{D}\setminus W^s_\mathbb{D}(z)$
by forward iterates of $z$ in $X$.
Without loss of generality, one can assume that $x$ belongs to $X$
and consider two forward iterates $x',x''\in X$ of $x$, arbitrarily close to $x$
and separated by $W^s_\mathbb{D}(x)$. See Figure~\ref{f.localization}.
Since $\Lambda$ crosses both stable branches of $x$,
there exists two connected compact sets 
$C_1,C_2\subset \Lambda$ which intersect both curves $W^s_\mathbb{D}(x')$, $W^s_\mathbb{D}(x'')$
and which do not contain $x$.
The connected component $R$ of $\mathbb{D}\setminus (W^s_\mathbb{D}(x')\cup W^s_\mathbb{D}(x'')\cup \Lambda)$
containing $x$ has its closure contained in the interior of $\mathbb{D}$:
otherwise, there would exist an arc connecting $x$ to the boundary of $\mathbb{D}$,
contained in the strip bounded by $W^s_\mathbb{D}(x')\cup W^s_\mathbb{D}(x'')$,
and disjoint from $\Lambda$, contradicting the connectedness of $C_1$ and $C_2$.
\end{proof}

\begin{figure}
\begin{center}
\includegraphics[width=5.3cm,angle=0]{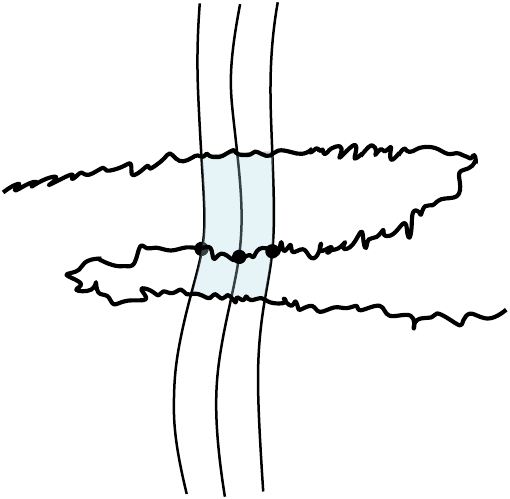}
\put(-33,36){\small $\Lambda$}
\put(-79,65){\small $x$}
\put(-100,63){\small $x'$}
\put(-68,62){\small $x''$}
\put(-95,-10){\small $W^s_{\DD}(x)$}
\put(-89,79){\small $R$}
\end{center}
\caption{$\Lambda$ crosses both stable branches of $x$.\label{f.localization}}
\end{figure}

The volume of the iterates $f^k(R)$ and the length of the iterates $f^k(W^s_\mathbb{D}(x'))$
and $f^k(W^s_\mathbb{D}(x''))$ decreases to zero as $k\to +\infty$.
Hence the distance between $f^k(R)$ and $\Lambda$ goes to zero when $k$ goes to $+\infty$.
By applying Theorem~\ref{t.measure revisited}, there exists a periodic point $q$ in $R$.
Let $\ell$ denote its period.
This periodic point also belongs to $f^{k\ell}(R)$ for $k$ arbitrarily large,
hence it also belongs to $\Lambda$ by our construction. The theorem follows in that case.

\paragraph{\it Second case: $\Lambda$ crosses only one stable branch of almost every point $x$.}
As in the proof of the previous claim,
we introduce a compact Pesin block $X\subset \Lambda$ for $\mu$ with no isolated point,
containing $x$ and with positive $\mu$-measure. One can replace $x$ by another point close in $X$
and require that $x$ is accumulated by $X$ in each  components of $\mathbb{D}\setminus W^s_\mathbb{D}(x)$.

\begin{claim}\label{c.reduction}
There exists $N\geq 0$ (arbitrarily large) and two points $x',x''\in X$ such that
\begin{itemize}
\item[--] $W^s_\mathbb{D}(f^{-N}(x))$ separates $f^{-N}(x')$ and $f^{-N}(x'')$ in $\mathbb{D}$,
\item[--] the image by $f^N$ of the strip bounded by $W^s_\mathbb{D}(f^{-N}(x'))$ and $W^s_\mathbb{D}(f^{-N}(x''))$ in $\mathbb{D}$
is an arbitrarily small neighborhood $R$ of $x$,
\item[--] $x''$ is a forward iterate $f^j(x')$ of $x'$,
\item[--] for any $n\geq 1$, the point $x''$ is accumulated by its forward iterates under $f^n$
in both components $\mathbb{D}\setminus W^s_\mathbb{D}(x'')$.
\end{itemize}
\end{claim}
\begin{proof}
Since the length of the iterates $f^n(W^s_\mathbb{D}(z))$ decreases uniformly to $0$ as $n$ goes to $+\infty$,
the curve $f^N(W^s_\mathbb{D}(f^{-N}(x)))$ is arbitrarily small for $N$ large enough.
Note that $f^N(\partial \mathbb{D})$ crosses both stable branches of $x$.
Considering points $x',x''$ close to $x$ in $X$, one defines a rectangle $R$ bounded
by $f^N(\partial \mathbb{D})\cup W^s_\mathbb{D}(x')\cup W^s_\mathbb{D}(x'')$.
Since $x'$ and $x''$ can be chosen in different components of $\mathbb{D}\setminus W^s_\mathbb{D}(x)$,
the point $x$ belongs to the interior of $R$.

Since $X$ has positive measure, one can choose $x',x''$ in the same orbit.
Moreover, up to removing a set with zero measure, one can choose $x'$ (and $x''$)
to be accumulated by its forward iterates under $f^n$ (for any $n\geq 1$) inside both components $\mathbb{D}\setminus W^s_\mathbb{D}(x'')$.
\end{proof}

In the following, one replaces $\mathbb{D}$ by $f^N(\mathbb{D})$ and $f$ by $f^j$.
Hence without any loss of generality one reduces to the case
where:
\begin{itemize}
\item[--] $W^s_\mathbb{D}(x)$ separates $x'$ and $x''$,
\item[--] $R$ is the strip in $\mathbb{D}$ bounded by $W^s_\mathbb{D}(x')$ and $W^s_\mathbb{D}(x'')$,
\item[--] $f(x')=x''$.
\end{itemize}
We now have to find a periodic point $q$ in $R\cap \Lambda$.
The ergodicity of the measure will not be used anymore.
We denote by $D'$ (resp. $D''$) the (open) component of $\mathbb{D}\setminus W^s_{\mathbb{D}}(x')$
(resp. of $\mathbb{D}\setminus W^s_{\mathbb{D}}(x'')$) which does not contain $W^s_{\mathbb{D}}(x'')$
(resp. $W^s_{\mathbb{D}}(x')$). See Figure~\ref{f.single-cross}.
\medskip

\begin{figure}
\begin{center}
\includegraphics[width=7cm,angle=0]{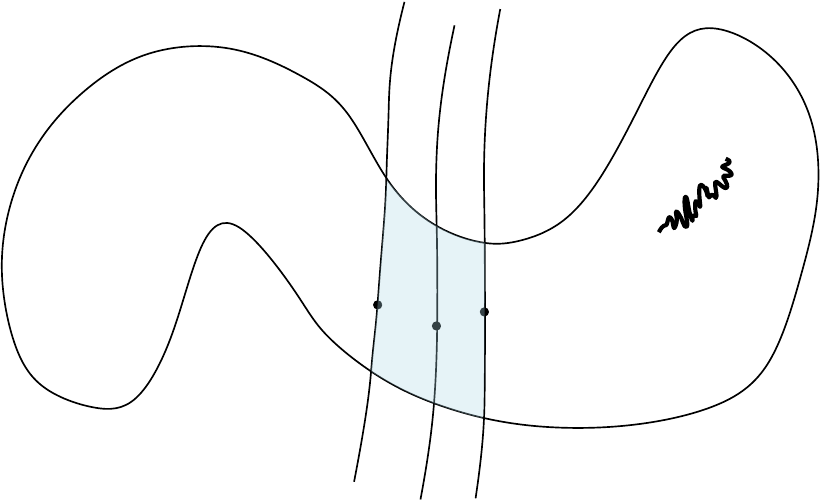}
\put(-33,34){\small $D''$}
\put(-33,85){\small $A$}
\put(-170,45){\small $D'$}
\put(-91,40){\small $x$}
\put(-117,45){\small $x'$}
\put(-80,50){\small $x''$}
\put(-112,-10){\small $W^s_{\DD}(x)$}
\put(-105,60){\small $R$}
\put(-150,115){\small $f^N(\mathbb{D})$}
\end{center}
\caption{Localization when $\Lambda$ crosses one or no stable branch of $x$.\label{f.single-cross}}
\end{figure}

The strategy now consist in using the stable manifolds of generic points of the measure to build a forward invariant cellular set $\Delta$ that contains $\Lambda$ and such that its forward iterates converge to $\Lambda$ (see Lemma \ref{l.delta}). Then, after  considering the following three sets, $\Delta'=\Delta\cap D', \Delta''=\Delta\cap D''$ and $\Delta\cap R$, we show that it is possible to build a continuous map $g$ that sends $\Delta$ into itself, coincides with an iterate of $f$ in $R$ and satisfies $g(\Delta')\cap   \Delta'=\emptyset$ and $g(\Delta'')\cap   \Delta''=\emptyset$ (see Lemma \ref{l.g}). From Proposition \ref{p.CL} it follows that $g$ has a fixed point in $\Delta;$ since that fixed point can not be neither in $\Delta'$ nor in $\Delta''$, it has to be in $R\cap \Delta$ and  so it is a periodic point for $f$; since the forward iterates of $\Delta$ converges to $\Lambda$, it follows that it has to be in $\Lambda.$ 

The last item of the Claim~\ref{c.reduction} implies that  there exists a compact set $A\subset D''$
which contains arbitrarily large iterates of $x'$ and $x''$,
which are contained in $f^m(X)$ for some $m\geq 1$ such that
$\Lambda$ crosses a stable branch of each point $z\in A$ (and is disjoint from the other one).
The stable curves $W^s_\mathbb{D}(z)$ vary continuously with $z\in A$ for the $C^1$ topology.

\begin{lemma}\label{l.delta}
There exists a connected compact set $\Delta$ which has the following properties:
\begin{itemize}
\item[i.] $\Delta$ is cellular (\emph{i.e.,} its complement is connected).
\item[ii.] $\Delta$ is forward invariant: $f(\Delta)\subset \Delta$.
\item[iii.] The forward orbit of any point in $\Delta$ accumulates on $\Lambda$.
\item[iv.] One stable branch of $x'$ is disjoint from $\Delta$,
the other one intersects $\Delta$ along an arc;
moreover there exists a (non-empty) arc in $W^s_\mathbb{D}(x')$
which contains $x'$ in its closure and is included in the interior of $\Delta$.
The same holds for the stable branches of $x''$.
\item[v.] There is $\varepsilon>0$ such that for any forward iterate $z\in A$ of $x'$,
there exists a curve of size $\varepsilon$ in $W^s_\mathbb{D}(z)$
containing $z$ in its closure and included in the interior of $\Delta$.
\end{itemize}
\end{lemma}
Let us denote $\Delta':=\Delta\cap D'$ and $\Delta'':=\Delta\cap D''$.
Note that it is enough now to obtain a periodic point $q\in R\cap \Delta$.
Indeed, since the accumulation set of the forward orbit of $q$ coincides with the orbit of $q$,
the item (ii) ensures that $q\in\Lambda$ as required.

\begin{proof}
We consider for each $z\in X\subset \Lambda$
the maximal curve $I_z$ in $W^s_{\mathbb{D}}(z)$ bounded by points of $\Lambda$
(possibly reduced to a point). The union $\Delta_0$ of $\Lambda$ with all the forward iterates
of the curves $I_z$, $z\in X$, is a forward invariant set which is compact (since the set $X$ is compact,
 the curves $W^s_{\mathbb{D}}(z)$ vary continuously with $z\in X$ in the $C^1$ topology
and  the length of their iterates decreases uniformly) and is connected (since $\Lambda$ is connected).
The set $\Delta$ is obtained by filling the union $\Delta_0$, \emph{i.e.,} it coincides with the complement
of the connected component of $\mathbb{D}\setminus \Delta_0$ which contains the boundary of $\mathbb{D}$.
Properties (i) and (ii) are satisfied.

In order to prove Property (iii), we consider a point $y\in \Delta$.
Note that if $y$ belongs to $\Lambda$ or to some $W^s_{\mathbb{D}}(z)$ with $z\in X$,
the conclusion of (iii) holds trivially. We thus reduce to the case where $z$ belongs to
a connected component $C$ of $\Delta\setminus \Delta_0$.
Note that the boundary of this component decomposes as the union of a subset of $\Lambda$
and a set contained in the union of the $f^n(W^s_\mathbb{D}(z))$ with $z\in X$ and $n\geq 0$.
Since the volume decreases under forward iterations,
for $n$ large enough the point
$f^n(z)$ gets arbitrarily close to the boundary of $f^n(C)$. Since the length of stable manifolds
$f^n(W^s_\mathbb{D}(z))$ gets uniformly arbitrarily small as $n\to +\infty$,
any point in $f^n(C)$ is arbitrarily close to $\Lambda$ provided $n$ is large enough, proving (iii).

By construction of $\Delta$, for any point $z\in X$,
the intersection $\Delta\cap W^s_\mathbb{D}(z)$ is an arc bounded by two points of $\Lambda$
(and not reduced to $z$).
This is the case in particular for the intersections $\Delta\cap W^s_\mathbb{D}(x')$
and $\Delta\cap W^s_\mathbb{D}(x'')$.
Since one stable branch of $x'$ (resp. $x''$) does not meet $\Lambda$, the first part of item (iv) follows.

Let $\sigma$ be the stable branch of $x'$ that is crossed by $\Lambda$
and let us choose a connected compact set $C_1\subset \Lambda$ as in Definition~\ref{d.crossing}.
One can choose another set $C_2$ which is contained in an arbitrarily small neighborhood of $x$.
Indeed, let $f^{-k}(x')$ be a backward iterate of $x'$ in $X$.
By choosing $k$ large, the image $f^k(W^s_\mathbb{D}(f^{-k}(x')))$
gets arbitrarily small. Let $C'_2$ be a connected set crossing a stable branch of $f^{-k}(x')$
as in Definition~\ref{d.crossing}. One can choose it in a small neighborhood of $W^s_\mathbb{D}(f^{-k}(x'))$
(by Remark~\ref{r.small}), hence the image $C_2:=f^k(C'_2)$ is contained in a small neighborhood of $x$.

One deduces that the smallest arc $\gamma$ connecting $C_1$ to $C_2$ inside $W^s_\mathbb{D}(x')$
is contained (after removing its endpoints) in the interior of $\Delta$.
Indeed, one can choose two points $y_l,y_r\in X$ close to $x'$,
separated by $W^s_\mathbb{D}(x')$ and with stable curves close to $W^s_\mathbb{D}(x')$ for the $C^1$ topology.
These curves are crossed by $C_1$ and $C_2$, hence the connected components of
$\mathbb{D}\setminus (C_1\cup C_2 \cup W^s_\mathbb{D}(y_l) \cup W^s_\mathbb{D}(y_r))$
containing $\gamma$ is bounded away from $\partial \mathbb{D}$ and contained in $\Delta$ as claimed.
See Figure~\ref{f.single-cross2}.

\begin{figure}
\begin{center}
\includegraphics[width=2cm,angle=0]{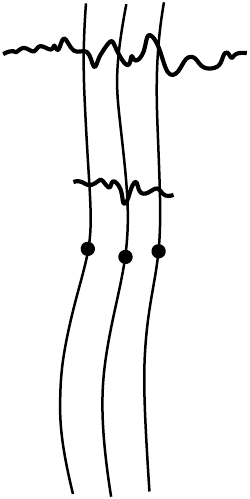}
\put(-68,100){\small $C_1$}
\put(-52,71){\small $C_2$}
\put(-35,50){\small $y$}
\put(-48,50){\small $y_l$}
\put(-20,50){\small $y_r$}
\put(-50,-10){\small $W^s_{\DD}(x)$}
\end{center}
\caption{The interior of $\Delta$ contains stable arcs. \label{f.single-cross2}}
\end{figure}

Since $C_2$ can be chosen in an arbitrarily small neighborhood of $x'$,
one deduces that the interior of $\Delta$
contains a (non-empty) arc in $W^s_\mathbb{D}(x')$ which contains $x'$ in its closure.
The same holds for the point $x''$ and (iv) is satisfied.

As a consequence, for any $z$ in a forward iterate of $X$,
the intersection $\Delta\cap W^s_\mathbb{D}(z)$ is a finite union of arcs bounded by points of $\Lambda$.

In order to check (v), one notices that by the same argument as in the previous paragraph,
for any point $z_0\in A$, there exists a non-trivial curve $\alpha_{z_0}\subset W^s_\mathbb{D}(z_0)$ contained in
the interior of $\Delta$ and containing $z_0$ in its closure.
By construction, the length of $\alpha_z$ is bounded from below for any $z\in A$ close to $z_0$.
by compactness of $A$, there exists a uniform bound $\varepsilon>0$ for all $\alpha_z$ with $z\in A$,
proving (v).
\end{proof}

Let $I':=\Delta\cap W^s_{\mathbb{D}}(x')$
and $I'':=\Delta\cap W^s_{\mathbb{D}}(x'')$.
By (iv), these are arcs.

\begin{claim}
For $\ell\geq 1$ large, $f^\ell(I'\setminus \{x'\})$ and $f^\ell(I''\setminus \{x''\})$
are in the interior of $\Delta$.
\end{claim}
\begin{proof}
Let us consider a large forward iterate $f^k(x')\in A$. The image $f^k(I'\setminus \{x'\})$ is arbitrarily small,
hence by (iv) and (v) is contained in the interior of $\Delta$.
By (ii) the interior of $\Delta$ is forward invariant.
This shows that for any integer $\ell$ large, the image
$f^\ell(I'\setminus \{x'\})$ is contained in the interior of $\Delta$.
And the same holds for $f^\ell(I''\setminus \{x''\})$.
\end{proof}

We choose $\ell\geq 1$ large such that $f^\ell(x')\in D''$.
Since the stable manifolds are disjoint or coincide,
this gives $f^\ell(I')\subset D''$.
Note that since $\mu$ is not a periodic measure, the large forward iterate of $x'$
do not intersect $W^s_\mathbb{D}(x')$.
One can thus choose $\ell$ such that we have also $f^{\ell+1}(x')\not\in \overline{D''}$:
this gives $f^{\ell+1}(I'')\subset \mathbb{D}\setminus \overline{D''}$.
We fix such an iterate $f^\ell$.

\begin{lemma}\label{l.g}
There exists a continuous map $g$ which:
\begin{itemize}
\item[(a)] maps $\Delta$ inside itself,
\item[(b)] is the restriction of an orientation-preserving homeomorphism of the plane,
\item[(c)] satisfies $g(\Delta')\cap\Delta'=\emptyset$ and $g(\Delta'')\cap\Delta''=\emptyset$,
\item[(d)] coincides with $f^\ell$ on $R$.
\end{itemize}
\end{lemma}
\begin{proof}
One chooses two small neighborhoods $U',U''$ of $x'$ and $x''$.
One builds a homeomorphism $\varphi$ which coincides with the identity on $R$ and near the boundary of $\mathbb{D}$
and which sends $\Delta'$ in a small neighborhood of $I'$ and $\Delta''$ in a small neighborhood of $I''$.

More precisely,
from Property (iv) of Lemma~\ref{l.delta}, the curve $(U'\cap I')\setminus \{x'\}$ is contained in the interior of $\Delta$
and the stable branch of $x'$ which does not meet $I'$ is disjoint from $\Delta$.
One can thus require that $\varphi(\Delta'\cap U')\subset \Delta'$ so that $f^\ell\circ \varphi(\Delta'\cap U')\subset \Delta$.
One can furthermore require that $\Delta'$ is sent in a small neighborhood of $I'$.
By our choice of $\ell$,
the compact arc $f^\ell(I' \setminus U')$ is contained in the interior of $\Delta$, hence one gets 
$f^\ell\circ \varphi(\Delta'\setminus U')\subset \Delta$.
This shows that $g:=f^\ell\circ \varphi$ satisfies $g(\Delta')\subset \Delta$.
A similar construction in $D''$, implies that $g(\Delta'')\subset \Delta$.
Since $f^\ell(\Delta)\subset \Delta$ by (ii), this implies $g(\Delta)\subset \Delta$, hence (a).
The properties (b) and (d) follows from the definition of $\varphi$ and $g$.

Since $\varphi(\Delta')$ is contained in a small neighborhood of $I'$ and since $f^\ell(I')\subset D''$,
one gets $g(\Delta')\subset D''$. Similarly $g(\Delta'')\subset \Delta'$.
Hence Property (c) holds.
\end{proof}

From (a), the sequence $g^n(\Delta)$ is decreasing and their  intersection
$\widetilde \Delta$ is $g$-invariant. From Property (i) and as the intersection of a decreasing sequence of cellular sets, it is cellular.
Together with (b), one can apply Cartwright-Littlewood's theorem (Proposition~\ref{p.CL}): the orientation preserving homeomorphism of the plane
$g$ has a fixed point $q\in \widetilde \Delta\subset \Delta$.
From (c), the fixed point does not belong to $\Delta'\cup \Delta''$, hence it belongs to $R\cap \Delta$.
From (d), it is an $\ell$-periodic point of $f$, as we wanted.
The proof of the theorem follows in the second case.

\paragraph{\it Third case: $\Lambda$ is disjoint from the two stable branches of almost every $x$.}
We adapt the proof done in the second case.
We can first reduce to the setting of the Figure~\ref{f.single-cross}:
$W^s_\mathbb{D}(x)$ separates two points $x'$ and $x''=f(x')$;
$R$ is the strip bounded by $W^s_\mathbb{D}(x')$ and $W^s_\mathbb{D}(x'')$;
we have to find a periodic point $q$ in $R\cap \Lambda$.

In this case, for any iterate $f^k(x')$, the set $\Lambda$ intersects $W^s_\mathbb{D}(f^k(x'))$
only at $x'$.
We choose $\ell\geq 1$ large such that $f^\ell(x')\in D''$ and $f^{\ell+1}(x')\in \mathbb{D}\setminus \overline{D''}$.
Note that the sets $(\Lambda\cap D')\cup \{x'\}$,
$(\Lambda\cap R)\cup \{x',x''\}$, $(\Lambda\cap D'')\cup \{x''\}$ are compact, connected,
and only intersect at $x'$ or $x''$.
The image by $f^\ell$ of the second intersects both $D''$ and $\mathbb{D}\setminus \overline{D''}$:
consequently it contains $x''$.
One deduces that the image $f^\ell(\Lambda\cap \overline{D'})$ does not intersect $W^s_\mathbb{D}(x'')$,
hence is contained in $D''$.
For the same reason the image $f^\ell(\Lambda\cap \overline{D''})$ does not intersect $W^s_\mathbb{D}(x'')$,
hence is contained in $\mathbb{D}\setminus \overline{D''}$.
This proves that $f^\ell$ has no fixed point in $(D'\cup D'')\cap \Lambda$.
By Cartwright-Littlewood's theorem (Proposition~\ref{p.CL}) it has a fixed point in the cellular set $\Lambda$,
hence in $\Lambda\cap R$ as wanted.
The proof of Theorem~\ref{t.measure local} is now complete. \qed

\section{No cycle}
\label{no cycle section}
One says that a diffeomorphism $f$ admits a \emph{cycle of periodic orbits}\index{cycle, no cycle} if there exists a {\color{black} sequence} of periodic orbits
$\cO_0$, $\cO_1$, $\dots,$ $\cO_n=\cO_0$ such that for each $i=0,\dots,n-1$, the unstable set of $\cO_i$
accumulates on $\cO_{i+1}$.
The goal of this section is to prove the following:

\begin{theorem}\label{t.cycle}
A mildly dissipative diffeomorphisms of the disc has zero topological entropy if and only if it does not admit any cycle of periodic orbit.
\end{theorem}

This result can be localized.
A set $U$ is \emph{filtrating}\index{filtrating set} for $f$ if it may be written as the intersection of two open sets $U=V\cap W$
such that $f(\overline V)\subset V$ and $f^{-1}(\overline W)\subset W$.

\addtocounter{theorem}{-1}
\renewcommand{\thetheorem}{\Alph{theorem}'}
\begin{theorem}[Local version]\label{t.cycle2}
Let $f$ be mildly dissipative diffeomorphisms of the disc and $U$ be a filtrating set.
The restriction of $f$ to $U$ has zero topological entropy
if and only if it does not admit any cycle of periodic orbits contained in $U$.
\end{theorem}
\renewcommand{\thetheorem}{\Alph{theorem}}

The non-existence of {\color{black} a} cycle of periodic orbits {\color{black} extends} to fixed arcs. One says that a diffeomorphism $f$ admits a \emph{cycle of fixed arcs} if there is  sequence of disjoint fixed arcs $I_0,I_1,\dots, I_{n}=I_0$
such that {\color{black} for $i=0,\dots,n-1$ the} arc $I_i$ admits a $f$-invariant unstable branch which accumulates
on $I_{i+1}.$
\begin{corollary}\label{c.cycle}
Consider a mildly dissipative diffeomorphism $f$ of the disc.
If $f$ has zero topological entropy, then it does not admit any cycle of fixed arcs.

The same property holds inside any filtrating set $U$.
\end{corollary}

\begin{proof} Let us assume that $f$ admits {\color{black} a cycle} of fixed {\color{black} arcs $I_0,\dots,I_n$}:
for each $i$, there exists a fixed point $p_i\in I_i$ with a $f$-invariant unstable branch $\Gamma_i$ that accumulates on $I_{i+1}$.
We may assume that the length $n$ is minimal: {\color{black} there is no cycle with smaller length.}

{\color{black} \begin{claim}
For each arc $I_i$, one component $V_i$ of $\mathbb{D}\setminus W^s_\mathbb{D}(I_i)$ contains all the other arcs $I_j$.
Moreover $W^s_\mathbb{D}(I_{i})$ is included in the boundary of $V_i$.
\end{claim}}
\begin{proof}
Since $f$ has no cycle of fixed points, {\color{black} the unstable branch} $\Gamma_i$ is disjoint from $W^s_{\mathbb{D}}(p_i)$.
As a consequence, $\Gamma_i$ is contained in a component {\color{black} $V_i$} of $\mathbb{D}\setminus W^s_\mathbb{D}(I_i)$,
which also contains $I_{i+1}$.
{\color{black} Let us assume by contradiction that $V_i$ does not contain all the arcs $I_j$, $j\neq i$.
One can thus find an arc $I_{j-1}\subset V_i$ such that $I_j$ is not contained in $V_i$.}
Consequently the unstable branch {\color{black} $\Gamma_{j-1}$} crosses $W^s_\mathbb{D}(I_i)$ and
{\color{black} there exists a cycle with smaller length.} The length $n$ of the cycle {\color{black} can} not be minimal.
{\color{black} By construction $W^s_\mathbb{D}(I_{i})$ is included in the boundary of $V_i$.}
\end{proof}

{\color{black} The component $V_i$ of $\mathbb{D}\setminus W^s_\mathbb{D}(I_i)$ contains the point $p_{i-1}$.
One deduces that $W^s_\mathbb{D}(I_i)$ separates $p_{i-1}$ from the arc $I_i$. Since $\Gamma_i$ accumulates on $I_i$,
it accumulates on $p_i$.  We have thus showed that} the {\color{black} sequence of} fixed points $p_0,p_1,\dots,p_n$ {\color{black} defines} a cycle and
by Theorem~\ref{t.cycle} $f$ has positive topological entropy.

Using Theorem~\ref{t.cycle2}, one gets the same property inside filtrating sets.
\end{proof}

{\color{black} As an example, one may consider a cycle reduced to a single fixed point:}
the unstable manifold of the fixed point accumulates on the stable one. In that context, \cite{pixton} proved that either there exists a transversal homoclinic point (and so the topological entropy is positive) or that {\color{black} a transverse homoclinic} intersection can be created by a smooth perturbation. Under the hypothesis of mild dissipation, we prove that if such a cycle {\color{black} of length $1$} exists, the topological  entropy is positive even if there is {\color{black} no} transverse intersection between the invariant manifolds of the fixed point.

The end of this section is devoted to the proof of Theorems~\ref{t.cycle} and~\ref{t.cycle2}.

\subsection{Homoclinic orbit of a fixed point}

We first consider the case of a cycle of a unique fixed point with a homoclinic orbit. 

\begin{lemma}\label{l.heteroclinic-cycle}
Let $f$ be a dissipative diffeomorphism of the disc.
If there exists a point $p$ with a fixed unstable branch $\Gamma$
which intersects $W^s(p)$,
then the topological entropy of $f$ is positive.
\end{lemma}
\begin{proof}
Let us fix $x\in \Gamma$ which also belongs to the local manifold $W^s_{loc}(p)$.
We denote by $\gamma$ the arc of $\Gamma$ which connects $p$ to $x$.
Up to {\color{black} replacing} $f$ by $f^2$, one can assume that the eigenvalues of $Df(p)$ are $0<\lambda<1\leq \mu$.
Since $f$ contract the volume, $\lambda\mu<1$.

We claim that there exists some point $z^s\in W^s_{loc}(p)$ such that $x$
belongs to the interior of the segment $[z^s,f(z^s)]$ in $W^s_{loc}(p)$
and such that $z^s$ is not accumulated by $\Gamma$.
Indeed let $D$ be the disc bounded by $\gamma$ and the arc connecting $x$ to $p$ inside $W^s_{loc}(x)$.
One may assume that $\Gamma$ does not cross $W^s_{loc}(x)$ (otherwise one can immediately
conclude that the entropy is positive), hence $\Gamma$ is contained in $D$.
Since $f$ is dissipative, $f(D)$ is strictly contained in $D$,
hence there exists $z^s\in W^s_{loc}(x)\setminus D$ such that $f(z^s)\in D$.
The required property follows.

Similarly, one can choose some point $z^u\in \Gamma$
such that $x$
belongs to the interior of the segment $[z^u,f(z^u)]$ in $\Gamma$
and such that the orbit of $z^u$ does not intersect the stable arc
$W^s_{loc}(p)$.
One chooses two small $C^1$ arcs $\alpha$, $\alpha'$ transverse to  $W^s_{loc}(p)$
at $z^s$ and $f(z^s)$. And one fixes two small arcs $\beta$, $\beta'$ transverse to $\gamma$
at $z^u$ and $f(z^u)$.
For $n$ large, there exist four arcs
$B\subset f^{-n}(\beta)$, $B'\subset f^{-n}(\beta')$
and $A\subset\alpha$, $A'\subset \alpha'$ which bound a rectangle $R$
whose $n$ first iterates remain close to the forward orbit of $x$
and the backward orbit of $x$.
See Figure~\ref{f.cross}.
\begin{figure}[ht]
\begin{center}
\includegraphics[scale=0.37]{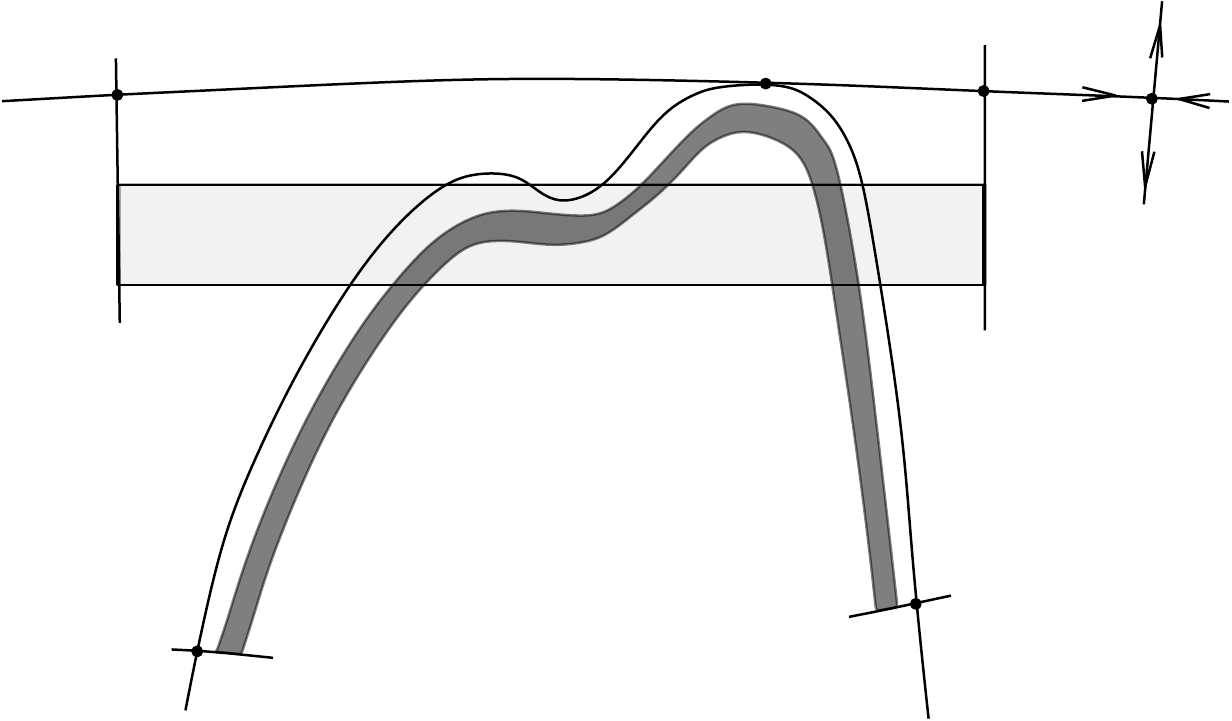}
\begin{picture}(0,0)
\put(-240,114){$W^s_{loc}(p)$}
\put(-198,115){$\alpha$}
\put(-47,115){$\alpha'$}
\put(-15,115){$p$}
\put(-55,0){$\Gamma$}
\put(-80,7){$\beta'$}
\put(-170,8){$\beta$}
\put(-190,85){$R$}
\put(-190,98){$\color{black} B$}
\put(-190,67){$\color{black} B'$}
\put(-162,40){$f^{n}(R)$}
\put(-95,120){$x$}
\end{picture}
\end{center}
\vspace{-0.5cm}
\caption{ Proof of Lemma~\ref{l.heteroclinic-cycle}.\label{f.cross}}
\end{figure}

For any $\varepsilon>0$ {\color{black} satisfying $(1+\varepsilon)^2\lambda\mu<1$} there is $C>0$ such that
if $n$ is large enough,
$$\min\big(d(B, W^s_\mathbb{D}(p)),d(B', W^s_\mathbb{D}(p))\big)\geq
C^{-1}(1+\varepsilon)^{-n}\mu^{-n},$$
$$\max \big(d(f^{n}(A),\gamma),
d(f^{n}(A'),\gamma)\big)
\leq C(1+\varepsilon)^{n}\lambda^{n}.$$
One chooses the integer $n$ such that
$$C^{-1}(1+\varepsilon)^{-n}\mu^{-n}>C(1+\varepsilon)^{n}\lambda^{n}.$$
In particular $f^{n}(R)$ ``crosses" $R$.

One deduces that for any curve $\delta$ in $R$ which connects the arcs $B,B'$,
the image $f^n(\delta)$ contains two curves $\delta'_1,\delta'_2\subset R$ which also
connect the arcs $B,B'$, and which are {\color{black} $\eta$}-separated for some ${\color{black} \eta}>0$
independent from $\delta$.
One can thus iterate $\delta$ and apply the property inductively. This implies
that the topological entropy is positive.
\end{proof}

\subsection{Periods and heteroclinic orbits.}

The following proposition  allows to get (topological) transverse heteroclinic intersections
between periodic orbits with different periods and will be used again in other sections.
\begin{proposition}\label{p.heteroclinic}
Let $f$ be a mildly dissipative diffeomorphism of the disc which preserves the orientation and {\color{black} has}  zero topological entropy.
Let $p$ be a fixed point having a real eigenvalue larger or equal to $1$
and $q$ be a periodic point with an unstable branch $\Gamma_q$ which is not fixed by $f$.
If $\Gamma_q$ accumulates on $p$,
then it intersects both components of $\mathbb{D}\setminus W^s_\mathbb{D}(p)$.
\end{proposition}

\begin{proof} First observe that the period of $q$ is larger than one: if it is fixed, since the branch $\Gamma_q$ that accumulates on the fixed points is not invariant, then both unstable branches in each component of $\DD\setminus W^s_\DD(q)$ accumulates on the same fixed point; a contradiction.

Let us assume now by contradiction that $\Gamma_q$ intersects only one component
of $\mathbb{D}\setminus W^s_\mathbb{D}(p)$.
Since the largest eigenvalue at $p$ is positive,
the components are locally preserved by $f$.
Hence each unstable branch $f^k(\Gamma_q)$ intersects the same components as
$\Gamma_q$. This proves that all the iterates of $q$ are contained in a same
component $U$ of $\mathbb{D}\setminus W^s_\mathbb{D}(p)$.

The set $(\partial U)\setminus \{p\}$ is an arc that may be parametrized by $\mathbb{R}$,
hence may be endowed with an order $<$.
{\color{black} For} each iterate $f^k(q)$, {\color{black} note that}  the other iterates $f^j(q)$, $j\neq k$, {\color{black} belong to}
the connected component of $\mathbb{D}\setminus W^s_\mathbb{D}(f^k(q))$ that contains $p$:  otherwise the unstable branch of some $f^j(q)$ {\color{black} would} cross the stable of $f^k(q)$, {\color{black} implying that the entropy is positive, a contradiction.}
{\color{black} Let $V_k$ be the component}
of $\mathbb{D}\setminus W^s_\mathbb{D}(f^k(q))$
which does not contain $p$, nor the other iterates of $q$.
The components $V_k$ are disjoint, hence ordered by their
prints on the boundary of $\partial U$.
This induces an ordering on the iterates of $q$.

\begin{figure}
\includegraphics[width=5cm,angle=0]{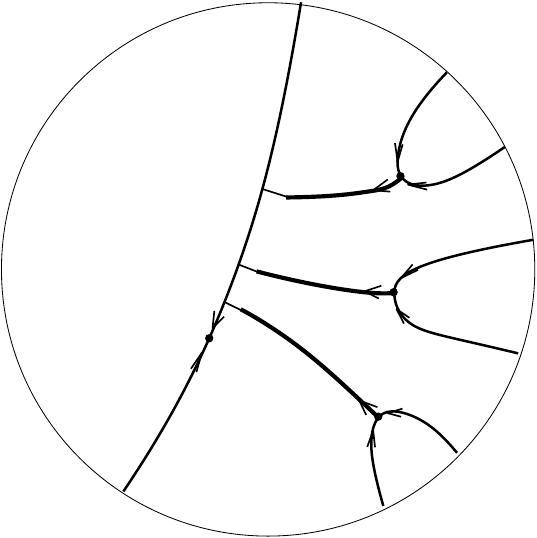}
\hspace{2cm}
\includegraphics[width=5cm,angle=0]{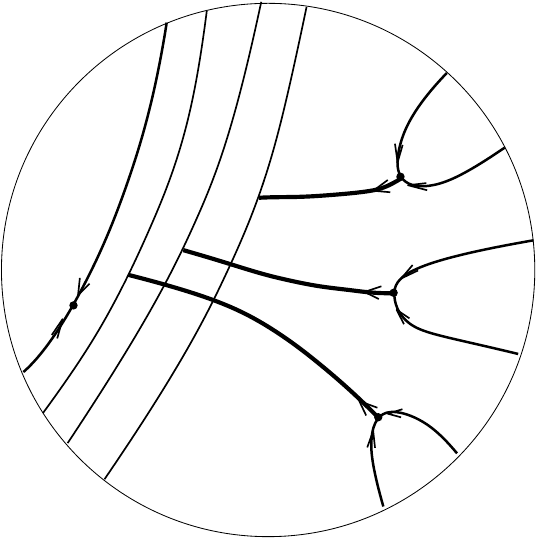}
\put(-300,57){\small $p$}
\put(-250,98){\small $q$}
\put(-233,105){\small $\color{black} V_0$}
\put(-256,70){\tiny $f(q)$}
\put(-225,65){\small $\color{black} V_1$}
\put(-286,91){\small $x$}
\put(-275,84){\small $y$}
\put(-62,145){\small $z$}
\put(-130,60){\small $p$}
\put(-45,98){\small $q$}
\put(-27,102){\small $\color{black} V_0$}
\put(-49,70){\tiny $f(q)$}
\put(-30,61){\small $\color{black} V_1$}
\put(-73,84){\small $y$}
\caption{Cases of the proof of Proposition~\ref{p.heteroclinic}: $W^u(q)$
accumulates $W^s(x)\setminus \{x\}$ ornot.\label{f.period-and-cycle}}
\end{figure}

\paragraph{\it First case. The branch $\Gamma_q$ accumulates on a point
$x$ of $W^s_{\mathbb{D}}(p)$ which is different from $p$.}
The iterates $f^k(x)$ converge to $p$ as $k\to +\infty$. Since $f$ preserves the orientation,
these iterates belong to the same branch of $W^s_\mathbb{D}(p)$. Up to
{\color{black} modifying} the parametrization of the boundary of $U$, one can assume that
the sequence $f^k(x)$ is increasing for the order$<$.

We choose $y\in \Gamma_q$ close to $x$, a small arc $\delta$
connecting $y$ to $x$ and consider the arc $\gamma\subset \Gamma_q$
connecting $q$ to $y$.
This gives an oriented arc $\sigma:=\delta\cup\gamma$ connecting $q$ to $x$
in $U$.

The set $U\setminus (\overline V_0\cup \sigma)$
has two connected components that are Jordan domains.
One of them (denoted by $O$) contains in its boundary all the forward iterates of $x$ and the point $p$.
Up to {\color{black} replacing} $q$ by another point in its orbit,
one can assume that $O$ contains all the iterates $f^k(q)\neq q$, that is $f^k(q)<q$.
See Figure~\ref{f.period-and-cycle}.

Since the endpoints of $\sigma$ (resp. of $f(\sigma)$)
do not belong to $f(\sigma)$ (resp. to $\sigma$),
the algebraic intersection number between $\sigma$ and $f(\sigma)$ is well defined.
Since $\sigma$ is contained in the boundary of $O$
and since the endpoints of $f(\sigma)$ belong to $O$,
the algebraic intersection number between $\sigma$ and $f(\sigma)$ is zero.

This implies that
for any $k\geq 0$, the intersection number between $f^k(\sigma)$ and $f^{k+1}(\sigma)$ is zero.
This proves that in $\overline U\setminus (f^k(\sigma)\cup W^s_\mathbb{D}(f^k(q)))$,
the points $f^{k+1}(x)$ and $f^{k+1}(q)$ are in the same connected component.
Since $f^k(x)<f^{k+1}(x)$,
one deduces that $f^{k+1}(q)<f^{k}(q)$ for any $k\geq 0$. This is a contradiction
since when $k+1$ coincides with the period of $q$, we have
$f^{k+1}(q)=q>f^k(q)$.

\paragraph{\it Second case. The accumulation set of $\Gamma_q$
is disjoint from $W^s_{\mathbb{D}}\setminus \{p\}$.}
We modify the previous argument.
Note that in this case the stable set of $p$ contains a neighborhood of $W^s_\mathbb{D}(p)$ in
$\overline U$. Moreover this neighborhood is foliated by strong stable curves, that we still denotes
by $W^s_\mathbb{D}(z)$.
{\color{black} Up to replacing $q$ by one of its iterates, one can assume that
$V_i<V_0$ for all $i$ such that $f^i(q)\neq q$.}

We choose $y\in \Gamma_q$ in the stable set of $p$ and consider the oriented arc $\sigma\subset \Gamma_q$
connecting $q$ to $y$. One can choose $y$ such that the arc $\sigma$ does not intersects
the component of $U\setminus W^s_{\mathbb{D}}(y)$ containing $p$.
Let $L$ be the half curve in $W^s_\mathbb{D}(y)$ connecting $y$ to a point $z$ in $\partial U$. We can choose the endpoint $z$ such that $V_0<z$. Since $V_1<V_0$, one deduces that $f(q)$ and $f(y)$ belong to the same
connected component of $U\setminus (V_0\cup \sigma\cup L)$. See Figure~\ref{f.period-and-cycle}.
In particular the algebraic intersection number between $\sigma$ and $f(\sigma)$ is zero.

For any $k\geq 0$,
let $L_k$ be the half curve in $W^s_\mathbb{D}(f^k(y))$ connecting $f^k(y)$ to a point $z_k$ in $\partial U$
such that $V_k<z_k$. Since $f^{k+1}(y)$ belongs to the strip
bounded by $W^s_{\mathbb{D}}(p)$ and $W^s_{\mathbb{D}}(f^k(y))$,
we have $z_k<z_{k+1}$. Since the algebraic intersection number between $f^k(\sigma)$
and $f^{k+1}(\sigma)$ is zero, one deduces that
$V_{k+1}$ and $f^{k+1}(y)$ belong to the same component of
$U\setminus (V_k\cup {\color{black} f^k(\sigma)}\cup L_k)$.
In particular $V_{k+1}<V_k$ for any $k\geq 0$. As in the previous case, this is a contradiction.
\end{proof}

\begin{remark}
In the case where $f$ does not preserves the orientation, the
same statement applies if one assumes that the period of $\Gamma_q$ is {\color{black} strictly} larger than $2$
(one applies the previous proposition to $f^2$).
\end{remark}

{\color{black}
\begin{proposition}\label{p.heteroclinic2}
Let $f$ be a mildly dissipative diffeomorphism of the disc which preserves the orientation and {\color{black} has}  zero topological entropy.
Let $\Gamma_p,\Gamma_q$ be unstable branches of some periodic points $p,q$
and $\mathcal{L}_p$ be a stable branch of $p$
such that $\Gamma_q$ intersects $\mathcal{L}_p$.
If $\Gamma_p,\Gamma_q$ have different periods, then
for any tubular neighborhood $T$ of $\mathcal{L}_p$,
the branch $\Gamma_q$ intersects both components of
$T\setminus \mathcal{L}_p$.
\end{proposition}}

{\color{black}
\begin{proof}
We argue by contradiction.
If the periods of $\Gamma_p,\Gamma_q$ do not coincide, there exists an iterate
which fixes one unstable branch $\Gamma_p$, or $\Gamma_q$ and not the other one.
If $\Gamma_p$ is fixed and $\Gamma_q$ is not, the proof is the same as the first case of Proposition~\ref{p.heteroclinic}.
We are thus reduced to the case where $\Gamma_q$ is fixed and $\Gamma_p$ is not.

Let $\tau$ be the period of $\Gamma_p$ and fix some $x\in \Gamma_q\cap \mathcal{L}_p$. Let $\sigma\subset \mathcal{L}_p$, $\gamma\subset \Gamma_q$ be the arcs connecting
$x$ to $f^\tau(x)$ and let $\mathcal{L}_{f(p)}$ be the stable branch of $f(p)$ containing $f(\mathcal{L}_p)$.
Since $\mathcal{L}_{p}$, $\mathcal{L}_{f(p)}$ do not cross $\Gamma_q$ and 
since $f$ is orientation preserving and fixes $\Gamma_q$,
the arcs $\mathcal{L}_{p}$, $\mathcal{L}_{f(p)}$ land at $x, f(x), f^\tau(x)$ from the same half tubular neighborhood of $\Gamma_q$.
Hence $\mathcal{L}_{f(p)}$ meets the interior of the region bounded by $\sigma\cup\gamma$; by mild dissipation it also meet its complement.
See Figure~\ref{f.period-branches}.
This is a contradiction since $\mathcal{L}_{f(p)}$ is disjoint from $\mathcal{L}_{p}$ and does not cross~$\Gamma_q$.
\end{proof}
\begin{figure}[ht]
\begin{center}
\includegraphics[scale=0.5]{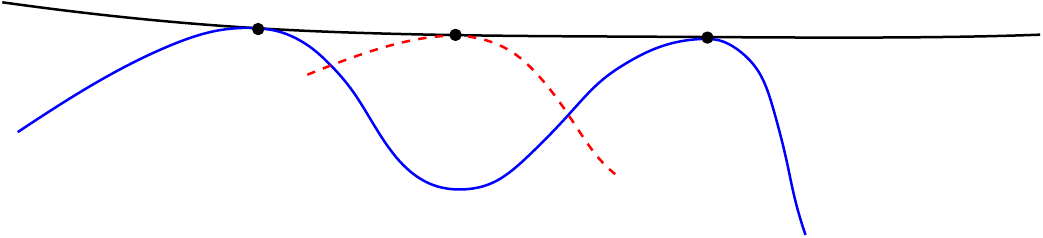}
\begin{picture}(0,0)
\put(-200,41){\small $x$}
\put(-150,35){\small $f(x)$}
\put(-98,35){\small $f^\tau(x)$}
\put(-110,10){$\mathcal{L}_{f(p)}$}
\put(-247,17){$\mathcal{L}_p$}
\put(-250,59){$\Gamma_q$}
\end{picture}
\end{center}
\vspace{-0.5cm}
\caption{Proof of Proposition~\ref{p.heteroclinic2}.  \label{f.period-branches}}
\end{figure}
}

\subsection{Cycles of fixed points}

\begin{lemma}\label{l.cycle-fixed}
If $f$ has a cycle of periodic orbits,
there is an iterate $f^m$, $m\geq 1$ which has a cycle of fixed points.
More precisely, there exists a fixed point $p$ for $f^m$ with a fixed unstable branch $\Gamma$
whose accumulation set contains $\Gamma$.
\end{lemma}
\begin{proof}
Let $\cO_0,\cO_1,\dots\cO_n=\cO_0$ be a cycle of periodic orbits. We extend periodically the sequence $(\cO_k)$ to any $k\in \mathbb{N}$.
By invariance of the dynamics, one deduces that for each $i$ and each $p_i\in \cO_i$, there exists an unstable branch
$\Gamma_i$ of $\cO_i$ which accumulates on a point of $\cO_{i+1}$.
One deduces that for any $p_0\in \cO_0$, there exists points $p_k\in \cO_k$, $k\geq 0$.
such that an unstable branch of $p_{k-1}$ accumulates on $p_{k}$ for each $k\geq 1$.
All the points $p_{\ell n}$ belong to $\cO_0$, hence two of them $p_{\ell_1n}, p_{\ell_2n}$ should coincide.

There exists $m\geq 1$ such that all the points in $\cup_k \cO_k$ are fixed.
The sequence $p_{\ell_1n},p_{\ell_1n+1},\dots,p_{\ell_2n}$ is a cycle of fixed points for $f^m$.
This proves the first assertion.
\medskip

Let us consider a cycle of fixed points for $g=f^m$ with minimal length.
Replacing $m$ by $2m$, one can also assume that all their unstable branches are fixed.
For each fixed point $p_i$ in the cycle, the other fixed points are all contained in a same component $U_i$ of
$\mathbb{D}\setminus W^s_\mathbb{D}(p_i)$: otherwise, one find a point $p_j\neq p_{i-1}$ with an unstable branch
which meets both connected components and one contradicts the minimality of the cycle.

Each fixed point $p_i$ has an unstable branch $\Gamma_i$ whose accumulation set contains $p_{i+1}$.
If $\Gamma_i$ intersects both components of $\color{black} \mathbb{D}\setminus W^s_\mathbb{D}(p_{i+1})$,
by Proposition~\ref{p.transitive} the accumulation set contains $\color{black} \Gamma_{i+1}$ and the second assertion of the lemma holds.
We are thus reduced to assume that $\Gamma_i$ is contained in the closure of $U_i$.
Since $p_i$ and $\Gamma_{i+1}$ are contained in $U_{i+1}$, one deduces that the accumulation set of $\Gamma_i$
contains a point of $\Gamma_{i+1}$;
by Proposition~\ref{p.transitive}, it contains $\Gamma_{i+1}$. Hence {\color{black} one can remove $P_{i+1}$ and} the cycle is not minimal, unless
$p_i=p_{i+1}$, \emph{i.e.,} the cycle {\color{black} is reduced to a unique} fixed point.
The second assertion holds.
\end{proof}

A sequence of fixed unstable branches $\Gamma_0,\Gamma_1,\dots,\Gamma_n=\Gamma_0$
associated to fixed points $p_0,\dots,p_n$ is a \emph{cycle of unstable branches} if for each
$0\leq i<n$, the accumulation set of $\Gamma_i$ contains $\Gamma_{i+1}$.
By Proposition~\ref{p.transitive}, this implies that for each $i,j$, the accumulation set of $\Gamma_i$ contains $\Gamma_j$.
We generalize Lemma~\ref{l.heteroclinic-cycle} to cycles of unstable branches.

\begin{lemma}\label{l.heteroclinic-cycle2}
Let $f$ be a mildly dissipative diffeomorphism of the disc.
If there exists a sequence of fixed unstable branches $\Gamma_0,\Gamma_1,\dots,\Gamma_n=\Gamma_0$
associated to fixed points $p_0,\dots,p_n$ such that
$\Gamma_i$ intersects $W^s_\mathbb{D}(p_{i+1})$ for each $0\leq i<n$, then the topological entropy of $f$ is positive.
\end{lemma}
\begin{proof}
From Lemma~\ref{l.heteroclinic-cycle},
{\color{black} one can assume that}, for each point $p_i$, the unstable branch $\Gamma_i$ is contained in a connected component $U_i$ of
$\mathbb{D}\setminus W^s_\mathbb{D}(p_i)$. Since the accumulation set of $\Gamma_i$ contains the other unstable branches,
all the $p_j$ and the $\Gamma_j$ are contained in $\overline U_i$.
Let $U$ be the intersection of the sets $U_i$:
it is a connected component $U$ of $\mathbb{D}\setminus \cup_i W^s_\mathbb{D}(p_i)$ whose closure contains all
the $\Gamma_i$.

We now argue as for Lemma~\ref{l.heteroclinic-cycle}.
Let us assume that $\Gamma_i$ intersects $W^s_\mathbb{D}(p_{i+1})$ at some point $x_i$.
We build a rectangle $R_i\subset U$ that stretches along
a fundamental domain of $W^s_\mathbb{D}(p_{i+1})$ containing $x_i$.
One chooses $n_i\geq 1$ large such that $f^{n_i}(R_i)$ crosses $R_{i+1}$.
The same argument as before applies following the periodic sequence of rectangles
$R_0,R_1,\dots,R_n=R_0$.
\end{proof}

\subsection{Pixton discs}

Let $p$ be a fixed point with a fixed unstable branch $\Gamma$ which is contained in its accumulation set.
We introduce a notion similar to a construction in~\cite{pixton}, which improved~\cite{robinson1}.
A compact set $D\subset \mathbb{D}$ is a \emph{(topological) disc} if it is homeomorphic to the unit disc.

\begin{definition}\label{d.pixton0}
A \emph{Pixton disc associated to $\Gamma$}\index{Pixton disc} is a disc $D$ bounded by three $C^1$ arcs:\begin{itemize}
\item[--] an arc $\gamma\subset\Gamma$ such that $p$ is one endpoint,
\item[--] an arc $\sigma\subset W^s_\mathbb{D}(p)$ whose endpoints are $p$ and a point
$x\neq p$ accumulated by $\Gamma$,
\item[--] a closing arc $\delta$  disjoint from $f(\delta)$,  joining $\sigma$ and $\gamma$ such that $\delta\cap W^s_\mathbb{D}(p)=\{x\}$.
\end{itemize}
\end{definition}
Note that the last property implies that $f(\delta\setminus \{x\})$
is contained in the interior of $D$.

\begin{figure}[ht]
\begin{center}
\includegraphics[scale=0.24]{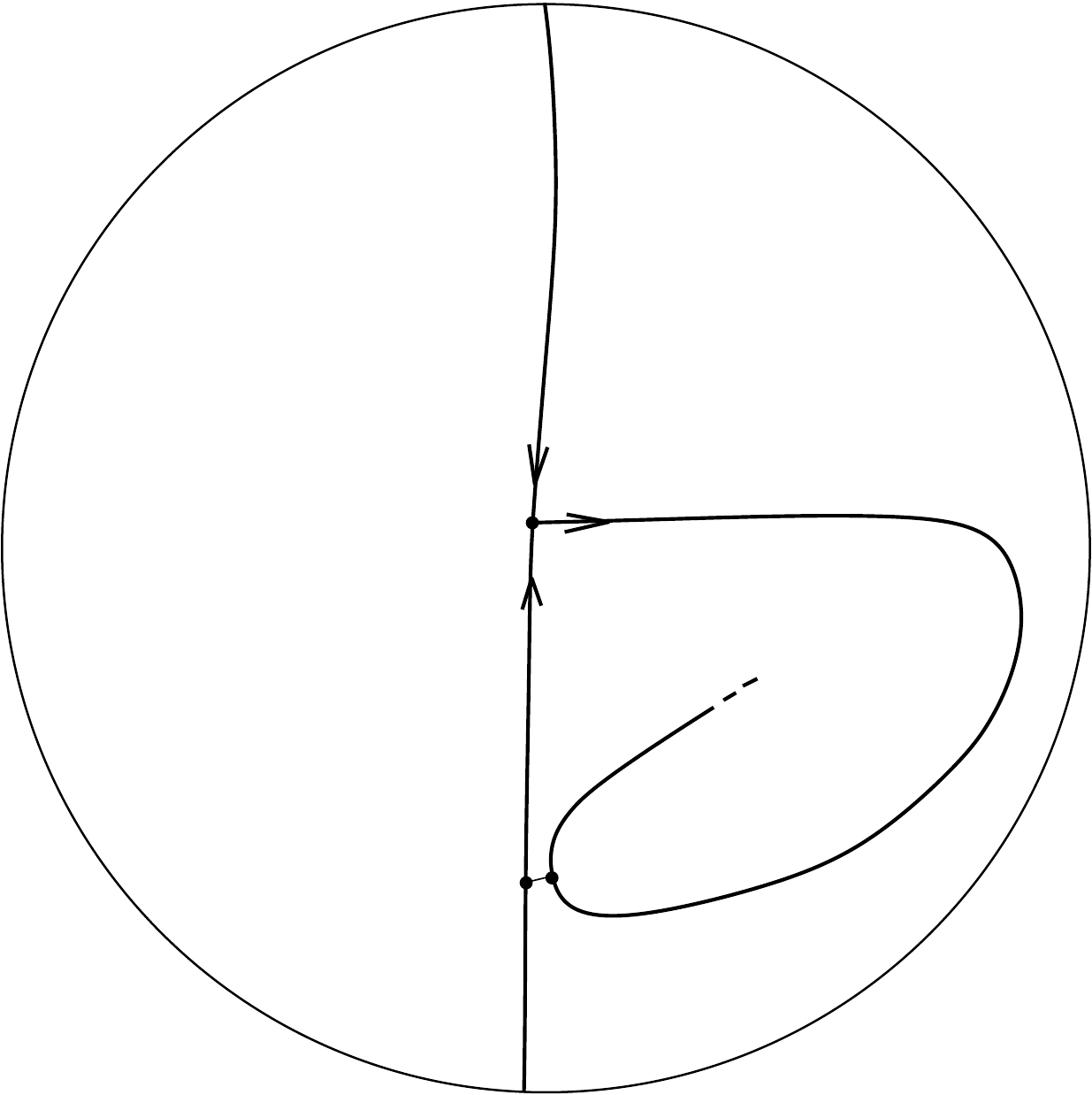}
\begin{picture}(0,0)
\put(-85,55){$\sigma$}
\put(-85,75){$p$}
\put(-55,80){$\gamma$}
\put(-76,12){$\delta$}
\put(-90,28){$x$}
\put(-40,50){$D$}
\put(-110,120){$\DD$}
\end{picture}
\end{center}
\vspace{-0.5cm}
\caption{A Pixton disc.  \label{f.disc}}
\end{figure}

\begin{lemma}\label{l.pixton-cycle}
Let $f$ be a mildly dissipative diffeomorphism of the disc with zero topological entropy
and let $p$ be a fixed point having a fixed unstable branch $\Gamma(p)$
which is contained in its own accumulation set.

Then, there exists an aperiodic ergodic measure $\mu$ such that
\begin{itemize}
\item[--] $\mu(D)=0$ for any Pixton disc $D$ associated to $\Gamma(p)$,
\item[--] the closure of $\Gamma(p)$ contains the support of $\mu$.
\end{itemize}
\end{lemma}
\begin{proof}
{\color{black}
The basic idea is to fix a point $y\in W^s_\mathbb{D}(p)\setminus \{p\}$ which is accumulated by $\Gamma(p)$ and
to consider its $\alpha$-limit set $\alpha(y)$. In good cases one may expect that $\alpha(y)$ does not intersect the interior of Pixton discs $D$
and it is a natural candidate for supporting the measure $\mu$.}
{\color{black}
In practice this is more delicate and we perform the following preparation:
\begin{itemize}
\item[(a)] Up to replacing $f$ by $f^2$, one reduces to the case where $f$ preserves the orientation. (Note that it is enough to prove the {\color{black} lemma} for $f^2$.)
\item[(b)] One considers a cycle of fixed unstable branches $\Gamma(p_0),\Gamma(p_1),\dots,\Gamma(p_n)=\Gamma(p_0)$
associated to fixed points $p_0,p_1,\dots,p_n=p_0$ such that $p_0=p$ and $\Gamma(p_0)=\Gamma(p)$. One supposes that the cardinality $n$ is maximal.
(This is possible since by Proposition~\ref{p.group} all the fixed points belong to finitely many disjoint fixed arcs.)
\item[(c)] There exists $p_i$ such that $W^s_\mathbb{D}(p_i)$ is disjoint from all the unstable branches
$\Gamma(p_j)$ for $1\leq j\leq n$. (This follows from Lemma~\ref{l.heteroclinic-cycle2}.)
\end{itemize}
By Lemma~\ref{p.transitive}, there is $y\in W^s_\mathbb{D}(p_i)\setminus \{p_i\}$ which is accumulated by each $\Gamma(p_j)$.
Let
$$K:=\alpha(y).$$
Indeed the preimages of $y$ are all well-defined in $\mathbb{D}$ (since $y$ is the limit of points in $\Gamma(p_i)$
having infinitely many preimages in $\mathbb{D}$ and since $f(\mathbb D)$ is contained in the interior of $\mathbb D$).
Note that $K$ is contained in the accumulation set of $\Gamma(p)$.
\medskip

The goal now is to show the following two properties:
\begin{itemize}
\item[(1)] {\it The interior of any Pixton disc associated to $\Gamma(p)$ is disjoint from the orbit of $y$.}
\item[(2)] {\it There exists an ergodic aperiodic measure $\mu$ whose support is included in $K$.}
\end{itemize}
They together imply Lemma~\ref{l.pixton-cycle}. Indeed, since $y$ is accumulated by the invariant curve $\Gamma(p)$, the set $K$ (and hence the support of $\mu$ by (2)) is contained in the closure of $\Gamma(p)$. By (1), the interior of any Pixton disc $D$ associated to $\Gamma(p)$
is disjoint from $K$, hence has zero measure for $\mu$. The boundary of $D$ is made of a fixed points $p$, some non-recurrent points
(in $\gamma$ and $\sigma$) and points (in $\delta$) whose orbit meet the interior of $D$, hence the aperiodic measure $\mu$
gives also zero mass to $\partial D$. We thus have $\mu(D)=0$.

Properties (1) and (2) are obtained in three steps that will be detailed below:
\begin{itemize}
\item[--] We give a criterion for a Pixton disc to have its interior disjoint from the orbit of $y$ (see Claim~\ref{c.back}).
We deduce the Property (1) in Corollary~\ref{c.check1}.
\item[--] We assume (by contradiction) that all the ergodic measures on $K$ are supported on periodic orbits:
we build an infinite sequence of such periodic points $(q_n)$ with unstable branches
$\Gamma_n$ accumulating on $\Gamma_{n+1}$ (see Corollary~\ref{c.existence} and Claim~\ref{prop-relation}).
We may choose $(\Gamma_n)$ such that the period of $\Gamma_n$ is minimal for each $n$.
\item[--] In the case where $(\Gamma_n)$ contains a cycle, we build a Pixton disc associated to $q_n$ which intersects $K$; this contradicts the criterion.
In the case where $(\Gamma_n)$ does not contain cycles,
the limit set of all the branches $\Gamma_n$ contains the unstable branch $\widetilde \Gamma$ of some periodic point $\widetilde q \in K$.
This contradicts the minimality of the periods of the branches in the sequence $(\Gamma_n)$.
In all cases we have obtained a contradiction, hence Property (2) holds.
\end{itemize}
\medskip

\noindent
{\bf Step 1. Proof of Property (1).} We first establish the mentioned criterion.}
\begin{claim}\label{c.back}
Let $D'$ be a Pixton disc for an iterate $g=f^m$, $m\geq 1$:
\begin{itemize}
\item[--] associated to an unstable branch $\Gamma_{D'}$ contained in the accumulation set of $\Gamma(p_i)$,
\item[--] whose closing arc $\delta\subset \partial D'$ does not meet both components of $\mathbb{D}\setminus W^s_\mathbb{D}(p_i)$.
\end{itemize}
Then the interior of $D'$ does not intersect the orbit of $\color{black} y$.
\end{claim}
\begin{proof}
Let $q$ be the periodic point associated to $D'$, which is fixed by $g$
and let $\gamma\cup\sigma\cup\delta$ be the boundary of $D'$.
Let $\{x_q\}=\delta\cap \sigma$.
If $f^{-k}(y)$ belongs to the interior of $D'$, {\color{black} the mild dissipation implies that} the stable manifold $W^s_\mathbb{D}(f^{-k}(y))$ intersects the interior of $D'$
and its complement.

Hence $\delta\cup \gamma$ intersects both components
of $\mathbb{D}\setminus W^s_\mathbb{D}(f^{-k}(y))$.
Note that $\gamma\subset \Gamma_{D'}$ does not  intersect both of these components:
since $\Gamma(p_i)$ accumulates on $\Gamma_{D'}$,
it would imply that $\Gamma(p_i)$ does also and then by iteration that
$\Gamma(p_i)$ intersects $W^s_\mathbb{D}(y){\color{black} =W^s_\mathbb{D}(p_i)}$ contradicting our {\color{black} assumption (c) above}.
As a consequence $\delta\setminus \{x_q\}$ intersects both components of
$\mathbb{D}\setminus W^s_\mathbb{D}(f^{-k}(y))$.
Since $g(\delta\setminus \{x_q\})=f^m(\delta\setminus \{x_q\})$ is contained in the interior of $D'$,
one deduces that $W^s_\mathbb{D}(f^{-k+m}(y))$ intersects the interior of $D'$.

By induction we get that $W^s_\mathbb{D}(f^{-k+\ell m}(y))$
intersects the interior of $D'$ for any $\ell\geq 0$ and that $\delta\setminus\{x_q\}$ intersects both components of $\mathbb{D}\setminus W^s_\mathbb{D}(f^{-k+\ell m}(y))$.
But for $\ell$ large $W^s_\mathbb{D}(f^{-k+\ell m}(y))=W^s_\mathbb{D}(p_i)$.
One deduces that
$\delta\setminus \{x_q\}$ intersects both components of $\mathbb{D}\setminus W^s_\mathbb{D}(p_i)$,
a contradiction with our assumptions.
\end{proof}

{\color{black}
\begin{corollary}\label{c.check1}
Pixton discs associated to $\Gamma(p)$ have interior disjoint from the orbit of~$y$.
\end{corollary}}
\begin{proof}
{\color{black} By definition,  if $D$ is a Pixton disc, its boundary decomposes as $\partial D:=\gamma\cup\sigma\cup\delta$
and one denotes by $x$ the intersection point between $\delta$ and $\sigma$.}
Replacing $D$ by $f(D)$ if necessary, one can suppose that
$f^{-1}(x)$ belongs to $W^s_\mathbb{D}(p)$.

Let $\delta_1$ be the connected component of $\delta\setminus f(\gamma)$
which contains $\{x\}$ and let $\gamma_1$ be the arc in $f(\gamma)$ connecting $\gamma$
to $\delta_1$. The disc bounded by $\sigma\cup\delta_1\cup \gamma_1$ is a Pixton disc
which contains $f(D)$. One can repeat that construction and define for each $n\geq 1$
a Pixton disc $D_n$ containing $f^n(D)$ with a closing arc $\delta_n\subset \delta$.
Note that $f^n(\gamma)\subset \partial f^n(D)$, has points arbitrarily close to $f^{-1}(x)$
as $n$ gets large. One can thus connect $f^n(\gamma)$ to $f^{-1}(x)$ by an arc $\delta'$ with small
diameter and build a Pixton disc $D'$ bounded by $\sigma'=f^{-1}(\sigma)$,
$\delta'$ and an arc $\gamma'\subset f^n(\gamma)$. By construction the disc $D'$ contains $f^n(D)$.

{\color{black} Since $\delta'$ can be chosen arbitrarily small, it does not meet both components of $\mathbb{D}\setminus W^s_\mathbb{D}(p_i)$.
Claim~\ref{c.back} can be applied and implies that the interior of $D'$ does not intersect the orbit of $y$.
Consequently the interior of $f^n(D)$ does not intersect the orbit of $y$.
The same holds for $D$, as required.}
\end{proof}


\bigskip

\noindent
{\bf Step 2. Chain of unstable branches.} Property (2) is obtained by contradiction:
\smallskip

\noindent
\textbf{We now assume that
ergodic measures on $K$ are supported on periodic orbits.}

\medskip

In this step we build a sequence of unstable branches of periodic points $q_n\in K$.


\paragraph{\it Quadrants.}
A quadrant of a periodic point $q$ is a pair $(\Gamma,\mathcal{L})$
of unstable and stable branches of $q$.
A sequence \emph{$(y(k))$ converges to $q$ in the quadrant $(\Gamma,\mathcal{L})$} if $y(k)\to q$ and if
for any neighborhood $U$ of the orbit of $q$, there exist $n'_k<0<n''_k$ satisfying:
\begin{itemize}
\item[--] for each $k$ large the piece of orbit $\{f^{n'_k}(y(k)),\dots, f^{n''_k}(y(k))\}$ is contained in $U$,
\item[--] $(f^{n'_k}(y(k)))$ converges to some point of $\mathcal{L}$ and $(f^{n''_k}(y(k)))$ to some point of $\Gamma$.
\end{itemize}

\paragraph{\it Quadruples.} We consider quadruples $(q,\Gamma,\mathcal{L},(y(k)))$ where:
\begin{itemize}
\item[--] $q\in K$ is a periodic point and $(\Gamma,\mathcal{L})$ is a quadrant of $q$,
\item[--] $(y(k))$ is a sequence of backward iterates $f^{-n_k}(y)$ of $y$ converging to $q$ in $(\Gamma,\mathcal{L})$.
\end{itemize}

\paragraph{\it A relation between quadruples.}
Let us write $(q,\Gamma,\mathcal{L},(y(k))) \prec (q',\Gamma',\mathcal{L}',(y'(\ell)))$ when:
\begin{itemize}
\item[--] for each $\ell$, there exist $n_\ell,k_\ell >0$ such that $f^{-n_\ell}(y'(\ell))=y(k_\ell)$,
\item[--] each neighborhood of the closure of $\cup_n f^n(\Gamma)$ contains
the pieces of orbit of the form $\{f^{-n_\ell}(y'(\ell)),\dots,y'(\ell)\}$ with $\ell$ large enough,
\item[--] $\Gamma$ contains a sequence converging to $q'$  in the quadrant $(\Gamma',\mathcal{L}')$.
\end{itemize}
We will say that a quadruple $(q,\Gamma,\mathcal{L},(y(k)))$ satisfies \emph{Property (*)} if the following holds:
\begin{description}
\item[(*)] There exists $z\in \Gamma$ such that, denoting $\gamma$ the arc in $\{q\}\cup\Gamma$ joining $q$ to $z$,
we have:
\end{description}
\begin{itemize}
\item[--] for any neighborhood $U$ of $\gamma$, there exist sequences $(n_\ell),(k_\ell)>0$ such that
$\{y(k_\ell),\dots, f^{n_\ell}(y(k_\ell))\}\subset U$ and $f^{n_\ell}(y(k_\ell))\to z$;
\item[--] if the orbit of $z$ meets a stable branch $\mathcal{L}'$, then for any tubular neighborhood $T$
of $\mathcal{L}'$, the unstable branch $\Gamma$ intersects both connected components of $T\setminus \mathcal{L}'$.
\end{itemize}


\begin{claim}\label{c.transitive}
The relation $\prec $ on quadruples is transitive.
\end{claim}
\begin{proof}
Let us consider three quadrants satisfying
$(q,\Gamma,\mathcal{L},(y(k))) \prec (q',\Gamma',\mathcal{L}',(y'(\ell)))$
and $(q',\Gamma',\mathcal{L'},(y'(\ell))) \prec (q'',\Gamma'',\mathcal{L}'',(y''(m)))$.
Applying successively the definition of the relation $\prec$, one finds preimages
$f^{-n'_m}(y''(m))$ of the points $y''(m)$
that belong to a subsequence of $(y'(\ell))$, and then further backwards iterates $f^{-n_m}(y''(m))$ that belong to a subsequence
of $(y(k))$. The first item of the definition of $\prec$ is thus satisfied between
$(q,\Gamma,\mathcal{L},(y(k)))$ and $(q'',\Gamma'',\mathcal{L}'',(y''(m)))$.

The pieces of orbit $\{f^{-n_m}(y''(m)),\dots,y''(m)\}$ (for $m$ large) are contained in the union of
small neighborhoods of the closures of $\cup_n f^n(\Gamma)$ and $\cup_n f^n(\Gamma')$.
Since $\cup_n f^n(\Gamma')$ is contained in the closure of $\cup_n f^n(\Gamma)$, one deduces
the second item of the definition of $\prec$ is satisfied.

Since $\cup_n f^n(\Gamma')$ contains a sequence of points converging to $q''$  in the quadrant $(\Gamma'',\mathcal{L}'')$,
the same property holds for $\cup_n f^n(\Gamma)$, giving the third item.
We have thus proved that $(q,\Gamma,\mathcal{L},(y(k))) \prec (q'',\Gamma'',\mathcal{L}'',(y''(m)))$.
\end{proof}

\begin{claim}\label{c.sequence}
For any periodic point $q\in K$, the unstable set $W^u(q)$ is disjoint from the orbit of $y$.
Moreover, $q$ admits an unstable branch which intersects $K$.
%
%
\end{claim}
\begin{proof}
Let us assume by contradiction that $y$ belongs to an unstable branch $\Gamma$ of $q$.
%
then $y$ belongs to an unstable branch $\Gamma$ of $q$.
Two cases occur.

We first assume that $\Gamma$ is fixed. Since $y\in \Gamma$ is accumulated by $\Gamma(p_i)$,
the branch $\Gamma$ is accumulated by $\Gamma(p_i)$. Since $\Gamma$ meets $W^s(p_i)$,
it accumulates on $\Gamma(p_i)$.
Since the periodic cycle $\Gamma(p_0),\dots,\Gamma(p_n)$ has maximal length
(assumption (b) made at the begining of the proof of Lemma~\ref{l.pixton-cycle}) the branch $\Gamma$ coincides with one of the $\Gamma(p_j)$.
This contradicts our assumption (c) since $y\in \Gamma\cap W^s_\mathbb{D}(p_i)$
but $W^s_\mathbb{D}(p_i)$ is disjoint from all the $\Gamma(p_j)$,
by our choice of $p_i$.

We then assume that $\Gamma$ has larger period.
From assumption (a), $f$ preserves the orientation.
Proposition~\ref{p.heteroclinic}
implies that $\Gamma$ intersects both components of $\mathbb{D}\setminus W^s_\mathbb{D}(p_i)$.
Since the accumulation set of $\Gamma(p_i)$ contains $\Gamma$,
one deduces that $\Gamma(p_i)$
intersects both components of $\mathbb{D}\setminus W^s_\mathbb{D}(p_i)$,
which contradicts the fact that it is disjoint from $W^s_\mathbb{D}(p_i)$.

In all cases we found a contradiction. This implies the first part of the claim.
Since $q$ belongs to $\alpha(y)$, there exists an unstable branch $\Gamma$
and some point $z$ in the closure of the backward orbit of $y$.
From the first part, $z$ is not on the orbit of $y$, hence is limit of a sequence of arbitrarily large backward iterates of $y$.
Hence $z\in K$.
%
%
%
%
%
%
\end{proof}

\begin{corollary}\label{c.existence}
There exists a quadruple $(q,\Gamma,\mathcal{L},(y(k)))$ satisfying Property (*).
\end{corollary}
\begin{proof} By our assumption at the beginning of Step 2, $K$ contains a periodic point $q_0$.
From Claim~\ref{c.sequence}, $q_0$ belongs to a quadruple $(q_0,\Gamma_0,\mathcal{L}_0,(y_0(k)))$.
If it does not satisfy (*), there exists a point $z\in \Gamma_0\cap K$ which belongs to a stable branch $\mathcal{L}_1$
such that $\Gamma_0$ does not meet both components $T_1\setminus \mathcal{L}_1$
for any tubular neighborhood $T_1$ of $\mathcal{L}_1$.
The branch $\Gamma_1$ belongs to a quadruple $(q_1,\Gamma_1,\mathcal{L}_1,(y_1(k)))$.
Proposition~\ref{p.heteroclinic2} then implies that $\Gamma_0$ and $\Gamma_1$ have the same period.
If one assumes by contradiction that no quadruple satisfies (*), one builds in this way in infinite sequence of quadruples
$(q_n,\Gamma_n,\mathcal{L}_n,(y_n(k)))$ such that $\Gamma_n$ intersects $\mathcal{L}_{n+1}$ for each $n$,
and such that all the $\Gamma_n$ have the same period.
In particular, a same unstable branch appears at least twice in this sequence.
We have thus obtained a cycle of unstable branches $\Gamma_n,\Gamma_{n+1},\dots,\Gamma_{n'}=\Gamma_n$
such that $\Gamma_{m}$ intersects $\Gamma_{m+1}$ for each $n\leq m<n'$.
Since the topological entropy vanishes, this contradicts Lemma~\ref{l.heteroclinic-cycle2}.
\end{proof}

\begin{claim}\label{prop-relation}
For any quadruple $(q,\Gamma,\mathcal{L},(y(k)))$ satisfying Property (*),
\begin{itemize}
\item[--] there exists $(q',\Gamma',\mathcal{L}',(y'(k)))$ such that $(q,\Gamma,\mathcal{L},(y(k))) \prec (q',\Gamma',\mathcal{L}',(y'(k)))$;
\item[--] given any $(q_1,\Gamma_1,\mathcal{L}_1,(y_1(k)))$ satisfying $(q,\Gamma,\mathcal{L},(y(k))) \prec (q_1,\Gamma_1,\mathcal{L}_1,(y_1(k)))$, there exists $(q_2,\Gamma_2,\mathcal{L}_2,(y_2(k)))$ satisfying $(q,\Gamma,\mathcal{L},(y(k))) \prec (q_2,\Gamma_2,\mathcal{L}_2,(y_2(k)))$ and Property (*) such that
$\Gamma_1$, $\Gamma_2$ have the same period.
\end{itemize}
\end{claim}
\begin{proof}
From Property (*), there exist $z\in \Gamma\cap K$ and sequences $(n_\ell),(k_\ell)$ such that
$f^{n_\ell}(y(k_\ell))\to z$ and $\{y(k_\ell), f(y(k_\ell)),\dots, f^{n_\ell}(y(k_\ell))\}$ is contained in an arbitrarily small neighborhood of $\cup_n f^n(\Gamma')$. By our assumption, the $\omega$-limit set of $z$ contains a periodic point $q'$.
One deduces that, up to taking a subsequence of $(y(k_\ell))$, there exist integers $m_\ell$
such that $y'(\ell):=f^{m_\ell}(y(k_\ell))$ converges to $q'$ inside a quadrant $(\Gamma',\mathcal{L}')$ of $q'$,
whereas $\{y(k_\ell), f(y(k_\ell)),\dots, f^{m_\ell}(y(k_\ell))\}$ is included in an arbitrarily small neighborhood of the closure of $\cup_n f^n(\Gamma')$
for $\ell$ large enough.
This shows that $(q,\Gamma,\mathcal{L},(y(k)))$ and $(q',\Gamma',\mathcal{L}',(y'(\ell)))$
satisfy the two first items in the definition of the relation $\prec$.

If the orbit of $z$ does not meet $\mathcal{L}'$, then a subsequence of the forward orbit of $z$
converges to $q'$ in the quadrant $(\Gamma',\mathcal{L}')$ so that $(q,\Gamma,\mathcal{L},(y(k))) \prec (q',\Gamma',\mathcal{L}',(y'(\ell)))$.
If the orbit of $z$ meets $\mathcal{L}'$, then Property (*)
directly implies that $\Gamma$ meets both components of $T\setminus \mathcal{L}'$
for any tubular neighborhood of $\mathcal{L}'$, implying the same conclusion. This proves the first property.
\medskip

We now turn to the second property of the claim and assume $(q,\Gamma,\mathcal{L},(y(k))) \prec (q_1,\Gamma_1,\mathcal{L}_1,(y_1(\ell)))$
for a quadruple $(q_1,\Gamma_1,\mathcal{L}_1,(y_1(\ell)))$ which does not satisfies Property (*).
In particular, up to taking a subsequence of $(y_1(\ell))$,
there is a point $z\in \Gamma_1\cap K$, which is the limit of a sequence of forward iterates
of $(y_1(\ell))$,
and whose orbit belongs to the stable branch $\mathcal{L}_2$ of a point
$q_2$. There also exists a small tubular neighborhood $T$ of $\mathcal{L}_2$ such that,
denoting $T^+,T^-$ the two connected components of $T\setminus \mathcal{L}_2$,
the branch $\Gamma_1$ intersects $T^-$, but it avoids $T^+$.
By Proposition~\ref{p.heteroclinic2}, one deduces that $\Gamma_1$ and $\Gamma_2$
have the same period $\tau$.
One can also extract a subsequence $(y_2(m))$ of forward iterates of $(y_1(\ell))$
and get a quadruple $(q_2,\Gamma_2,\mathcal{L}_2,(y_2(m)))$
such that $(q,\Gamma,\mathcal{L},(y(k)))$ and $(q_2,\Gamma_2,\mathcal{L}_2,(y_2(m)))$
satisfy the two first items in the definition of the relation $\prec$.

By construction the point $z$ is the limit of a sequence $f^{-n_m}(y_2(m))$ of backward iterates of the points $y_2(m)$.
We distinguish two cases:
\smallskip

\noindent
\emph{Case 1:} the points $f^{-n_m}(y_2(m))$ belong to $T^-$.
Note that $\Gamma$ also intersects $T^-$ since its closure contains $\Gamma_1$ which intersects $T^-$.
Hence $(q,\Gamma,\mathcal{L},(y(k))) \prec (q_2,\Gamma_2,\mathcal{L}_2,(y_2(m)))$.

\smallskip

\noindent
\emph{Case 2:} the points $f^{-n_m}(y_2(m))$ belong to $T^+$.
We consider the two arcs
$\gamma\subset \Gamma_1$, $\lambda\subset \mathcal{L}_2$ connecting $z$ to $f^\tau(z)$.
The arc $\lambda$ and the points $f^{n_m}(y_2(\ell_m))$ are contained on the same side of $\Gamma_1$.
By the definition of the quadrant $(q_1,\Gamma_1,\mathcal{L}_1,(y_1(\ell)))$,
the branch $\Gamma$ contains a sequence of points which accumulates on $\gamma\subset \Gamma_1$,
on the same side of $\Gamma_1$ as the points $f^{n_m}(y_2(\ell_m))$.
Since $\Gamma$ is disjoint from $\Gamma_1$, it has to cross the arc $\sigma\subset \mathcal{L}_2$.
See Figure~\ref{f.heredity}.
As a consequence $\Gamma$ accumulates on $q_2$ inside the quadrant $(\Gamma_2,\mathcal{L}_2)$.
Hence $(q,\Gamma,\mathcal{L},(y(k))) \prec (q_2,\Gamma_2,\mathcal{L}_2,(y_2(m)))$.
\smallskip
\begin{figure}[ht]
\begin{center}
\includegraphics[scale=0.5]{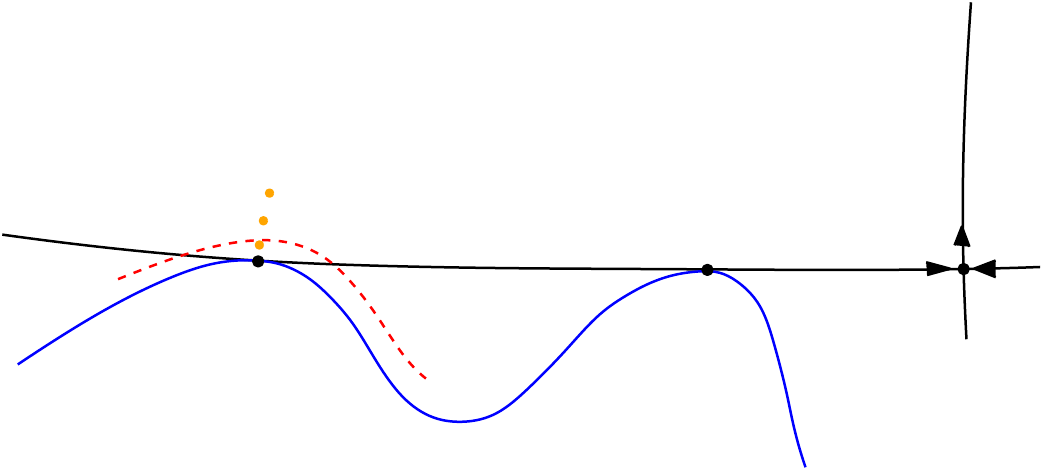}
\begin{picture}(0,0)
\put(-200,70){\scriptsize $f^{-n_m}(y_2(m))$}
\put(-120,50){$\sigma$}
\put(-127,12){$\gamma$}
\put(-200,41){\small $z$}
\put(-98,35){\small $f^\tau(z)$}
\put(-20,54){$q_2$}
\put(-50,54){$T^+$}
\put(-50,34){$T^-$}
\put(-20,100){$\Gamma_2$}
\put(-247,17){$\Gamma_1$}
\put(-235,40){$\Gamma$}
\put(-250,59){$\mathcal{L}_2$}
\end{picture}
\end{center}
\vspace{-0.5cm}
\caption{Proof of Claim~\ref{prop-relation}.  \label{f.heredity}}
\end{figure}

We have shown that $(q,\Gamma,\mathcal{L},(y(k))) \prec (q_2,\Gamma_2,\mathcal{L}_2,(y_2(m)))$
and that $\Gamma_1,\Gamma_2$ have the same period.
If one assumes by contradiction that the second item of the claim does not hold, then
$(q_2,\Gamma_2,\mathcal{L}_2,(y_2(m)))$ does not satisfy Property (*).
The previous construction made for $(q_1,\Gamma_1,\mathcal{L}_1,(y_1(m)))$
applies for this new quadruple. Arguing inductively,
one builds an infinite sequence of quadruples $(q_n,\Gamma_n,\mathcal{L}_n,(y_n(m)))$
such that $\Gamma_n$ intersects $\mathcal{L}_{n+1}$ for each $n$
and such that all the $\Gamma_n$ have the same period.
One concludes as in the proof of Corollary~\ref{c.existence} and get a contradiction.
This gives the second item.
\end{proof}

\bigskip

\noindent
{\bf Step 3. Conclusion of the proof of Lemma~\ref{l.pixton-cycle}.}
By Corollary~\ref{c.existence} and Claim~\ref{prop-relation}, there exists an infinite sequence of quadruples
$(q_n,\Gamma_n,\mathcal{L}_n,(y_n(k)))$ satisfying for each $n$ the relation
$(q_n,\Gamma_n,\mathcal{L}_n,(y_n(k))) \prec (q_{n+1},\Gamma_{n+1},\mathcal{L}_{n+1},(y_{n+1}(k)))$.

By the second item of Claim~\ref{prop-relation}, one can furthermore assume the following property:
\begin{description}
\item[(**)] For each $n$, if $(q_n,\Gamma_n,\mathcal{L}_n,(y_n(k)))\prec (q,\Gamma,\mathcal{L},(y(k)))$ holds for some
quadruple $(q,\Gamma,\mathcal{L},(y(k)))$ then the period of $\Gamma$ is larger or equal to the period of $\Gamma_{n+1}$.
\end{description}
By the transitivity in Claim~\ref{c.transitive}, the sequence of periods of the branches $\Gamma_n$ is non-decreasing with $n$.
In order to conclude the proof, we consider two different cases.

\paragraph{\it First case: the period of the branches $\Gamma_n$ is uniformly bounded.}
There exist at most finitely many periodic points in $K$ with a given period,
hence a same quadruple appears several times in the sequence.
By Claim~\ref{c.transitive}, there exists a quadruple satisfying
$(q_n,\Gamma_n,\mathcal{L}_n,(y_n(k)))\prec (q_n,\Gamma_n,\mathcal{L}_n,(y_n(k)))$.
Let $f^m$ be an iterate which fixes the branch $\Gamma_n$.
Since $q_n$ is accumulated by points of $\Gamma_n$ in the quadrant $(\Gamma_n,\mathcal{L}_n)$,
one can build a Pixton disc $D'$ for $f^m$ bounded by an arc $\gamma\subset \Gamma_n$,
an arc $\sigma\subset\mathcal{L}_n$ and a closing arc $\delta$ whose diameter is arbitrarily small,
so that it is is disjoint from one of the components of $\mathbb{D}\setminus W^s_\mathbb{D}(p_i)$.
Since $q_n$ is accumulated by the sequence $(y_n(k))_{k\in \NN}$ in the quadrant $(\Gamma_n,\mathcal{L}_n)$,
the interior of $D'$ contains iterates of $y$.
This contradicts the Claim~\ref{c.back}.

\paragraph{\it Second case: the period of the branches $\Gamma_n$ goes to $+\infty$
as $n\to +\infty$.}
There exists a subsequence of the family of periodic orbits of the points $q_n$ which converges toward an invariant
compact set $\Lambda\subset K$ for the Hausdorff topology.
By our assumption, there exists a periodic point $\widetilde q\in \Lambda$ which is the limit of a subsequence of $(q_n)$.
Let us take $n_0$ such that the period of $\Gamma_{n_0}$ is larger than twice the period of $\widetilde q$.
Up to taking a subsequence, there exists a quadrant $(\widetilde \Gamma,\widetilde{\mathcal{L}})$ of $\widetilde q$
such that $(q_n)$ converges to $\widetilde q$ in $(\widetilde \Gamma,\widetilde{\mathcal{L}})$.

Since $\prec$ is transitive,
we have $(q_{n_0},\Gamma_{n_0},\mathcal{L}_{n_0},(y_{n_0}(k))) \prec (q_{n},\Gamma_{n},\mathcal{L}_{n},(y_{n}(k)))$ for $n>n_0$.
Hence there exist sequences $k_\ell, n_\ell\to +\infty$ such that
$\widetilde y(\ell):=(f^{n_\ell}(y_{n_0}(k_\ell)))$ converges to $\widetilde q$ in $(\widetilde \Gamma,\widetilde{\mathcal{L}})$.
This implies $(q_{n_0},\Gamma_{n_0},\mathcal{L}_{n_0},(y_{n_0}(k))) \prec (\widetilde q,\widetilde \Gamma,\widetilde {\mathcal{L}},(\widetilde y_\ell(k)))$. From Property (**), the period of $\widetilde \Gamma$ is larger or equal to the period of $\Gamma_{n_0}$.
But this contradicts the choice of $n_0$: the period of $q_{n_0}$ is larger than twice the period of $\widetilde q$.
\medskip

In both cases we found a contradiction.
Hence the assumption made at the beginning of Step 2 does not hold.
This proves that $K$ supports an aperiodic ergodic measure.
The proof of Lemma~\ref{l.pixton-cycle} is now complete.
\end{proof}

\subsection{Proof of Theorems~\ref{t.cycle} and~\ref{t.cycle2}}
We first prove Theorem~\ref{t.cycle}.
It is well known that if $f$ has positive entropy, then it admits horseshoes~\cite{Ka}
and in particular a cycle of periodic orbits.

Conversely, let us assume that $f$ has a cycle of periodic orbits.
Up to {\color{black} replacing} $f$ by an iterate, one can suppose (by Lemma~\ref{l.cycle-fixed}) that $f$ has a fixed unstable branch $\Gamma$
which is contained in its accumulation set. Lemma~\ref{l.pixton-cycle} then gives an aperiodic ergodic measure $\mu$
supported on the closure of $\Gamma$ such that $\mu(D)=0$ for any Pixton disc $D$ associated to $\Gamma$.
Let us denote by $\mathbb{D}_\Gamma$ the closure of the connected component of $\mathbb{D}\setminus W^s_\mathbb{D}(p)$
which contains $\Gamma$. By our assumption,
the support of $\mu$ is contained in $\mathbb{D}_\Gamma$.

\begin{lemma}\label{l.support}
There exists a neighborhood $U$ of $p$ such that $\mu(U)=0$.
In particular, the support of $\mu$ is disjoint from $\Gamma$ and $W^s_\mathbb{D}(p)$.
\end{lemma}
\begin{proof}
The measure $\mu$ is supported
inside $\mathbb{D}_\Gamma$. Moreover we have $\mu(p)=0$ since $\mu$ is not atomic.
Hence if one assumes that any neighborhood of $p$ has positive $\mu$-measure,
there exists some point $x\neq p$ in $W^s_\mathbb{D}(p)$
 which belongs to the support of $\mu$. One deduces that $x$ is accumulated by $\Gamma$.
One can thus build a Pixton disc $D$ by closing near $f^{-1}(x)$: the disc contains a neighborhood of
$x$ in $\mathbb{D}_\Gamma$, hence has positive measure. This contradicts Lemma~\ref{l.pixton-cycle}.
\end{proof}

Let $W^{s,+}_\mathbb{D}(p)$ be one of the components of $W^s_\mathbb{D}(p)\setminus \{p\}$ which
contains points accumulated by $\Gamma$.
Let $\Gamma_{loc}$ be a local unstable manifold of $p$, \emph{i.e.,} a neighborhood of $p$ inside $\Gamma$
for the intrinsic topology. It separates small neighborhoods of $p$ in $\mathbb{D}_\Gamma$ into two components:
we denote by $U^+$ the component which meets $W^{s,+}_\mathbb{D}(p)$. See Figure~\ref{f.quadrant}.

\begin{figure}[ht]
\begin{center}
\includegraphics[width=5cm]{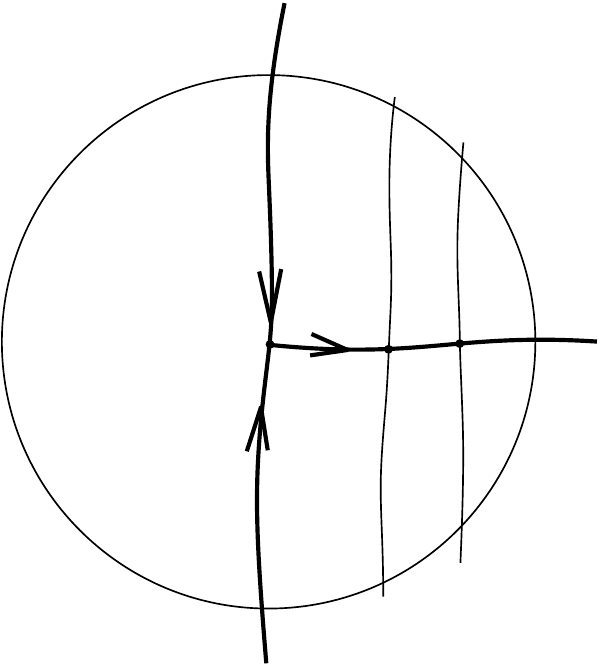}
\begin{picture}(0,0)
\put(-135,20){$U$}
\put(-92,70){$p$}
\put(-121,147){$W^{s,+}(p)$}
\put(-10,65){$\Gamma_{loc}$}
\put(-50,80){$U^+$}
\put(-40,20){$W_{n-1}$}
\put(-70,5){$W_n$}
\end{picture}
\end{center}
\caption{Quadrant separated by $\Gamma$ and $W^{s,+}(p)$.\label{f.quadrant}}
\end{figure}

Note that $\mu$-almost every point $x$ is accumulated by its orbit inside each component of
$\mathbb{D}\setminus W^s_\mathbb{D}(x)$. In particular $\Gamma$ meets these two components and
intersects $W^s_\mathbb{D}(x)$ at some point $z$.
Iterating backward $W^s_\mathbb{D}(z)$, one thus gets a sequence of stable curves $W_n\subset \DD^+$ such that $f(W_n)\subset W_{n-1}$,
$f^n(W_n)\subset W^s_\mathbb{D}(z)$, which converge to $W^s_\mathbb{D}(p)$ for the Haudsdorff topology.
We denote by $W_n^+$ a connected component of $W_n\setminus W^s_\mathbb{D}(x)$ which is close to $W^{s,+}_\mathbb{D}(p)$
for the Hausdorff topology. By choosing $n$ large enough, $W^+_n$ separates $W^+_{n-1}$ and $W^{s,+}_\mathbb{D}(p)$ in $U^+$. See Figure~\ref{f.quadrant}.
\bigskip

Let $x^s\in W^{s,+}_\mathbb{D}(p)$ be a point that is not accumulated by $\Gamma$ and let $\beta^s$
be a small $C^1$ arc transverse to $W^{s,+}_\mathbb{D}(p)$ at $x^s$.
We also choose $x^u\in \Gamma_{loc}$ and a small $C^1$ arc $\beta^u$ transverse to $x^u$ at $\Gamma$.
For $m\geq 1$ large, the arcs $f^{-m}(\beta^u)$, $\beta^s$, $f(\beta^s)$ and $W^{s,+}_\mathbb{D}(p)$ bound a rectangle $R$.
Similarly, the arcs $f^{-1}(\beta^u)$, $\beta^u$, $f^m(\beta^s)$ and $\Gamma_{loc}$ bound a rectangle $R'$.
We may choose $W^+_n$ and $W^+_{n-1}$ to separate $p$ from $\beta$ and $m$ large enough.
One thus get the following properties: 

\begin{itemize}
\item[(a)] $R'$ is separated from $R$ by $W^+_{n-1}$ in $U^+$,
\item[(b)] Any point in $R\setminus W^s_\mathbb{D}(p)$ has a forward iterate in $R'$.
\end{itemize}
Note that the forward iterates of $W_n$ and $W_{n-1}$ accumulate on the support of $\mu$.
As a consequence of Lemma~\ref{l.support}, if $R$ has been chosen small enough, we get:
\begin{itemize}
\item[(c)] The forward iterates of $W_n$ and $W_{n-1}$ do not meet $R$.
\end{itemize}
Let $D$ be a Pixton disc associated to $\Gamma$, whose
boundary is the union of three arcs:
$\sigma\subset W^{s,+}(p)$, $\gamma\subset \Gamma$ and a closing arc $\delta$.
We chose $D$ so that $\delta$ is contained in $R$
and $\gamma$ intersects $R$ in only one point.
Since $f^{-1}(\gamma)\subset \gamma$, this implies that the arc $\gamma$ is disjoint from the forward iterates of $R$.
In particular,
\begin{itemize}
\item[(d)] $R'$ is contained in $D$.
\end{itemize}

Let us consider the two curves $\alpha'\subset W_n^+$ and $\alpha\subset W_{n-1}^+$, contained in $U^+\cap D$
which connect $\Gamma_{loc}$ to another point of the boundary of $D$
(and intersecting the boundary of $D$ only at their endpoints, which by construction belong to $\Gamma$).
Note that $f(\alpha)\subset \alpha'$ by definition;
both are contained in $W_{n-1}$.
The curve $\alpha\cup\gamma$ bounds a disc $\Delta$
whereas the curve $\alpha'\cup\gamma$ bounds a disc $\Delta'$.
Since $W^+_n$ separates $W^+_{n-1}$ and $W^{s,+}_\mathbb{D}(p)$, the discs are nested: $\Delta'\subset \Delta$.
See Figure~\ref{f.disc2}.
\begin{figure}[ht]
\begin{center}
\includegraphics[width=10cm]{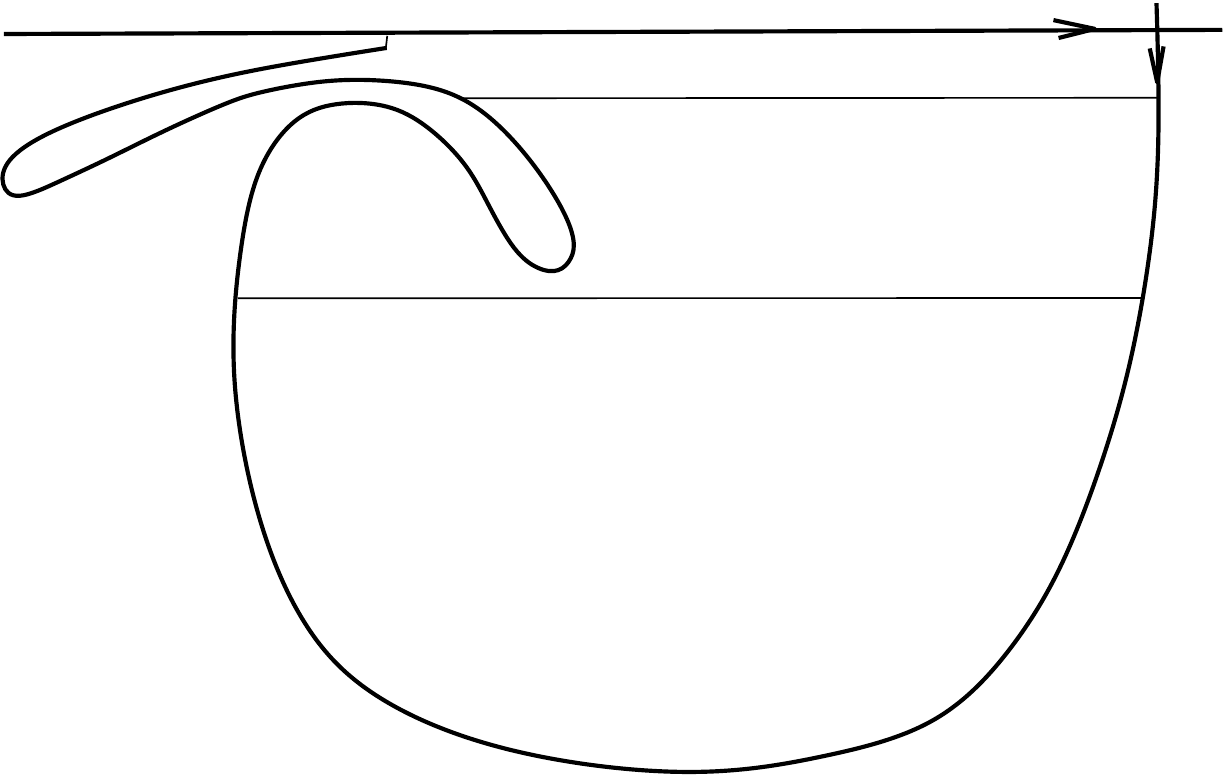}
\begin{picture}(0,0)
\put(-73,95){$\alpha'$}
\put(-100,70){$\Delta'$}
\put(-100,130){$\Delta$}
\put(-100,175){$\sigma$}
\put(-197,175){$\delta$}
\put(-10,160){$p$}
\put(-40,37){$\gamma$}
\put(-70,145){$\alpha$}
\end{picture}
\end{center}
\caption{Proof of Theorem~\ref{t.cycle}.\label{f.disc2}}
\end{figure}

Let $\overset{\circ}\alpha$ and $\overset{\circ}{\alpha}{}'$ denote the arcs $\alpha,\alpha'$ without their endpoints.
Using (c), that $f^{-1}(\gamma)\subset \gamma$ and that $\gamma$ is disjoint from $\overset{\circ}\alpha\cup \overset{\circ}{\alpha}{}'$,
one deduces that the forward iterates of $\overset{\circ}\alpha$, $\overset{\circ}{\alpha}{}'$ do not intersect $\partial D=\sigma\cup\gamma\cup\delta$.
Hence:
\begin{itemize}
\item[(e)] Any forward iterate of $\overset{\circ}\alpha$ or $\overset{\circ}{\alpha}{}'$ is either in the interior of $D$ or disjoint from $D$.
\end{itemize}
\bigskip

For $k$ large, the images $f^k(\alpha)$ and $f^k(\alpha')$ are contained in a small neighborhood of the support of $\mu$, hence are outside $D$.
Since $f(\alpha)\subset \alpha'$, one gets:
$f^{i+1}(\overset{\circ}\alpha)$ is disjoint from $D$ if and only if $f^{i}(\overset{\circ}{\alpha}{}')$ is disjoint from $D$.
Together with (e), one deduces that there exists $k_0$ such that $f^{k_0}(\overset{\circ}{\alpha}{}')$ is disjoint from $D$, $f^{k_0}(\overset{\circ}{\alpha})$ is in the interior of $D$ and all
the larger iterates are disjoint from $D$. This implies the following lemma.

\begin{lemma}\label{l.separated}
There exists $k_0\geq 1$
such that for any curve $\beta\subset \Gamma$ in the interior of $D$ and connecting $\alpha$ to $\alpha'$:
\begin{itemize}
\item[--] $f^{k_0}(\beta)$ meets $\delta$, has one endpoint in the interior of $D$ and another one outside $D$,
\item[--] all the forward iterates of the endpoints of $f^{k_0}(\beta)$ are outside $D$.
\end{itemize}
\end{lemma}
\begin{proof}
Consider the iterate $k_0$ such that $f^{k_0}(\alpha)$ is inside, $f^{k_0}(\alpha')$ is outside and all
larger iterates of $\alpha'$ are outside also. Note that the forward iterates of $\beta\subset \Gamma\setminus \gamma$
do not intersects $\gamma$ nor $\sigma\subset W^s_\mathbb{D}(p)$.
\end{proof}

From a curve $\beta$, one gets two new ones $\beta_1,\beta_2$.

\begin{lemma}\label{l.separation}
There exists $k_1$ and $\varepsilon>0$
such that any curve $\beta$, in the interior of $D$ and connecting $\alpha$ to $\alpha'$,
contains two sub-curves $\beta_1$, $\beta_2$ such that:
\begin{itemize}
\item[] $f^{k_1}(\beta_1)$, $f^{k_1}(\beta_2)$ are $\varepsilon$-separated, contained in $\interior(D)$
and connect $\alpha$ to $\alpha'$.
\end{itemize}
\end{lemma}
\begin{proof}
By Lemma~\ref{l.separated},
the curve $\beta$ contains sub curves $\bar \beta_1, \bar \beta_2$ such that:
\begin{itemize}
\item[--] $\bar \beta_1$ (resp. $\bar \beta_2$) contains an endpoint $b_1$ (resp. $b_2$)
of $\beta$ and a point of $f^{-k_0}(\delta)$,
\item[--] $f^{k_0}(\bar \beta_1)$ is disjoint from the interior of $D$,
\item[--] $f^{k_0}(\bar \beta_2)$ is contained in $D$.
\end{itemize}
From Lemma~\ref{l.separated}, all the forward iterates of $f^{k_0+1}(b_i)$, for $i\in\{1,2\}$
are outside $D$.
From (b), there exists $k\geq k_0$ such that $f^{k}(\bar \beta_1)$ and $f^{k}(\bar \beta_2)$
have a point in the interior of $R'$, hence in the interior of $\Delta'$ (from (d)).
Thus by (a) these curves contain a point of
$\alpha'$ and a point of $\alpha$. Since the iterates of $\beta$
never intersects $\gamma\cup \sigma$, one deduces that for any $k'\geq k$,
$f^{k'}(\bar \beta_1)$ and $f^{k'}(\bar \beta_2)$ still intersect $\alpha$ and $\alpha'$
and in particular contain two curves connecting $\alpha$ to $\alpha'$.

The integer $k$ may depend on $\beta$, but since $f^{k_0}(\beta)$ intersect $\delta\setminus W^s_\mathbb{D}(p)$
in a compact set which does not depend on $\beta$, the integer $k$ is is uniformly bounded.
One can thus find $k'=k_1$ independent from $\beta$ such that
both $f^{k_1}(\bar \beta_1)$ and $f^{k_1}(\bar \beta_2)$ meet $\alpha$ and $\alpha'$.

Let us choose the minimal curve $\hat \beta_1\subset \bar \beta_1$
which connects $b_1$ to $f^{-k_1}(\alpha')$ and the minimal curve $\hat \beta_2\subset \bar \beta_2$
which connects $b_2$ to$ f^{-k_1}(\alpha')$. In particular for any
$k_0\leq k<k_1$, the curves $f^k(\hat \beta_i)$ are disjoint from $\Delta'$.

One then choose $\beta_1\subset \bar \beta_1$ and $\beta_2\subset \bar \beta_2$
such that $f^{k_1}(\beta_1)$ and $f^{k_1}(\beta_2)$ meet $\alpha$ and $\alpha'$ at their endpoint
and nowhere else.

By construction $f^{k_0}(\beta_1)$ is disjoint from $D$ and
$f^{k_0}(\beta_2)$ is contained in the interior of $D$.
They are contained in
two different connected components of
$f^{k_1-k_0}(\Delta\setminus \Delta')\setminus \delta$.
Moreover they avoid a uniform neighborhood of $\tilde \delta:=\delta\cap f^{k_1-k_0}(\Delta\setminus \Delta')$:
indeed there exists $\ell_0$ such that any point $y$ in $\tilde \delta$ has a forward iterate
$f^\ell(y)$ in $R'$ with $\ell\leq \ell_0$. By compactness the same holds for any point in a neighborhood of $\widetilde \delta$.
But by construction for any point in $f^{k_0}(\beta_1)$ and $f^{k_0}(\beta_2)$, the $k_1-k_0-1$ first iterates are disjoint from $R'$.
Hence by choosing $k_1>k_0+\ell$, one can ensure that $f^{k_0}(\beta_1)$ and $f^{k_0}(\beta_2)$ are disjoint from
a uniform neighborhood of $\widetilde \delta$.

After having fixed $k_1$ and having chosen $\varepsilon>0$ small enough, one deduces that the curves $f^{k_1}(\sigma_1)$, $f^{k_1}(\sigma_2)$ are $\varepsilon$-separated for some $\varepsilon>0$ small
as required.
\end{proof}

Note that $\Gamma\setminus \gamma$ contains  an arc that connect a point in $R$ with a point in $R'$:
this shows that there exists a curve $\beta\subset \Gamma$ contained in the interior of $D$
which connects $\alpha$ to $\alpha'$. 
One then apply Lemma~\ref{l.separation} inductively:
it shows that for each $\ell$, the arc $\beta$ contains $2^\ell$ orbits of length $\ell.k_1$ that are $\varepsilon$-separated.
One deduces that the topological entropy of $f$ is larger than $\log(2)/k_1$, hence positive.

The proof of Theorem~\ref{t.cycle} is now complete.
\bigskip

The proof of Theorem~\ref{t.cycle2} is the same, working inside the filtrating domain.
\qed

\section{Generalized Morse-Smale diffeomorphisms}\label{s.gms}

We extend the Definition \ref{d.generalizedMS} to filtrating sets:

\begin{definition}
A diffeomorphism is \emph{generalized Morse-Smale}\index{Morse-Smale and generalized Morse-Smale diffeomorphism} in a filtrating set $U$ if
\begin{itemize}
\item[--] the $\omega$-limit set of any forward orbit in $U$ is a periodic orbit,
\item[--] the $\alpha$-limit set of any backward orbit in $U$ is a periodic orbit,
\item[--] the period of all the periodic orbits contained in $U$
is bounded by some $K>0$.
\end{itemize}
\end{definition}

We also say that
a diffeomorphism is \emph{mildly dissipative} in a filtrating set $U$ if
for any ergodic measure $\mu$ for $f|_U$,
which is not not supported on a hyperbolic sink, and for
$\mu$-almost every $x$, $W^s_U(x)$ separates $U$.

\begin{proposition}\label{p.MS}
Any diffeomorphism of the disc which is
mildly dissipative and generalized Morse-Smale in a filtrating set $U$
has zero topological entropy in $U$.
Moreover the chain-recurrent points in $U$ are all periodic.
\end{proposition}
\begin{proof}
Any ergodic measure of $f|_U$ is supported on a periodic orbit, hence has zero entropy.
The variational principle (that states that the topological entropy is the supremum of the entropy of the invariant measures) concludes that the topological entropy of $f|_U$ is zero.

Up to {\color{black} replacing} $f$ by an iterate, one can suppose that all the periodic points and all the unstable
branches in $U$ are fixed by $f$.
Let us assume by contradiction that there exists a chain-recurrent point $x$ which is not periodic.
One chooses as in Proposition~\ref{p.group} a finite collection of disjoint fixed arcs $\cI$ of $U$.
One can require that they do not contain $x$.
By our assumption, $x$ belongs to an unstable branch of an arc $I_0$, which accumulates on an arc $I_1$.
Since $x$ is chain-recurrent, there exists pseudo-orbits from $I_1$ to $I_0$,
hence there exists an unstable branch of $I_1$ which accumulates on another arc $I_2$
and there exists pseudo-orbits from $I_2$ to $I_0$.
Arguing inductively, one builds a sequence of fixed arcs $I_n$ in $U$ such that the unstable manifold of $I_n$
accumulates on the arc $I_{n+1}$.
Since $\cI$ is finite, this implies that there exists a cycle of arcs in $U$, contradicting Corollary~\ref{c.cycle}.
\end{proof}

\begin{proposition}
The set of diffeomorphisms of the disc which are mildly dissipative and generalized Morse-Smale in a filtrating set $U$
is open for the $C^1$ topology.
\end{proposition}
\begin{proof}
Note that if $U$ is filtrating for $f$, it is still filtrating for diffeomorphisms close.

From Proposition~\ref{p.group}, there exists a finite collection $\cI$ of disjoint arcs in $U$ that are fixed by an iterate
$f^k$ of $f$. By normal hyperbolicity (see~\cite{BoCr}), for any $I\in \cI$, there exists a neighborhood $V_I$
such that for any diffeomorphism $g$ that is $C^1$ close to $f$, any orbit of $g^k$ contained in $V_I$
is contained in a $g^k$-fixed arc contained in $V_I$; such an arc is still normally contracted.
One deduces that any forward (resp. backward) $g^k$-orbit contained in $V_I$ accumulates on a fixed point of $g^k$.

Since $\cI$ is finite and since the neighborhoods $V_I$ of the arcs $I$ may be chosen small,
one gets a neighborhood $V=\cup V_I$ of the set of periodic points of $f$ with the following property:
for any diffeomorphism $g$ that is $C^1$ close to $f$, the $\omega$-limit of any forward orbit of $g$ contained in $V$ is
an orbit of period less or equal to $k$ and the same holds for the $\alpha$-limit of any backward orbit of $g$ contained in $V$.

The chain-recurrent set varies upper semi-continuously with the dynamics.
Hence for any diffeomorphism $g$ that is close to $f$, the $\omega$-limit and the $\alpha$-limit sets of a $g$-orbit contained in $U$
is contained in $V$, and they are periodic orbits of period less or equal to $k$.
This proves that $g$ is generalized Morse-Smale in $U$.
\end{proof}

\section{Stabilization, decoration, structure of periodic points}
\label{decoration section}

In this section $f$ is a mildly dissipative diffeomorphism of the disc with zero entropy.
First we  introduce and discuss two related types of configurations of saddle periodic orbits: the decoration and the stabilization (subSection \ref{s.stab-dec}).
We then describe how the set of fixed points (or points of a given period) are organized through chains (see Sections \ref{s.connectedness-fixed}). 
Later, using the chains, we define a hierarchy between periodic points (Section \ref{s.connectedness-periodic}) and at the end, in Proposition \ref{p.decreasing-chain}, we show that all periodic points are related through this hierarchy.

\subsection{Stabilization and decoration}\label{s.stab-dec}
\begin{definition}\label{d.stabilization}
 A periodic point $p$ is \emph{stabilized by a fixed point $q$}\index{stabilization} if one of the two following cases occurs (see Figure~\ref{f.stabilized}):
\begin{itemize}
\item[--] either $p=q$ is a fixed point, not a sink, and $D_pf$ has an eigenvalue $\lambda^+_p\leq -1$,
\item[--] or $p$ has period larger than $1$ and an unstable branch $\Gamma$ which accumulates on $q$;
{\color{black} in the case where $q$ is not a sink and $D_pf$ has an eigenvalue $\lambda^+_p\leq -1$,
we also require that $\Gamma$ intersects both components of $\mathbb{D}\setminus W^s_{\DD}(q)$.}

\end{itemize}

Sometimes we also say that the orbit of $p$ is a \emph{stabilized periodic orbit} and that $q$ is a \emph{stabilizing point}.
The unstable branch that accumulates on $q$ is called a \emph{stabilizing branch}.
\end{definition}
\begin{figure}
\begin{center}
\includegraphics[width=12cm,angle=0]{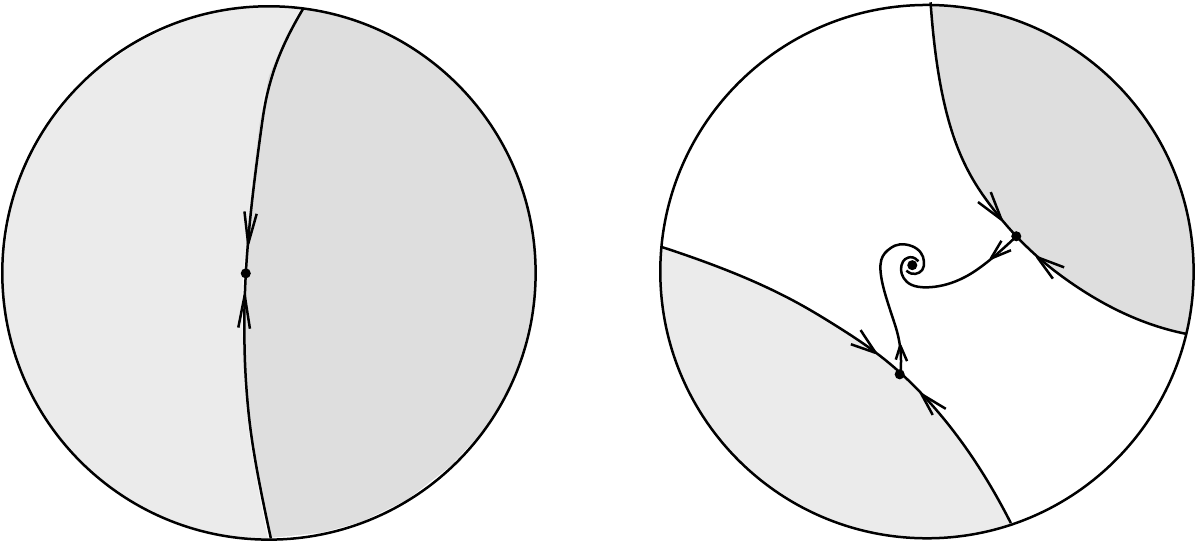}
\put(-50,92){$p$}
\put(-90,95){$q$}
\put(-115,43){$f(p)$}
\put(-260,70){$p=q$}
\end{center}
\caption{Stabilized periodic orbits with period $1$ and $2$.
(Decorated regions in grey.)\label{f.stabilized}}
\end{figure}

\begin{remarks} Let us make a few observations about stabilized and stabilizing points.

\begin{enumerate}
\item The first case can be considered as a degenerate case of the second: as explained in Remark \ref{r.degenerated},
$p$ can be considered as a $2$-periodic point which has collided with the stabilizing fixed  sink $q$.
The stabilizing branches are hidden in $q$ in this case.

\item In the second case, $q$ could be a fixed point of any type: a sink,  indifferent or a  saddle (in that case, it could be either stabilized or not).

\item There may exist several stabilizing points $q$ associated to a stabilized point $p$.

\end{enumerate}
\end{remarks}

We have introduced the notion of \emph{decorated periodic orbit} in Section~\ref{ss.decoration}.

\begin{proposition}[Stabilization implies decoration]\label{p.stab-decorate}
If $f$ is a mildly dissipative diffeomorphism with zero entropy, then any periodic orbit $\cO$ which is stabilized by a fixed point
is decorated.
Each point $p\in \cO$ has at most one stabilizing unstable branch.
\end{proposition}

\begin{proof}
In the particular case where $\cO$ is a fixed point, the statement become trivial,
we will thus assume that $\cO$ has period larger than $1$.

Consider $p\in \cO$ and the connected component $C$ of $\mathbb{D}\setminus W^s_\mathbb{D}(p)$ which does not contain
the stabilizing point $q$. If one assumes that some iterate $f^j(p)$ belongs to $C$, then
the unstable branch of $f^j(p)$ which accumulates on $q$ intersects both components of $\mathbb{D}\setminus W^s_\mathbb{D}(p)$,
hence intersects $W^s(p)$: this implies that $f$ has a cycle of periodic orbit, a contradiction.

We have proved that the orbit of $p$ is contained in the connecting component of $\mathbb{D}\setminus W^s_\mathbb{D}(p)$
which contains the stabilizing branch. In particular $p$ has at most one stabilizing unstable branch.
\end{proof}

\begin{definition} \label{d.decorated region}
When $p$ is stabilized by a fixed point $q$,
the connected component of $\mathbb{D}\setminus W^s_\mathbb{D}(p)$ which does not
contain $q$ is called the \emph{decorated region of $p$}.\index{decoration}
(In the special case where $p=q$ is a fixed point, it admits two decorated regions.)

The \emph{period} of the decorated region is either the period of $p$ (when $p$ is not fixed)
or $2$ (when $p$ is fixed): this is the return time to the decorated region for points close to $p$.
\end{definition}

\begin{proposition}
If $f$ is a mildly dissipative diffeomorphism with zero entropy which reverses the orientation,
then each stabilized orbit has period $1$ or $2$.
\end{proposition}
\begin{proof}
Let us consider a stabilized periodic point $p$ with period $k$.
By definition there exists an unstable branch $\Gamma$ of $p$ which accumulates on a fixed point $q$.
We denote by $K_p$ the accumulation set of $\Gamma$. This set is cellular, fixed by $f^k$ and $f^n(K_p)$ contains $q$ for any $n$.
Hence the set $K:=\cup_n f^n(K_p)$ is a cellular set fixed by $f^k$. The complement $\mathbb{D}\setminus K$
is an invariant annulus. Let us denote by $B=\mathbb{R}\times (0,1)$ the universal cover with the covering
automorphism $(x,t)\mapsto (x+1,t)$. The map $f$ on the annulus lifts as a map $h$ on $B$
which reverses the orientation and satisfies $h(x+1,t)=h(x,t)-(1,0)$.
Let $\gamma$ be the union of $\Gamma$ with only one local stable branch of $p$: it is a proper curve in the annulus which connects
one end to the other one. It lifts in $B$ as a curve $\widehat \gamma_0$
whose complement has two connected components.
Repeating the construction for each iterate of $p$ and considering the translated curves,
one obtains a family of curves $(\widehat \gamma_n)_{n\in \ZZ}$ in $B$
with the properties:
\begin{itemize}
\item[--] $\widehat \gamma_{n+k}=\widehat \gamma_n+(1,0)$,
\item[--] $B\setminus \gamma_n$ has two connected components $U^-_n$ and $U^+_n$
satisfying $U^-_n\subset U^-_m$ when $n\geq m$,
\item[--] there exists a bijection $\tau$ of $\mathbb{Z}$ such that $h(\gamma_n)\subset \gamma_{\tau(n)}$.
\end{itemize}
In particular $\tau$ is monotone. Since $h$ reverses the orientation
there exists $a\in \ZZ$ such that $\tau(n)=-n+a$ for each $n\in\ZZ$.
In particular either $\tau$ has a fixed point or a point of period $2$.
This implies that $p$ is either fixed or has period $2$.
\end{proof}

The previous proposition shows that when $f$ reverses the orientation,
all the decorated regions have period $2$.
\begin{corollary}\label{c.orientation}
If $k$ is the period of a decorated region, then $f^k$ preserves the orientation.
\end{corollary}

\subsection{Structure of the set of fixed points}\label{s.connectedness-fixed}

We introduce a notion which generalizes the fixed arcs.
\begin{definition}\label{d.chain}
Let $p,p'$ be two fixed points.
A \emph{chain} for $f$ between $p$ and $p'$ is a (not necessarily compact) connected set $C$
which is the union of:
\begin{itemize}
\item[--] a set of fixed points $X$ containing $p$ and $p'$,
\item[--] some $f$-invariant unstable branches of points in $X$.
\end{itemize}

\end{definition}

\begin{proposition}\label{p.chain} If $f$ is an orientation preserving mildly dissipative diffeomorphism with zero entropy, 
then between any pair of fixed points $p,p'$ there exists a chain for $f$.
\end{proposition}

The end of this section is devoted to the proof of this proposition. Previous proposition also holds for mildly dissipative diffeomorphisms without cycles.

\begin{lemma}\label{p.accumulation}
If $f$ is an orientation-preserving dissipative diffeomorphism of the disc,
any $f$-invariant unstable branch $\Gamma$ accumulates on a fixed point.
\end{lemma}
\begin{proof}
Let $\gamma\subset \Gamma$ be a curve which is a fundamental domain.
By definition the accumulation set $\Lambda$ is an invariant compact set. Since it is arbitrarily close in the Hausdorff topology to
{\color{black} the closure of the curve $\bigcup_{i\geq i_0} f^i(\gamma)$ (with $i_0$ large)}, the set $\Lambda$ is connected. Since $f$ is dissipative, the complement
$\mathbb{D}\setminus \Lambda$ is connected. One deduces from Proposition \ref{p.CL} that $\Lambda$ contains a fixed point.
\end{proof}

Let us consider the finite set $\cI$ of normally hyperbolic fixed arcs as in Section~\ref{ss.arc} and recall that fixed points can be treated as arcs.

\begin{lemma}\label{l.acc}
For any $f$-invariant unstable branch $\Gamma$ of an arc $I\in \cI$, the accumulation set intersects an arc $I'\in \cI$ of index $0$ or $1$.
\end{lemma}
\begin{proof}
The branch $\Gamma_0:=\Gamma$ is contained in the unstable set of a fixed point $p_0$.
From Lemma~\ref{p.accumulation}, the accumulation set of $\Gamma_0$ contains a fixed point $p_1$.
Let $I_1\in \cI$ be the fixed arc that contains $p_1$.
If $I_1$ has index $0$ or $1$, the lemma holds. Otherwise $I_1$ has the type of a saddle with no {\color{black} reflection}.
Since $p_0\not\in I_1$, the branch $\Gamma_0$ accumulates on  the stable manifold of an endpoint of $I_1$.
One can thus reduce to the case where $p_1$ is an endpoint of $I_1$ and where $I_1$ and $p_1$ have a common unstable branch $\Gamma_1$ which intersects the accumulation set of $\Gamma_0$ (see Proposition \ref{p.transitive}).
One can repeat the previous construction with the arc $I_1$ and the unstable branch $\Gamma_1$.
One build in this way inductively a sequence of arcs $I_n$ with an unstable branch $\Gamma_n$ which accumulate on $I_{n+1}$.
Since the number of arcs in $\cI$ is finite, and there is no cycle (Corollary~\ref{c.cycle}),
this sequence stops with one arc of index $0$ or $1$.
By construction, each unstable branch $\Gamma_n$ accumulates on the unstable branch $\Gamma_{n+1}$.
The Proposition~\ref{p.transitive} shows that $\Gamma_0$ accumulates on $\Gamma_{\ell-1}$, hence on the last arc $I_\ell$.
\end{proof}

We introduce the following equivalence relation $\sim$ between fixed arcs $I,I'\in \cI$:
\begin{description}
\item[$I\sim I'$:] \emph{There exists a sequence of arcs $I=I_1,I_2,\dots, I_{\ell}=I'$ in $\cI$ such that for each $0\leq i<\ell$
either $I_i$ admits a $f$-invariant unstable branch which accumulates on $I_{i+1}$ or
$I_{i+1}$ admits a $f$-invariant  unstable branch which accumulates on $I_i$.}
\end{description}

\begin{lemma}\label{l.class-unique}
The relation $\sim$ has only one equivalence class.
\end{lemma}
\begin{proof}
It is enough to prove that the sum of the indices of the arcs in an equivalence class is larger or equal to $1$.
Then the Lefschetz formula (Proposition~\ref{p.lefschetz}) will conclude that there is at most one class.

Let $C$ be any equivalence class for $\sim$. It always contains a fixed arc of index $1$:

\begin{claim} The class $C$ contains a fixed arc with no $f$-invariant unstable branch.
\end{claim}
\begin{proof}
From Lemma~\ref{p.accumulation},
each $f$-invariant unstable branch of a fixed interval accumulates
on a fixed interval. If the conclusion of the claim does not hold, one thus obtain an infinite sequence
$I_n$ in $C$ such that $I_n$ admits a $f$-invariant unstable branch which accumulates on $I_{n+1}$.
Since the set $\cI$ of fixed intervals is finite, one gets a cycle between fixed interval, contradicting the conclusion
of Corollary~\ref{c.cycle}.
\end{proof}

As a consequences, we associate to any fixed arc of index $-1$ another arc of index $1$.
Let:
$$\mathcal{N}:=\mathbb{D}\setminus \bigcup\{W^s_\mathbb{D}(I_i), I_i\in \cI \text{ of index $-1$}\}.$$

\begin{claim}\label{c.index1}
Let $U$ be a connected component of $\mathcal{N}$.
Let $I\in \cI$ be an arc of index $-1$ such that $W^s_\mathbb{D}(I)$ bounds $U$.
Then $U$ contains an arc $I'\in \cI$ of index $1$ such that $I\sim I'$.
\end{claim}
\begin{proof}
We consider the sequences of arcs $I_1,\dots,I_\ell$ in $\cI$ such that
$I_1=I$, for each $k$ there exists an unstable branch of $I_k$ which accumulates on $I_{k+1}$
and each $W^s_\mathbb{D}(I_k)$ either is contained in $U$ or bounds $U$.
From Corollary~\ref{c.cycle} such a sequence is finite.
One can assume that it has maximal length.

We claim that its last element $I_\ell$
has index $1$ (hence is included in $U$).
If this is not the case and $I_\ell$ has index $0$, it is contained in $U$ and admits an unstable branch.
From Lemma~\ref{l.acc}, either this unstable branch accumulated on a fixed arc contained in $U$
or it intersects one of the boundaries $W^s_\mathbb{D}(\widetilde I)$ of $U$. In both case, we build a new
fixed arc and the sequence $I_1,\dots,I_\ell$ is not maximal, a contradiction.

By construction the last element $I':=I_\ell$ belongs to the class $C$.
\end{proof}

Now we proceed to finish the proof of Lemma \ref{l.class-unique}.  Let us choose arbitrarily a fixed arc $I(0)\in C$ of index $1$.
For each arc $I\in \cI$ of index $-1$, let us consider the connected component $V$ of
$\mathbb{D}\setminus W^s_\mathbb{D}(I)$ which does not contain $I(0)$.
Let $U$ be the connected component of
$\mathcal{N}$
which is contained in $V$ and whose boundary intersects $W^s_\mathbb{D}(I)$.
The previous claim associates to it an arc $I'\in C$ of index $1$ contained in $U$.
It is by construction different from $I(0)$.

Note that if $\widetilde I\in C$ is another arc of index $-1$,
the associated arc $\widetilde I'$ of index $1$ is different: indeed, in each component $U$ of $\mathbb{D}\setminus W^s_\mathbb{D}(I)$
which does not contain $I(0)$, there exists a unique $I\in \cI$ such that $W^s_\mathbb{D}(I)$ bounds $U$ and separates $U$
from $I(0)$.

We have shown that in $C$ the number of arcs of index $-1$ is smaller than the number of arcs of index $1$.
This concludes the proof of the Lemma~\ref{l.class-unique}.
\end{proof}

\begin{proof}[Proof of Proposition~\ref{p.chain}]
Any normally hyperbolic fixed arc is a chain for $f$.
The Lemma~\ref{l.class-unique} proves that the union $C$ of arcs in $\cI$ with their $f$-invariant unstable branches is a connected set.
Note that any arc is the union of a set of fixed points with $f$-invariant unstable branches.
This shows that $C$ of $f$ is a chain between any pair of fixed points.
\end{proof}

\begin{remark}\label{r.lefschetz}
The proof of previous proposition (and Claim~\ref{c.index1}) shows the following property:

\emph{Assume that $f$ preserves the orientation. Let $\cI$ be a finite collection of disjoint isolated arcs fixed by $f$
such that for any $I\in \cI$ and any $f$-invariant unstable branch $\Gamma$ of $I$, any periodic point in the accumulation set of $\Gamma$
belongs to some $I'\in \cI$. Then the sum of the indices $index(I,f)$ of all $I\in \cI$ is larger or equal to $1$.}
\end{remark}

\subsection{Points decreasing chain related to a stabilized point}
\label{s.connectedness-periodic}

In the present section we discuss how periodic points of larger period are related to points of lower period. Since Proposition \ref{p.chain} holds for any (orientation preserving) iterate of $f$, any periodic point  can be related to the fixed points through a chain associated to a large iterate of $f$.
In the next definitions and propositions we show that these chains have a particular structure that link points of larger period to points of lower one.

\begin{definition}
\label{d.decreasing-chain}
Let $p$ be a stabilized periodic point.
A periodic point $w\neq p$ is \emph{decreasing chain related}\index{chain, decreasing chain relation} to $p$
if there exists $k\geq 2$ and a chain $C$ for $f^k$ between $w$ and $p$ which
is contained in the closure of a decorated region of $p$.
\end{definition}

\begin{remark}\label{r.decreasing}
Note that any iterate $f^i(C)$ of the chain is contained in the closure of a decorated region of $f^i(p)$,
hence $f^i(w)$ is decreasing chain-related to $p$.
One deduces that the period of the decorated region of $p$ divides the period of $w$.
We also say that the orbit of $w$ is decreasing chain related to the orbit of $p$.
\end{remark}

The unstable set of a decreasing chain related point can be localized.

\begin{proposition}\label{p.unstable}
If $w$ is decreasing chain related to a stabilized periodic point $p$,
then the unstable set of $w$ is contained in the closure of a decorated region $V$ of $p$.

Moreover, if the period of $w$ is larger than the period of $V$ and $f$ is orientation preserving, then the closure of the unstable set of $w$ is contained in $V$.
\end{proposition}

\begin{proof}
Let us consider the two connected components of $\mathbb{D}\setminus W^s_\mathbb{D}(w)$ (one has to consider only periodic points which are not sinks, so as described in Section \ref{preliminaries}, there is a unique stable manifold well defined).
Since $w$ belongs to a decorated region $V$ of $p$, one of these components $U_1$ is contained in $V$.
The other one is denoted by $U_2$.
From Theorem~\ref{t.cycle} (no cycle), any unstable branch $\Gamma$ of $w$ is contained in one of these components.
If $\Gamma$ is included in $U_1$, then it is included in a decorated region of $p$ and the proof is concluded in that case.

We may thus assume that $\Gamma$ is included in the component $U_2$ of $\mathbb{D}\setminus W^s_\mathbb{D}(w)$
which contains $p$ and let us prove first that its accumulation set is contained in the closure of $V.$ See Figure~\ref{f.related-localized}.
Since $w$ is decreasing chain related to $p$ then  there exists a chain $C$ for an iterate $f^k$ which contains $w,p$ and which is included in the closure
of the decorated region $V$ (recall Definition \ref{d.decreasing-chain}). If $\Gamma$ is part of the chain $C$, by definition is included in the closure of $V$. So, let us consider the case where $\Gamma$ is not part of the chain;  then there exist points of $C$ in $U_2$ which accumulate on $\Gamma$.
Since the period of points in $C$ is uniformly bounded and since $\Gamma$ is an unstable branch in $U_2$,
the point $w$ is not accumulated by periodic points of $C\cap U_2$.
Consequently, there exists an unstable branch $\Gamma_C$ in $C$ which accumulates on a point of $\Gamma$.
From Proposition~\ref{p.transitive}, one deduces that the accumulation set of $\Gamma$ is included in the accumulation set of $\Gamma_C$,
which is contained in  $\overline V$, and so the accumulation set of $\Gamma$ is also contained in $\overline V$.
\begin{figure}
\begin{center}
\includegraphics[width=5cm,angle=0]{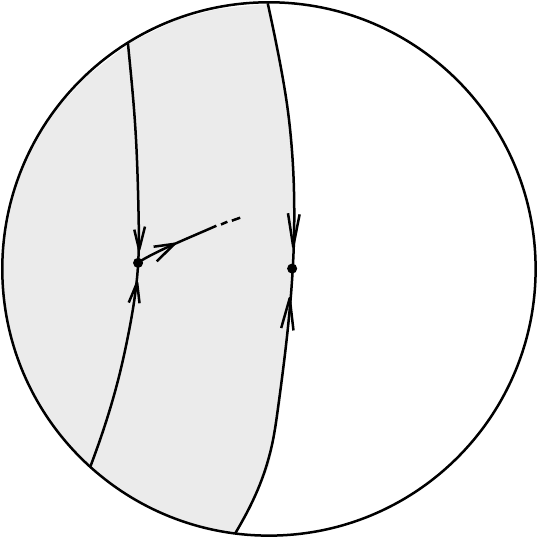}
\put(-83,88){\small $\Gamma$}
\put(-60,65){\small $p$}
\put(-115,72){\small $w$}
\put(-130,40){\small $U_1$}
\put(-90,30){\small $U_2$}
\end{center}
\caption{Proof of Proposition~\ref{p.unstable}.\label{f.related-localized}}
\end{figure}
\medskip

In order to conclude, we distinguish three cases:
\begin{itemize}
\item[--] The period of $p$ is larger than $1$.
Let $\Gamma_p$ be the unstable branch of $p$ which accumulates on a fixed point $q$ (not contained in $V$).
If $\Gamma$ is not included in the closure of $V$, it crosses $W^s_\mathbb{D}(p)$.
As a consequence the accumulation set of $\Gamma$ contains the accumulation set of $\Gamma_p$, hence $q$.
This is a contradiction since we have shown before that it is contained in $\overline V$.
\item[--] The point $p$ is a fixed saddle with {\color{black} reflection}: it admits an unstable branch $\Gamma_p$
which accumulates on a periodic point $q$ which is not in $\overline V$ (by Lemma~\ref{p.accumulation}).
One can conclude as in the previous case.
\item[--] {\color{black} The point $p$} is a fixed point, admits an eigenvalue $\lambda^+_p=-1$,
is not a sink, and not a saddle. In that case, $p$ is accumulated by points $z\in \mathbb{D}\setminus \overline V$ of period $2$.
If $\Gamma$ is not included in the closure of $V$, it crosses $W^s_\mathbb{D}(p)$,
hence it intersects the stable manifold of one of them. As a consequence the accumulation set of $\Gamma$ contains $z$.
This is a contradiction since we have shown before that it is contained in $\overline V$.
\end{itemize}
In all the cases we have shown that any unstable branch of $w$ is contained in $\overline V$.
\medskip

Let $k$ be the period of the decorated region of $p$.
If the period of $w$ is larger than $k$ and if $f$ is orientation preserving,
one applies Proposition~\ref{p.heteroclinic} to the diffeomorphism $f^k$:
the unstable set of $w$ does not accumulate on $p$;
hence its closure is contained in $V$, proving the last part of the proposition.
\end{proof}

\begin{corollary}\label{c.dichotomy-stabilized}
A point $w$ which is stabilized can not be decreasing chain-related to a stabilized point $p$.
\end{corollary}
\begin{proof}
We argue by contradiction.
Let us first assume that $p$ is not fixed.
From Proposition~\ref{p.unstable},
the unstable set of each iterate $f^i(w)$ is contained in the decorated region of $f^i(p)$.
Since the decorated region has period larger or equal to $2$ and since $p$ is not fixed,
the unstable set of $w$ does not accumulate on a fixed point, a contradiction.

 When $p$ is fixed, since the unstable manifold of $w$ is contained in a decorated region,
the point $w$ can only be stabilized by $p$.
The definition of stabilized point for $p$ gives that $p$ is not a sink;
the definition for $w$ implies that the unstable manifold of $w$
crosses $W^s_\mathbb{D}(p)$, a contradiction.
\end{proof}

\begin{proposition}\label{p.both-related}
Let us consider two stabilized fixed points $p_1,p_2$ with decorated regions $V_1,V_2$.
If there exists a point $w\in V_1$ that is decreasing chain related to $p_2$, then $V_2\subset V_1$.
\end{proposition}
\begin{proof}
By assumption $q\in V_1\cap V_2$. Since the stable manifolds of $p_1$ and $p_2$ are disjoint,
if the conclusion of the proposition does not hold, then $V_1\subset V_2$.

Since $w$ is decreasing chain related to $p_2$ in $V_2$,
there exists an iterate $f^\ell$ and a chain $C$ for $f^\ell$ included in $\overline {V_2}$
which contains both $w$ and $p_2$.
In particular it intersects $W^s_{\DD}(p_1)$.
This implies that there exists an unstable branch $\Gamma\subset C$ which
meets both components of $\DD\setminus W^s_\DD(p_1)$.
One deduces from Proposition~\ref{p.transitive} that the closure of the unstable set of $p_1$
is contained in the closure of $\Gamma$, hence in $\overline {V_2}$.

The closure of the unstable set $p_1$ contains a fixed point (since $p_1$ is stabilized):
the only possible fixed point is $p_2$. By definition of stabilization,
either $p_2$ is not stabilized or the unstable set $p_1$ meets both components of $\DD\setminus W^s_\DD(p_2)$.
In both cases we get a contradiction
\end{proof}

\begin{corollary}\label{c.unique-stabilized}
A point $w$ can not be decreasing chain-related to two different stabilized point $p_1,p_2$.
\end{corollary}
\begin{proof}
Otherwise $w$ would belong to decorated regions $V_1$ and $V_2$ for $p_1$ and $p_2$ and respectively.
The Proposition~\ref{p.both-related} would imply  that $V_1\subset V_2$ and $V_2\subset V_1$ simultaneously. A contradiction
since $p_1\neq p_2$.
\end{proof}

One also describes the accumulation sets of $f$-invariant unstable branches.

\begin{proposition}\label{p.chain-decreasing}
Let $z$ be a fixed point and $\Gamma$ be a $f$-invariant unstable branch of $z$.
Let $C$ be a chain for some iterate $f^k$ between two periodic points $w,p$ such that:
\begin{itemize}
\item[--] $p$ is stabilized by a fixed point $q$,
\item[--] $w$ is decreasing chain-related to $p$,
\item[--] $C$ is contained in the closure of a decorated region of $p$.
\end{itemize}
If the accumulation set of $\Gamma$ contains $w$, then it also contains $p$ (see Figure~\ref{f.unstable}).
\end{proposition}
\begin{proof}
By invariance, the points $f(w)$ and $f(p)$ and the chain $f(C)$ satisfy the same properties.
Note that the decorated region of $p$ containing $w$ and the decorated region of $f(p)$ containing $f(w)$
are disjoint: this is a consequence of Proposition~\ref{p.stab-decorate} when $p$ has period larger than $1$;
when $p$ is fixed, this is a consequence of the fact that its two decorated regions are locally exchanged (since in the case that  $p$ is fixed by Definition \ref{d.stabilization} the non-stable eigenvalue with modulus larger and equal to one is  negative).

The $f$-invariant unstable branch $\Gamma$ accumulates in $w$ and $f(w)$ and so it has to intersect two different decorated regions,
hence intersects the stable manifold of the orbit of $p$.
This gives the conclusion.
\end{proof}
\begin{figure}
\begin{center}
\includegraphics[width=5cm,angle=0]{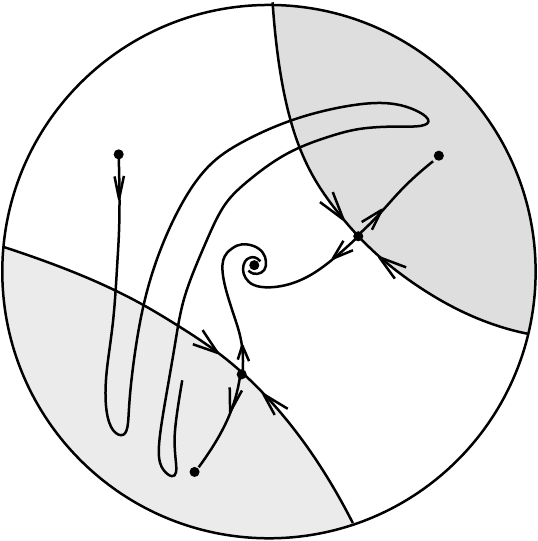}
\put(-70,78){\small $q$}
\put(-50,70){\small $p$}
\put(-120,95){\small $z$}
\put(-75,45){\small $f^2(p)$}
\end{center}
\caption{A $f$-invariant unstable branch intersecting a decorated region.\label{f.unstable}}
\end{figure}

\subsection{Lefschetz formula associated to a stabilized point}
Using the notion of decreasing chain related periodic point, we define the notion of index of a decorated region.

\paragraph{\it Index of a decorated region.}\index{index, Lefschetz formula} 
Given a decorated region $V$ of a stabilized periodic point $p$,
and a multiple $n$ of the period $k$ of $V$,
{\color{black} we introduce $C(p,V,n)$, the set} of points $w\in V$ that are fixed by $f^n$ and decreasing chain-related to $p$.
Note that by Corollary~\ref{c.orientation},
the map $f^n$ preserves the orientation.

{\color{black} When the set $C(p,V,n)$ is finite, the total index $L(V,f^n)$ of points that are fixed by $f^n$ in $\overline V$ is the sum of the indices of the points $w$ in $C(p,V,n)$ and of the half index $index(p,V,f^n)$, as defined in Section~\ref{ss.index arc}.
$$L(V,f^n):=index(p,V,f^n)+\sum_{w\in C(p,V,n)} index(w, f^n). $$}
{\color{black} When the set is infinite, one also defines an index, by grouping the periodic points inside arcs.}
From Section \ref{ss.arc}, there exists a finite family $\cI$ of disjoint arcs
that are fixed by $f^n$ and contained in $\overline V$ such that the set of periodic points in $\cup_{I\in\cI} I$
is exactly $\{p\}\cup C(p,V,n)$. We denote by $I_0$ the arc of $\cI$ which contains $p$.
Note that the other arcs $I\in \cI$ are isolated, hence have an index $index(I, f^n)$.
The arc $I_0$ is maybe not isolated (in the case $p$ is a fixed stabilized point),
but one can consider the index $index(I_0, V,f^n)$ of the half arc $I_0$ in the region $V$ for $f^n$
as defined in Section~\ref{ss.index arc}.

Then, one defines the index of the decorated region $\overline V$ for $f^n$ as
$$L(V,f^n):=index(I_0,V,f^n)+\sum_{I\in \cI\setminus \{I_0\}} index(I, f^n). $$
Observe that the number $L(V,f^n)$ does not depend on the choice of the family $\cI$:

\begin{proposition}\label{p.local-lefschetz}
For any decorated region $V$ of a stabilized periodic point $p$,
and for any multiple $n$ of the period $k$ of $V$,
the index $L(V,f^n)$ equals $1/2$.
\end{proposition}

The proof is postponed to Section~\ref{ss.structure}.
Before, we prove a weaker statement.
\begin{lemma}\label{l.local-lefschetz}
For any decorated region $V$ of a stabilized periodic point $p$,
and for any multiple $n$ of the period $k$ of $V$,
the index $L(V,f^n)$ is larger or equal to $1/2$.
\end{lemma}
\begin{proof}
Since $f^n$ preserves the orientation,
we follow the proof of Proposition~\ref{p.chain} and prove a version of Remark~\ref{r.lefschetz}
inside the decorated region  $V$,
after making the following observations:
\begin{itemize}
\item[--] for any $I\in \cI$, any $f^n$-invariant unstable branch $\Gamma$ of $I$ is contained in $\overline V$
(unless when $I=I_0$ and $\Gamma$ is the unstable branch that stabilizes $p$),
\item[--] any point fixed by $f^n$ in the accumulation set of $\Gamma$ is contained in $\overline V$
(since such a point coincides with $p$ or is decreasing chain related to $p$).
\end{itemize}
Let:
$$\mathcal{N}:=V\setminus \bigcup\{W^s_\mathbb{D}(I_i), I_i\in \cI \text{ of index $-1$}\}.$$
The proof of Claim~\ref{c.index1} shows that any component $U$ of $\mathcal{N}$, either it 
contains an arc $I'\in \cI$ of index $1$ or it is the component bounded by $W^s_\mathbb{D}(I_0)$
and $I_0$ is semi-attracting in $V$.

To each arc $I\in \cI\setminus \{I_0\}$ of index $-1$, one
let $V_I$ be the component of $\mathbb{D}\setminus W^s_\mathbb{D}(I)$
which does not contain the stabilizing unstable branch of $I_0$.
One associates by Claim~\ref{c.index1} an arc $I'$ of index $1$
in the component of $\mathcal{N}$ bounded by $W^s_\mathbb{D}(I)$ which belongs to $V_I$.

When the arc $I_0$ has a $f^n$-invariant unstable branch in $V$
(and has half index $index(I_0,V,f^n)=-1/2$),
one can also associate by the Claim~\ref{c.index1} an arc of index $1$ which belongs to the component of $\mathcal{N}$
bounded by $W^s_\mathbb{D}(I_0)$.

The number of arcs of index $1$ in $\cI$ is thus larger or equal to the number of arcs of index $-1$, and it is  larger or equal to the number of arcs of index $-1$
plus $1$ in the case $index(I_0,V,f^n)=-1/2$.
This proves that the sum of the indices $L(V,f^n)$ is always larger or equal to $1/2$.
\end{proof}

\subsection{Structure of the set of periodic points}\label{ss.structure}
The next proposition classifies the periodic points.

\begin{proposition}\label{p.decreasing-chain}
For any periodic point $w$, one and only one of the possibilities occurs:
\begin{itemize}
\item[(1)] $w$ is fixed and either is a sink or $Df(w)$ has an eigenvalue $\geq 1$,
\item[(2)] $w$ is stabilized,
\item[(3)] $w$ is decreasing chain related to a stabilized periodic point.
\end{itemize}
\end{proposition}

\begin{proof} The options (1) and (2) are incompatible by definition of the stabilization.
Options (2) and (3) are incompatible by Corollary~\ref{c.dichotomy-stabilized}.
Also (1) and (3) are incompatible by Remark~\ref{r.decreasing}.
It remains to prove that any periodic point $w$ satisfies one of the cases.
 
Let $f^n$ be an orientation-preserving iterate that fixes $w$
and let $\cI$ be a finite collection of isolated arcs fixed by $f^n$
which contains all the points fixed by $f^n$.
Let $\cI_0$ be the set of intervals $I\in \cI$ containing a periodic point satisfying one of the cases (1), (2) or (3).

\begin{claim}\label{c.decompose-arc}
For $I\in \cI_0$, any periodic point in $I$ satisfies the proposition.
More precisely one and only one of the following cases occurs:
\begin{itemize}
\item[--] the periodic points in $I$ are all fixed and not stabilized,
\item[--] $I$ contains either  a stabilized point $p$ or a point decreasing chain related to a stabilized point $p$:
all the other periodic points in $I$ are decreasing chain related to $p$.
\end{itemize}
\end{claim}
\begin{proof}
We can assume that $I$ is not reduced to a single periodic point
(in that case the statement holds immediately). We consider three cases:

If $I$ contains a fixed point $q$ with an eigenvalue $\lambda^+_q\geq 1$,
then any periodic point in $I$ is fixed and can not be stabilized. The first case occurs.

If $I$ contains a fixed point $q$ with eigenvalue $\lambda^+_q\leq -1$,
the other periodic points in $I$ have period $2$: if $q$ is not a sink, it is stabilized
and the other periodic points in $I$ are decreasing chain-related to $q$;
if $q$ is a sink, its basin in $I$ is bounded by a $2$-periodic orbit $\{p,f(p)\}$,
the other periodic points in $I$ are decreasing chain related to $p$ or $f(p)$.

If $I$ does not contain any fixed point, but contains a stabilized point $p$,
then it is contained in the closure of the decorated region $V$ of $p$.
Otherwise, by $f^n$-invariance, $I$ would contain the stabilized unstable branch of $p$
and its accumulation set: a contradiction since $I$ does not contain any fixed point.
One deduces that any periodic point $I$ different from $p$
is decreasing chain related to $p$.

If $I$ does not contain any fixed point, nor any stabilized periodic point,
but contains a point decreasing chain related to a stabilized point $p$,
one deduces that $I$ is contained in a decorated region $V$ of $p$.
Otherwise $I$ would intersect $W^s_\DD(p)$, and hence by $f^n$-invariance would contain $p$.
Therefore any periodic point in $I$ is also decreasing chain related to $p$.
\end{proof}

\begin{claim}\label{c.control-unstable}
For any $I\in \cI\setminus \cI_0$ and any $f^n$-invariant unstable branch $\Gamma$,
any periodic point in the accumulation set of $\Gamma$ belongs to some $I'\in \cI\setminus \cI_0$.
\end{claim}
\begin{proof}
Let us consider an endpoint $z\in I$ with a $f^n$-invariant unstable branch $\Gamma$
whose accumulation set contains a $f^n$-invariant point $q$.
Let us assume by contradiction that the interval $I'\in \cI$ containing $q$ belongs to $\cI_0$.
We distinguish two cases.
\smallskip

{\em i--} The point $q$ is fixed: if $q$ satisfies case $(1)$, then $z$ is stabilized, a contradiction; if $q$ satisfies case $(2)$, since $z$ is not stabilized, the Definition~\ref{d.stabilization} implies that $\Gamma$ does not intersect one of the components of $\DD\setminus W^s_\DD(q)$. Therefore, by Definition~\ref{d.decreasing-chain} one deduces that $z$ is decreasing chain related to $q$;  this is a contradiction since $I\not\in \cI_0$.\hspace{-1cm}\mbox{}
\smallskip

  {\em ii--}  The point $q$ is not fixed: Since $I'\in \cI_0$, from the previous claim there exists a stabilized point $p$
such that all the periodic points in $I'$ are decreasing chain related to $p$ or coincide with $p$.
Let $V$ be the decorated region associated to $p$ which contains $q$.
By Definition~\ref{d.decreasing-chain}, there exists a chain $C\subset \overline V$ for $f^n$ containing $q$ and $p$.
Note that $\Gamma$ cannot intersect the region $\DD\setminus \overline V$:
when $p$ is fixed, this would immediately imply that $z$ is stabilized, a contradiction;
when $p$ is not fixed, this would imply (by Proposition~\ref{p.transitive})
that the accumulation set of $\Gamma$ would contain the accumulation set of the stabilized branch of $p$,
and then that $z$ is stabilized, a contradiction.
One deduces that $I\cup I'\cup \Gamma\cup C$ is a chain for $f^n$
containing $z$ and $p$. It is contained in $\overline V$ hence $z$ is decreasing chain related to $p$.
A contradiction.
\end{proof}

We can now conclude the proof of the Proposition~\ref{p.decreasing-chain}.
From Claim~\ref{c.decompose-arc}, one can for each stabilized point $p$ consider the family $\cI_p$ of arcs $I\in \cI_0$
such that all the periodic points in $\cI_p$ are decreasing chain related to $p$ or equal to $p$.
One can also consider the family $\cI_{fix}$ of arcs whose periodic points are fixed and not stabilized.
The family $\cI_0$ decomposes as the disjoint union of $\cI_{fix}$ with the families $\cI_p$, for $p$ stabilized.

Let $p$ be a stabilized fixed point, with decorated regions $V_1,V_2$.
Lemma~\ref{l.local-lefschetz} implies
\begin{equation}\label{e.fixed-stab}
\sum_{I\in\cI_p}index(I,f^n)=L(V_1,f^n)+L(V_2,f^n)\geq 1.
\end{equation}
Let $p$ {\color{black} be} a stabilized point fixed by $f^n$ but not by $f$. It has one decorated region $V$.
Let $I_p$ be the arc in $\cI_p$ which contains $p$.
Since $p$ has an unstable branch in the region $\mathbb{D}\setminus \overline V$,
we get $index(I,\mathbb{D}\setminus \overline V,f^k)=-1/2$. Consequently Lemma~\ref{l.local-lefschetz} implies
\begin{equation}\label{e.notfixed-stab}
\sum_{I\in\cI_p}index(I,f^n)=L(V,f^n)+index(I_p,\mathbb{D}\setminus \overline V,f^n)\geq 0.
\end{equation}
 
Note that if $I\in \cI$ contains a stabilized fixed point, then $index(I,f)=1$, whereas for $I\in \cI_{fix}$ one has $index(I,f)=index(I,f^n)$.
Therefore the Lefchetz formula (Proposition~\ref{p.lefschetz}) for $f$ gives
$$\sum_{I\in\cI_{fix}}index(I,f^n)=\sum_{I\in\cI_{fix}}index(I,f)=1-\operatorname{Card}\{p \text{ fixed and stabilized}\}.
$$
Combining the three previous inequalities give
$$\sum_{I\in\cI_{0}}index(I,f^n)\geq 1.$$

If one assumes that $\cI\setminus \cI_0$ is non-empty,
the Claim~\ref{c.control-unstable} and the Remark~\ref{r.lefschetz} give
$$\sum_{I\in\cI\setminus \cI_{0}}index(I,f^n)\geq 1.$$
This gives $\sum_{I\in\cI}index(I,f^n)\geq 2$ which contradicts the Lefschetz formula  (Proposition~\ref{p.lefschetz}).
Consequently $\cI=\cI_0$ and any point fixed by $f^n$ satisfies one of the cases of the Proposition~\ref{p.decreasing-chain}.
The proof is complete.
\end{proof}

We can now complete the proof of the Lefschetz formula inside a decorated region.
\begin{proof}[Proof of Proposition~\ref{p.local-lefschetz}]
We argue as in the proof of the Proposition~\ref{p.decreasing-chain} for the orientation preserving iterate $f^n$.
We consider a collection $\cI$ of disjoint isolated arcs fixed by $f^n$.
For each stabilized point $p$, we consider the collection of arcs $\cI_p$ containing points decreasing chain related to $p$
and the point $p$ itself.
We also consider the family $\cI_{fix}$ of arcs whose periodic points are fixed and not stabilized.
The family $\cI$ is partitioned as the disjoint union of $\cI_{fix}$ with the families $\cI_p$, for $p$ stabilized.

Arguing as before, the conclusion of Lemma~\ref{l.local-lefschetz} gives the inequality~\eqref{e.fixed-stab} for any $p$ stabilized and fixed 
and it gives the inequality~\eqref{e.notfixed-stab} for $p$ stabilized and not fixed.
If one of these inequalities is strict, one deduces $\sum_{I\in\cI}index(I,f^n)> 1$
and contradicts the Lefschetz formula  (Proposition~\ref{p.lefschetz}).
Consequently the inequalities~\eqref{e.fixed-stab} and~\eqref{e.notfixed-stab} are equalities.
This means that the inequality in Lemma~\ref{l.local-lefschetz} is an equality and Proposition~\ref{p.local-lefschetz} holds.
\end{proof}

\section{Trapping discs}
\label{ss.trapping}

A compact set $\Delta\subset \mathbb{D}$ is a \emph{(topological) disc} if it is homeomorphic to the unit disc.
It is \emph{trapping}\index{trapping disc} for $f$ if $f(\Delta)\subset \operatorname{Interior}(\Delta)$.
In this section we prove the following result.

\begin{theorem}\label{t.renormalize}
Let $f$ be a mildly dissipative diffeomorphism of the disc with zero topological entropy
and $\Gamma$ be a $f$-invariant unstable branch of a fixed point $p$.
Then there exists a trapping disc $\Delta$ containing the accumulation set of $\Gamma$ and disjoint from $W^s(p)$.
\end{theorem}

It is enough to prove the theorem in the case where $f$ is orientation preserving.
Let us consider the finite set $\cI$ of isolated fixed arcs as introduced in Section~\ref{ss.arc}.
Since there is no cycle of fixed arcs (Corollary~\ref{c.cycle}),
the elements of $\cI$ can be ordered as a sequence $I_1,\dots,I_n$ such that there is no $f-$invariant unstable branch of $I_i$ which accumulates on $I_j$
when $j\geq i$. The proof first deals with the $f$-invariant unstable branches of the arcs $I_i$, by induction on $i$.
In this case we have a more precise version.

\addtocounter{theorem}{-1}
\renewcommand{\thetheorem}{\Alph{theorem}'}
\begin{theorem}\label{t.renormalize2}
Let $f$ be a mildly dissipative diffeomorphism of the disc with zero topological entropy
and $\cI$ a set of isolated fixed arcs as introduced in Section~\ref{ss.arc}.
For any $I_i\in \cI$ and any $f$-invariant unstable branch $\Gamma$ of $I_i$,
let $Z$ be the closure of the union of:
\begin{itemize}
\item[--] the accumulation set $\Lambda$ of $\Gamma$,
\item[--] the arcs $I_j\in \cI$ for $j<i$,
\item[--] the $f$-invariant unstable branches of the arcs $I_j$ for $j<i$.
\end{itemize}
Then, $\Lambda$ is included in {\color{black} a} trapping disc $\Delta$ which is contained in an arbitrarily small neighborhood of $Z$.
\end{theorem}
\renewcommand{\thetheorem}{\Alph{theorem}}

In the following we will first prove the second theorem and then deduce the first. As an immediate consequence one gets:

\begin{corollary}\label{c.trapping}
Let us consider an isolated fixed arc $I=I_i$
which is not reduced to a fixed point with eigenvalue $-1$.
Let $U$ be an open set which contains 
the arcs $I_j\in \cI$ for $j\leq i$ and the closure of
their $f$-invariant unstable branches.
Then there exists a trapping disc $\Delta\subset U$
which contains $I$.
\end{corollary}
\medskip

One also deduces that periodic points are almost isolated in the recurrent set of $f$.

\begin{corollary}\label{c.isolated}
Let us consider an isolated fixed arc $I$
which is not reduced to a fixed point with eigenvalue $-1$.
Then, there exists a neighborhood $W$ of $I$ such that
\begin{itemize}
\item[--] the $\alpha$-limit set of any point $z\in W$ is either disjoint from $W$ or a fixed point of $I$,
\item[--] the $\omega$-limit set of any point $z\in W$ is either disjoint from $W$ or
a fixed point of $I$.
\end{itemize}
\end{corollary}

Note that a fixed point with eigenvalue $-1$ is contained in an isolated
fixed arc $I'$ for $f^2$ to which the corollary may be applied.
This gives:

\begin{corollary}\label{c.isolated2}
Any periodic orbit $\cO$, with period $N$, admits a neighborhood $W$
such that any ergodic measure $\mu$ satisfying $\mu(W)>0$
is supported on a periodic orbit with period less or equal to $2N$.
\end{corollary}

The construction of the  trapping domains in a small neighborhood that contains the accumulation set of $\Gamma$ in the proofs of Theorems~\ref{t.renormalize} and \ref{t.renormalize2} go along the following lines:
\begin{enumerate}
\item Using a slight variation of Definition \ref{d.pixton0} we build {\em Pixton {\color{black} discs which either are bounded by}} arcs  of $\Gamma$ and local stable {\color{black} manifolds} of stabilized periodic orbits with period one or larger (Lemma \ref{l.highperiod}), {\color{black} or are basins} of attraction of (semi)attracting fixed points (Lemma \ref{l.periodone}).
\item The union of these Pixton discs can be refined in a larger Pixton disc that contains all its iterates, the periodic points accumulated by $\Gamma$ and any decreasing chain-related point to them (Corollary \ref{c.pixton}).
\item Using the closing lemma (Theorem~\ref{t.measure local}) we prove that {\color{black} the backward orbit of} any point in the accumulation set of $\Gamma$ {\color{black} is contained} in the interior of the  Pixton disc described in previous item.
That allows to {\color{black} slightly modify} the Pixton disc {\color{black} in order to} guarantee  that {\color{black} its} forward
{\color{black} its forward iterates are} contained in {\color{black} the interior of the disc}.
\end{enumerate}

\subsection{Pixton discs revisited}\label{ss.pixton}
We prepare here the proof of Theorem~\ref{t.renormalize2}.
We assume in this section that $f$ preserves the orientation.

We consider an arc $I_i\in \cI$ and a $f$-invariant unstable branch $\Gamma$ of an endpoint $p$ of $I_i$.
Arguing by induction, we may assume that Theorem~\ref{t.renormalize2} holds for the $f$-invariant unstable branches of any arc $I_k\in \cI$ with $k<i$.
Let $Z$ be the invariant compact set introduced in the statement of the theorem. By assumption on the order inside the family $\cI$,
the set $Z$ disjoint from $W^s_\mathbb{D}(p)$.
We choose a neighborhood $U$ of $Z$ disjoint from $W^s_{\mathbb{D}}(p)$.
\medskip

We introduce the following notion, which is slightly different than the Definition \ref{d.pixton0} given before.

\begin{definition}\label{d.pixton} Given a $f$-invariant unstable branch $\Gamma$, a \emph{Pixton disc}\index{Pixton disc} associated to $\Gamma$ is a closed topological disc $D$ whose boundary is a Jordan curve which decomposes into
\begin{itemize}
\item[--] a closed set $\gamma^s$ satisfying $f^n(\gamma^s) \subset \operatorname{Interior}(D)$ for all $n$ larger than some $n_D\geq 1$,
\item[--] and its complement $\gamma^u$ (that could be empty)
which is contained in $\Gamma$.
\end{itemize}
\end{definition}

\begin{remarks}\label{r.pixton}
{\color{black} The following straightforward statements hold for Pixton discs:}
\begin{enumerate}
\item A trapping disc is a Pixton disc. Conversely a Pixton disc such that $\gamma^u=\emptyset$
is a trapping disc. In particular, an attracting fixed point {\color{black} is} associated {\color{black} to a} Pixton disc.
\item The forward iterates of a Pixton disc are Pixton discs.
\item If $D_1,D_2$ are two Pixton discs whose intersection is non-empty, then one obtains a new Pixton disc
$D$ by considering their ``filled union": this is the union of $D_1\cup D_2$ with all the connected components
of {\color{black} $\DD\setminus (D_1\cup D_2)$} which do not contain the boundary of $\mathbb{D}$. By~\cite{kerekjarto} (see also~\cite{LY}), the filled union is a disc. 
The new set $\gamma^s$ is contained in the union of the sets $\gamma^s_1,\gamma_2^s$ associated to $D_1,D_2$.
The same holds for $\gamma^u$.
\end{enumerate}
\end{remarks}
Observe that previous remark provides the proof of the first step in the induction argument: the first arc $I_1$ in $\cI$ is an attracting arc.

In what follow until the end of the section, $p, \Gamma, U$ are the fixed point, unstable arc and neighborhood defined at the beginning of the section. 
In order to prove Theorem~\ref{t.renormalize}, we need to cover periodic points in the accumulation set of $\Gamma$
by Pixton discs. This is done first for periodic points with period larger than $1$, and later for fixed points.

\begin{lemma}\label{l.highperiod}
Consider a periodic orbit $\cO$ accumulated by $\Gamma$ and stabilized by a fixed point $q$.
Then there exists a Pixton disc $D\subset U$
which contains $\cO$ in its interior and whose stable boundary $\gamma^s$ is contained in the stable manifold $W^s_\mathbb{D}(\cO)$ of $\cO$.
\end{lemma}
\begin{proof}

We first assume that $\widetilde w$ has period $\tau\geq 2$.
See Figure~\ref{f.pixton}.

Let us consider the universal cover $\widetilde{ \mathbb{D}}$ of $\mathbb{D}\setminus \{q\}$:
it is homeomorphic to the strip $\mathbb{R}\times [0,1)$ and the translation $(x,y)\mapsto (x+1,y)$ can be chosen to be a covering automorphism
which generates the fundamental group.
Let $\widetilde p$ and $\widetilde \Gamma$ be lifts of $p$ and of the unstable branch $\Gamma$.
We choose the lift $\widetilde f$ of $f$ which preserves $\widetilde p$ and $\widetilde \Gamma$.

\begin{figure}
\begin{center}
\includegraphics[width=8cm,angle=0]{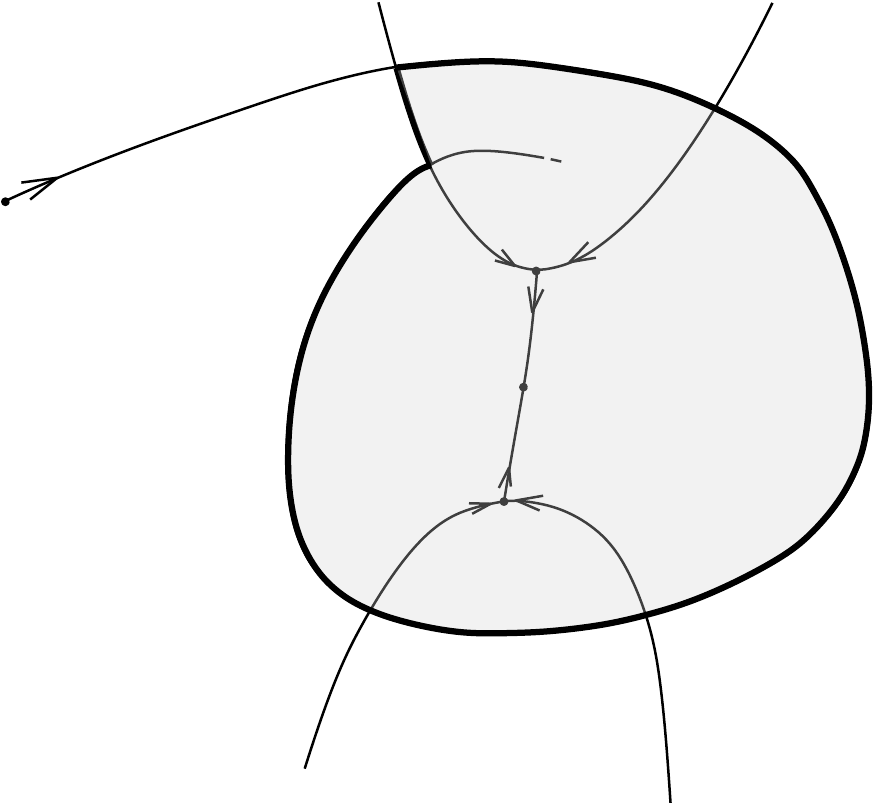}
\put(-178,168){\small $\Gamma$}
\put(-88,110){\small $q$}
\put(-220,150){\small $p$}
\put(-83,133){\small $w$}
\put(-92,83){\small $f(w)$}
\put(-35,80){\small $D$}
\put(-170,80){\small $\gamma^u$}
\put(-135,180){\small $\gamma^s$}
\put(-125,200){\small $W^s$}
\put(-107,120){\small $W^u$}
\includegraphics[width=9cm,angle=0]{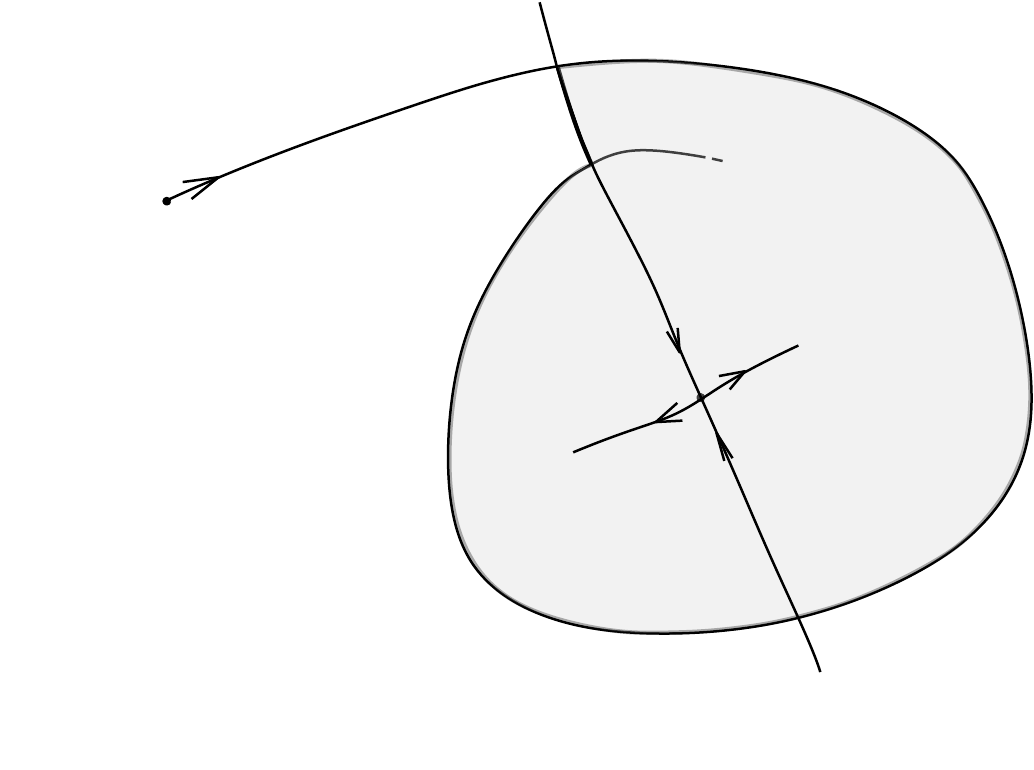}
\put(-210,135){\small $p$}
\put(-163,153){\small $\Gamma$}
\put(-160,80){\small $\gamma^u$}
\put(-125,155){\small $\gamma^s$}
\put(-35,80){\small $D$}
\put(-120,190){\small $W^s$}
\put(-78,88){\small $w$}
\end{center}
\caption{Construction of a Pixton disc: $w$ has period $2$ (left) or $1$ (right).\label{f.pixton}}
\end{figure}
\begin{figure}
\begin{center}
\includegraphics[width=12cm,angle=0]{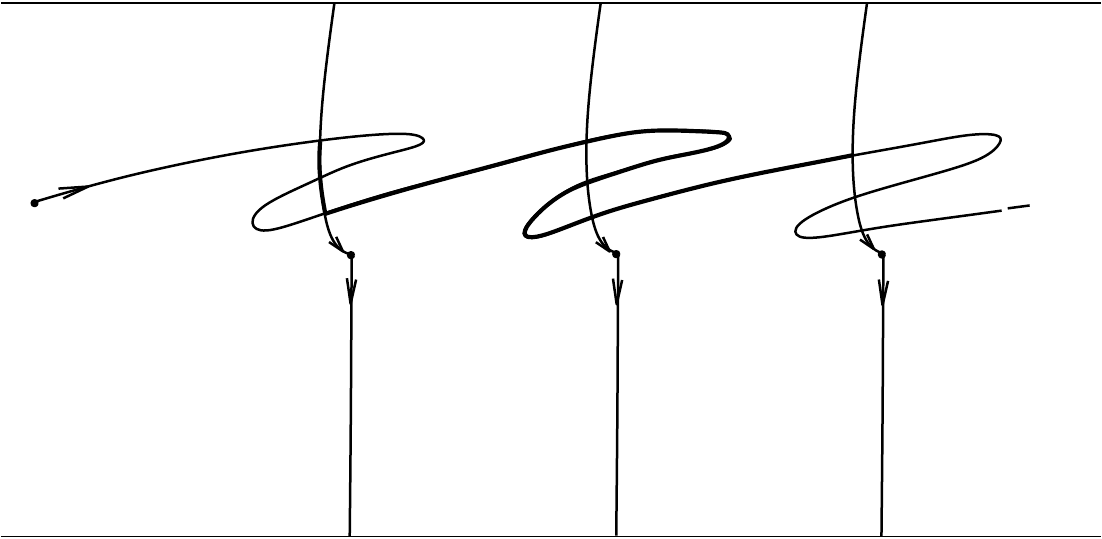}
\put(-300,115){\small $\widetilde \Gamma$}
\put(-328,98){\small $p$}
\put(-238,95){\small $a$}
\put(-75,122){\small $b$}
\put(-238,140){\small $\widetilde W^s_n$}
\put(-230,40){\small $\widetilde W^u_n$}
\put(-153,140){\small $\widetilde W^s_{n+1}$}
\put(-73,140){\small $\widetilde W^s_{n+\tau}$}
\end{center}
\caption{Proof of Lemma~\ref{l.highperiod}.\label{f.pixton-lift}}
\end{figure}

Consider $w\in \cO$, one of its stable branches $W^s\subset W^s_\mathbb{D}(w)$ connecting $w$ to a point $z$ in the boundary of $\mathbb{D}$,
and $W^u$ the unstable branch that accumulates on $q$. The points $w,z$ and the curve $W:= W^s\cup W^u$ lift as
$\widetilde w, \widetilde z\in \widetilde W=\widetilde W^s\cup\widetilde W^u$. One may assume that $\widetilde z=(0,0)$.
Note that $\widetilde W$ separates the strip: its complement contains two components
bounded by $(-\infty,0)\times \{0\}$ and $(0,+\infty)\times \{0\}$ respectively.

To any lift $\widetilde w'=\widetilde f^k(\widetilde w)+(\ell,0)$ of any iterate $f^k(w)$ of $w$,
one associates in a same way a curve $\widetilde W'$, disjoint from $\widetilde W$:
it either lands on $(-\infty,0)\times \{0\}$ (in which case one denotes $\widetilde W'<\widetilde W$) or on $(0,+\infty)\times \{0\}$.
One defines in this way a totally ordered collection of separating sets $\dots<\widetilde W_{n-1}<\widetilde W_n<\widetilde W_n<\dots$
such that $\widetilde W_n+(1,0)=\widetilde W_{n+\tau}$. Since the point $w$ is not fixed, the sets $\widetilde W_n=\widetilde W^s\cup\widetilde W^u$ are not fixed by $\widetilde f$:
there exists $j\neq 0$ such that $\widetilde f(\widetilde W_n)\subset \widetilde W_{n+j}$ for any $n\in \mathbb Z$.
We may assume without loss of generality that $j\geq 1$. See Figure~\ref{f.pixton-lift}.

The $\widetilde f$-invariant curve $\widetilde \Gamma$ accumulates on each set $\widetilde f^k(\widetilde W)\subset \widetilde W_{kj}$, $k\geq 0$.
Since the sets are separating, it intersects all the sets $\widetilde W_n$, $n\geq 0$.
Note that the unstable branch $\widetilde \Gamma$ does not intersects the curve $\widetilde W^u_n$.
It follows that it intersects all the $\widetilde W^s_n$, $n\geq 0$.

For $n\geq 1$ large, let $\widetilde \gamma^u$ be a (open) curve in $\widetilde \Gamma$ which connects $\widetilde W^s_n$ to  $\widetilde W^s_{n+\tau}=\widetilde W^s_n+(1,0)$
at two points $a\in \widetilde W^s_n$ and $b\in \widetilde W^s_n+(1,0)$.
Let $\widetilde \gamma^s\subset \widetilde W^s_n$ be the (closed) curve which connects $a$ to $b-(1,0)$.
The curve $\widetilde \gamma^s\cup \widetilde \gamma^u$ projects on a simple closed curve $\gamma=\gamma^s\cup \gamma^u$ of $\mathbb{D}$
which bounds a disc $D$.
By construction, the large forward iterates of $\gamma^s$ converge to the orbit of $w$, hence are contained in $D$.
One deduces that $D$ is a Pixton disc.

Note that the lift $\widetilde \gamma=\cup_{k\in \mathbb Z} (\widetilde \gamma^s\cup \widetilde \gamma^u+(k,0))$ separates
the boundary $\mathbb{R}\times \{0\}$ from the sets $\widetilde W^u_n$.
This implies that the disc $D$ contains all the unstable branches $f^k(W^u)$ of the iterates of $w$ and in particular the orbit $\mathcal O$.

Up to {\color{black} replacing} $D$ by a large iterate, one find a Pixton disc whose unstable boundary $\gamma^u$ is arbitrarily close to the limit
set $\Lambda$, whose stable boundary $\gamma^s\subset W^s_\mathbb{D}(w)$ has arbitrarily small diameter,
and whose area is arbitrarily small. One deduces that the disc is in an arbitrarily small neighborhood of its unstable boundary, hence of $\Lambda$.
Consequently it is included in $U$ as required.
\medskip

In the case where $w=q$ has period $1$ but negative eigenvalue, we argue in a similar way.
We denote by $W^s_0$ and $W^s_1$ the two stable branches of $q$ and we lift them as an ordered collection
of separating curves $\dots<\widetilde W^s_n<\widetilde W^s_{n+1}<\dots$ such that the curves $\widetilde W^s_{2n}$ lift $W^s_0$
and the curves $\widetilde W^s_{2n+1}$ lift $W^s_1$. Moreover $\widetilde W^s_{n+2}=\widetilde W^s_n+(1,0)$.
Since $f(W^s_0)\subset W^s_1$ and $f(W^s_1)\subset W^s_0$, the curves $\widetilde W^s_n$ are not fixed by $\widetilde f$.
The end of the proof is similar: we get a curve $\widetilde \gamma$ which separates the boundary $\mathbb{R}\times \{0\}$
from a line $\mathbb{R}\times \{1-\delta\}$, $\delta>0$ small. It projects as a simple closed curve which bounds a Pixton disc containing $q$
as required.
\end{proof}

\begin{lemma}\label{l.periodone}
Each fixed point $p'$ accumulated by $\Gamma$ and which does not have an eigenvalue less or equal to $-1$
is contained in a trapping disc $D\subset U$.
\end{lemma}
\begin{proof}
We use the inductive assumption stated before the Section~\ref{ss.pixton}.
The fixed point $p'$ belongs to an arc $I'=I_j$ in $\cI$. From our choice of the order on $\cI$, we have $j<i$.
If $I'$ has the type of a sink, it admits arbitrarily small neighborhoods that are trapping disc.
Note that $I'$ can not have the type of a point with {\color{black} reflection} (since $p'$ does not have an eigenvalue less or equal to $-1$).
\begin{figure}
\begin{center}
\includegraphics[width=8cm,angle=0]{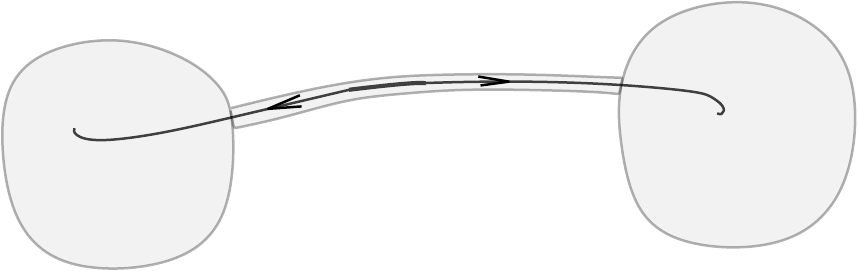}
\put(-80,55){\small $V$}
\put(-130,38){\small $I'$}
\put(-30,20){\small $D'_2$}
\put(-200,20){\small $D'_1$}
\end{center}
\caption{A Pixton disc $D=D'_1\cup V\cup D'_2$ obtained by gluing two Pixton dics.\label{f.tube}}
\end{figure}

Consequently, we are reduced to consider the case where $I'$ has a non-trivial bundle $F$ and each endpoint is either attracting in the direction $F$ or attached
to a $f$-invariant unstable branch $\Gamma'$ (it has the type of a saddle-node or of a saddle with no {\color{black} reflection}).
The proposition holds for the branches $\Gamma'$ (this is our inductive assumption).
One deduces that there exists one or two trapping discs $D'$ containing the accumulation sets of these branches and included in $U$.
Taking the union with a tubular neighborhood $V$ of $I'$ and of the branches $\Gamma'$, one obtains a trapping disc $D\subset U$ which contains
the fixed point $p'$. See Figure~\ref{f.tube}.\end{proof}

\begin{corollary}\label{c.pixton}
Under the setting of Theorem~\ref{t.renormalize}, there exists a collection $\cD$ of Pixton discs $D$
(disjoint from $W^s_\DD(p)$) such that:
\begin{itemize}
\item[(a)] all the forward iterates $f^k(D)$ of discs $D\in \cD$ are included in $U$,
\item[(b)] any periodic orbit $\cO$ in the accumulation set of $\Gamma$ is contained in one $D\in\cD$,
\item[(c)] for any periodic orbit $\widetilde \cO$ in the accumulation set of $\Gamma$ and which is stabilized by
a fixed point, there exists a Pixton disc $D\in \cD$ which contains the unstable set of $\widetilde \cO$,
any periodic orbit $\cO$ decreasing chain-related to $\widetilde \cO$ and the unstable set of $\cO$.
\end{itemize}
\end{corollary}
\begin{proof}
For fixed points $p'$ accumulated by $\Gamma$, we either apply the Remark~\ref{r.pixton} (when $p'$ is a sink),
Lemma~\ref{l.highperiod} (when $p'$ is not a sink and has an eigenvalue smaller or equal to $-1$: it is then stabilized by the fixed point $q=p'$),
or Lemma~\ref{l.periodone} (in the other case).

For any periodic orbit $\widetilde \cO$ which is stabilized by a fixed point,
the Lemma~\ref{l.highperiod}  provides a Pixton disc $D\subset U$
which contains $\widetilde \cO$ in its interior and whose stable boundary $\gamma^s$ is contained in $W^s_\mathbb{D}(\widetilde \cO)$.
By the no-cycle theorem (Theorem~\ref{t.cycle}), the unstable manifold of $\widetilde \cO$ does not intersects $W^s_\mathbb{D}(\widetilde \cO)\setminus \widetilde \cO$.
This proves that the unstable set of $\widetilde \cO$ is included in $D$.

Let $\cO$ be a periodic orbit decreasing chain related to $\widetilde \cO$.
For any point $w\in \cO$, there exists $\widetilde w\in \widetilde \cO$
and a chain $C\subset D$ for an iterate of $f$ which contains $w$ and $\widetilde w$.
The closure of $C$ is fixed by an iterate of $f$ and is disjoint from $W^s_\mathbb{D}(p)$ (since
it is contained in $U$). As a consequence, it is disjoint from $W^u(p)$.
Since $C$ is connected and contained in the closure of a decorated region of $\widetilde \cO$,
it intersects at most one connected component of $\mathbb{D}\setminus W^s_\mathbb{D}(\widetilde \cO)$.
Since the boundary of $D$ is contained in $W^s_\mathbb{D}(\widetilde \cO)\cup W^u(p)$,
the chain $C$ is contained in $D$. This proves that $\cO$ is included in $D$.
By Proposition~\ref{p.unstable} the unstable set of $w$ intersects at most one component
of $\mathbb{D}\setminus W^s_\mathbb{D}(\widetilde \cO)$. It does not intersects $W^u(p)$ either.
Consequently, $W^u(\cO)$ is also included in $D$.
This gives the item (c).

From Proposition~\ref{p.decreasing-chain},
any periodic orbit $\cO$ in the accumulation set of $\Gamma$ which has period larger than $1$
and is not stabilized by a fixed point is decreasing chain related to a periodic orbit $\widetilde \cO$ stabilized by a fixed point.
Moreover from Proposition~\ref{p.chain-decreasing}, the stabilized orbit $\widetilde \cO$ is also accumulated by $\Gamma$.
The item (c) provides a Pixton disc $D\subset U$ which contains $\widetilde O$ and $\cO$.
This completes (b).

It remains to prove the item (a).
By construction, the discs $D$ are contained in $U$.
In the case $D$ is trapping, all its forward iterates are contained in $U$ also.
In the other cases, $D$ is obtained with Lemma~\ref{l.highperiod}: since $f$ is dissipative,
the volume of $f^k(D)$ is arbitrarily small for $k$ large; since the stable boundary $\gamma^s$ is contained in the stable
curve of a periodic orbit, its length $f^k(\gamma^s)$ gets arbitrarily small; since the unstable boundary $\gamma^u$ is contained in $\Gamma$,
one concludes that all the iterates $f^k(D)$, for $k$ large, are contained in $U$.
Up to {\color{black} replacing} $D$ by a larger iterate, Property (a) is satisfied.
\end{proof}

\subsection{Proof of Theorem~\ref{t.renormalize2}}
We first assume that $f$ preserves the orientation
and we consider the setting of the Section~\ref{ss.pixton}.
The accumulation set $\Lambda$ of $\Gamma$ may be covered by Pixton discs.

\begin{lemma}\label{l.covering}
Let us consider the family of Pixton discs $\cD$ obtained by Corollary~\ref{c.pixton}.
Then, any point $x$ in the accumulation set $\Lambda$ of $\Gamma$ has a backward iterate in the interior of one of the Pixton discs $D\in \cD$.
\end{lemma}
\begin{proof}
The proof of this lemma is done by contradiction.
If the conclusion does not hold,
the backward orbit of a point $x\in \Lambda$ accumulates on an invariant set $K\subset \Lambda$
that is disjoint from the interior of all the discs $D\in \cD$.
Then $K$ supports an ergodic measure $\mu$.
From the item (b) of Corollary~\ref{c.pixton}, this measure $\mu$ is non-atomic.
The Pesin theory associates a compact set $B\subset \operatorname{Support}(\mu)$
with $\mu(B)>0$ such that all the points $z$ in $B$ have a stable manifold $W^s_\DD(z)$
which separates $\DD$ and varies continuously with $z\in B$ for the $C^1$ topology.
We can thus choose $z\in B$ whose forward orbit is dense in the support of $\mu$ and two forward iterates $z',z''\in B$ close to $z$
and separated by $W^s_\DD(z)$.
In particular the region $R\subset \mathbb{D}$ bounded by $W^s_\DD(z')$ and $W^s_\DD(z'')$
does not contain any fixed point.
From the closing lemma (Theorem~\ref{t.measure local}), there exists a sequence $(w_k)$
of periodic points in $\Lambda$ converging to $z$. See picture~\ref{f.finiteness-cross}.
\begin{figure}
\begin{center}
\includegraphics[width=4cm,angle=0]{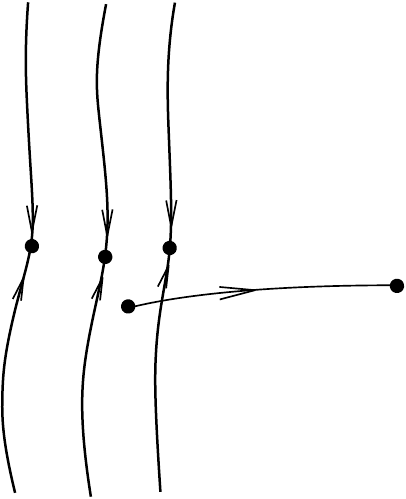}
\put(-80,120){\small $R$}
\put(-60,72){\small $z''$}
\put(-100,70){\small $z'$}
\put(-80,68){\small $z$}
\put(-10,50){\small $q$}
\put(-83,45){\small $w_k$}
\end{center}
\caption{Proof of Lemma~\ref{l.covering}.\label{f.finiteness-cross}}
\end{figure}
\medskip

We first assume that $w_k$ is stabilized by a fixed point $q$; since $w_k$ are in $\Lambda$, they are accumulated by $\Gamma$ and so by item (c) in Corollary~\ref{c.pixton}, there  a Pixton disc  $D\in \cD$  which contains the orbit of $w_k$ and its unstable set.
Since $w_k\in R$ and $q\not\in R$,
the unstable set of $w_k$ intersects the boundary of $R$, hence the stable manifold of
$z'$ or $z''$. Since the forward orbits of $z'$ and $z''$ equidistribute towards the measure $\mu$,
this implies that $\mu$ is supported on $D$. Since the support of $\mu$ is contained in $K$ which is disjoint with the interior of $D$ it follows that the support is contained in the boundary of $D$.
This is a contradiction since the orbit of any point in the boundary of $D$ converges in the future or in the past
towards $p$ or the orbit of $w_k$.
\medskip

When $w_k$ is not stabilized by a fixed point,
Proposition~\ref{p.decreasing-chain} implies that $w_k$ is decreasing chain-related to a periodic point $\widetilde w_k$ which is stabilized by
a fixed point $q_k$ and Proposition~\ref{p.chain-decreasing} implies that $\widetilde w_k$ also belongs to $\Lambda$.
In the case $\widetilde w_k$ belongs to $R$, the previous argument applies and gives a contradiction.
We are thus reduced to the case where $\widetilde w_k$ does not belong to $R$.

Let us consider a chain $C$ for an iterate of $f$ which contains $w_k$ and $\widetilde w_k$.
Let $D\in \cD$ be a Pixton disc associated to $\widetilde w_k$ as in Corollary~\ref{c.pixton} item (c):
in particular it contains the chain $C$. Since $C$ is connected and intersects both $R$ and its complement, there exists an unstable branch $\Gamma$
of a periodic point $w'_k\in C$ which intersects the stable curve of $z'$ or $z''$.
Since $\Gamma\subset D$, this implies as before that $\mu$ is supported on $D$ and gives a contradiction.
\end{proof}

We can now complete the proof of the theorem.

\begin{proof}[End of the proof of Theorem~\ref{t.renormalize2}]
Let $\Lambda$ be the accumulation set of $\Gamma$: this is an invariant compact set.
Let us consider the collection $\cD$ of Pixton discs given by Corollary~\ref{c.pixton}.
Let $V$ be the union of all the open sets
$\operatorname{Interior}(f^k(D))$ over $D\in \cD$ and over all $k\geq 0$. This is an open set satisfying $f(V)\subset V\subset U$.
By Lemma~\ref{l.covering}, any point in the accumulation set $\Lambda$ of $\Gamma$
has a backward iterate in $V$. Since $V$ is forward invariant and by compactness of $\Lambda$,
there exists a finite number of Pixton discs $f^k(D_n)$ such that the union of their interiors covers $\Lambda$.
The Remark~\ref{r.pixton}.(3) allows to replace any two of these discs which intersect by a single Pixton disc.
We repeat this inductively. Since $\Lambda$ is connected, one gets a Pixton disc $\widetilde D$ whose interior contains $\Lambda$.
We denote by $\widetilde \gamma^s, \widetilde \gamma^u$ its stable and unstable boundaries.

We modify $\widetilde D$ in order to obtain a Pixton disc satisfying some forward invariance.
One chooses $k$ large such that $f^k(\widetilde \gamma^u)$ is contained in a small neighborhood of $\Lambda$, hence in $\operatorname{Interior}(\widetilde D)$.
From the definition of the Pixton disc we also have $f^k(\widetilde \gamma^s)\subset \operatorname{Interior}(\widetilde D)$.
One applies Remark~\ref{r.pixton}.(3) again in order to build a Pixton disc $D$
which contains $\widetilde D\cup f(\widetilde D)\cup\dots\cup f^{k-1}(\widetilde D)$.
Since $f(\widetilde \gamma^s)\subset \widetilde D$, the stable boundary $\gamma^s$ of $D$ is included in $\widetilde \gamma^s$:
in particular, the $k$ first iterates of $\gamma^s$ are contained in $\operatorname{Interior}(D)$.
By construction, the $k-1$ first iterates of the boundary of $D$ are contained in $D$ and all the larger iterates are contained in $\operatorname{Interior}(\widetilde D)$.
This proves that the $k-1$ first iterates of $\gamma^u$ are contained in $D$ and the $k-$th iterate is included in $\operatorname{Interior}(D)$.
One deduces that $D$ is a Pixton disc whose interior contains $\Lambda$ and which furthermore
satisfies $f(D)\subset D$ and $f^k(D)\subset \operatorname{Interior}(D)$.

One finally modifies $D$ in order to build a disc $\Delta$ trapped by $f$:
for each $x$ in the boundary of $D$, one considers
the smallest integer $i\geq 1$ such that $f^i(x)\in \operatorname{Interior}(D)$;
one chooses small closed discs $D_{x,0}$, $D_{x,1}$, \dots, $D_{x,i-1}$
centered at $x$, $f(x)$, \dots, $f^{i-1}(x)$ respectively
such that $f(D_{x,j})\subset \operatorname{Interior}(D_{x,j+1})$ when $j<i-1$
and $f(D_{x,i})\subset \operatorname{Interior}(D)$.
By compactness, one selects finitely many points $x_1$, \dots, $x_m$
in the boundary of $D$, such that the union of the interior of the $D_{x_i,0}$ covers the boundary of $D$.
By construction, the union of $D$ with all the discs $D_{x_k,j}$ is a compact set $\widetilde \Delta$
whose image is contained in $\operatorname{Interior}(\widetilde \Delta)$.

{\color{black} As in Remark~\ref{r.pixton}, one can fill $\widetilde \Delta$ and obtain a disc $\Delta\supset \widetilde \Delta$} whose boundary is
{\color{black} contained} in the boundary of $\widetilde \Delta$.
In particular, $f(\Delta)\subset \operatorname{Interior}(\Delta)$.
We have thus obtained a trapping disc which contains $\Lambda$. From the item (b) of Corollary~\ref{c.pixton},
the trapping disc is disjoint from $W^s_\DD(p)$ as required.
The conclusion of the Theorem~\ref{t.renormalize} thus holds for $\Gamma$.

Since all the forward iterates of the Pixton discs $D\in \cD$ are included in $U$, the disc $\Delta$ may be chosen in a small neighborhood of $\overline U$.
This argument applied inductively
concludes the proof of Theorem~\ref{t.renormalize2} in the case $f$ is orientation preserving.

When $f$ is orientation reversing,
one first considers $f^2$ and gets a disc $\Delta_0$ disjoint from $W^s(p)$,
which contains the accumulation set of $\Gamma$ and is trapping for $f^2$.
The filled union $\Delta_1$ of $\Delta_0$ and $f(\Delta_0)$ has the same property (see Remark~\ref{r.pixton}),
but also satisfies $f(\Delta_0)\subset \Delta_0$.
Arguing as above, one can then modify $\Delta_1$ and get a disc $\Delta\supset \Delta_1$ contained in an arbitrarily small
neighborhood of $\Delta_1$ which is trapping for $f$.
\end{proof}

\subsection{Proof of Theorem~\ref{t.renormalize} and its consequences}

\begin{proof}[Proof of Theorem~\ref{t.renormalize}]
From Theorem~\ref{t.renormalize2}, the conclusion of Theorem~\ref{t.renormalize}
holds for the $f$-invariant unstable branches of the arcs $I_i\in \cI$.
One can easily conclude for the other $f$-invariant branches, \emph{i.e.,} for the unstable branches $\Gamma$ contained in $I_i$.
Indeed the accumulation set of such a branch belongs
to an isolated fixed arc $I\subset I_i$, which is disjoint from the unstable branch $\Gamma$ and bounded by an endpoint $p_i$ of $I_i$.
If $I_i$ has the type of a sink, it admits arbitrarily small neighborhoods that are trapping discs and the proposition follows.
Otherwise $p_i$ has a $f$-invariant unstable branch and we know from Theorem~\ref{t.renormalize2} that its accumulation set is contained in a trapping disc $\Delta_0$ disjoint from $I_i$.
One can then extend the disc $\Delta_0$ with a tubular neighborhood of $I$ and of the unstable branch of $p$: one then gets a
trapping disc $\Delta$
which contains $I$ (hence the accumulation set of $\Gamma$) as required.
\end{proof}

\begin{proof}[Proof of Corollary~\ref{c.isolated}]
Since $I$ is not reduced to a fixed point with eigenvalue $-1$, 
three types are possible (see Definition~\ref{d.type-arc}).

If $I$ has the type of a sink, the $\omega$-limit set of any point in an open neighborhood $W$
is a fixed point of $I$. Moreover by compactness, there exists $k\geq 1$ such that
$f^k(\overline W)\subset W$ and $\cap_{n\geq 0} f^n(W)=I$.
Hence the $\alpha$-limit set of any point in $W\setminus I$ is disjoint from $W$.

If $I$ has the type of a saddle with no {\color{black} reflection}, one applies Theorem~\ref{t.renormalize} and 
consider two trapping discs $V_1,V_2$ disjoint from $I$ which contain the accumulation sets
of the unstable branches of $I$.
The forward orbit of any point in a neighborhood $W$ of $I$
either intersect $V_1\cup V_2$ (in this case the $\omega$-limit set is contained in $V_1\cup V_2$
and is disjoint from $I$) or is contained in $I$ (it is a fixed point).
Let us define $W'=V_1\cup V_2\cup (\cup_n f^n(W))$.
By compactness, there exists $k\geq 1$ such that
$f^k(\overline W')\subset W'$.
Hence the $\alpha$-limit set of any point in $W$ is either disjoint from $W$ or contained in $W$.
Since $I$ is normally hyperbolic, any $\alpha$-limit set contained in $W$ is a fixed point of $I$.

If $I$ has the type of a saddle-node,
one {\color{black} considers} a trapping disc $V$ disjoint from $I$ which contains the accumulation set
of the (unique) unstable branch of $I$.
The forward orbit of any point in a neighborhood $W$ of $I$
either intersects $V$ or is contained in $I$.
One introduces $W'=V\cup (\cup_n f^n(W))$ and one argues as in the previous case.
\end{proof}

\subsection{Trapping discs and periodic measures}

As a byproduct of the previous arguments we obtain the following property which will be useful later.
\begin{proposition}\label{p.trapping-aperiodic}
Let $f$ be a mildly dissipative diffeomorphism of the disc with zero entropy,
and let $\mu$ be an aperiodic invariant measure.
Then for $\mu$-almost every point $z$ there exists $\varepsilon>0$
with the following property:
if $\Delta$ is a disc trapped by $f$ which contains a point
$\varepsilon$-close to $z$, then $\mu$ is supported on $\Delta$.
\end{proposition}
\begin{proof}
The argument appeared in the proof of Lemma~\ref{l.covering}.
The point $z$ is contained in a strip $R$ bounded by two stable manifolds $W^s_{loc}(z'), W^s_{loc}(z'')$
such that the forward orbits of $z'$ and $z''$ equidistribute towards $\mu$,
and such that $R$ does not contain any fixed point.

If the disc $\Delta$ contains a point close to $z$, it intersects $R$.
Since $\Delta$ is trapped, it also contain a fixed point. Consequently
$\Delta$ meets the stable manifold of $z'$ or $z''$, and therefore the forward orbit of a large iterate of this point.
This forward orbit equidistributes towards $\mu$.
This shows that $\mu$ is supported on $\Delta$.
\end{proof}

\section{Local renormalization}
\label{ss.local renormalization}
In this section we prove the Theorem~\ref{t.theoremA} about the existence of a renormalizable disc.
We also explain in Proposition~\ref{p.semi} how to renormalize inside a decorated region; this proposition is the main step to prove the global renormalization stated by Theorem~\ref{t.renormalize-prime} in Section \ref{s.renormalize}. 

\subsection{Renormalizable diffeomorphisms, proof of Theorem~\ref{t.theoremA}}
Let $f$ be a mildly dissipative diffeomorphism of the disc with zero entropy. We distinguish two cases, either all periodic points are fixed or not. In the first case, one have to prove that $f$ is generalized Morse-Smale; in the second, that there is a renormalizable domain.

\paragraph{\it First case: any periodic point of $f$ is fixed.}
For any $x\in \DD$, let $\mu$ be an ergodic measure supported on $\omega(x)$.
By the closing lemma (Theorem~\ref{t.measure local}), $\mu$ is supported on a fixed point $p$ and in particular, the forward orbit of $x$ accumulates on $p$.
If $x$ does not belong to the stable set of $p$ then $p$ has an unstable branch and the forward orbit of $x$ accumulate on that unstable branch.
By Theorem~\ref{t.renormalize} there exists a disc $\Delta$ which is trapped (by $f$ or by $f^2$) containing the accumulation set of the unstable branch and
disjoint from a neighborhood of $p$. In particular, $\omega(x)\subset \Delta\cup f(\Delta)$ and so $\omega(x)$ does not contain $p$; a contradiction.
We have shown that any forward orbit converges to a fixed point, thus $f$ is a generalized Morse-Smale.

\paragraph{\it Second case: there are periodic points with period larger than $1$.}
By Proposition~\ref{p.decreasing-chain}, there exists a stabilized periodic point $p$.

If $p$ has period $k>1$, one  considers the decorated region $V_p$ associated to $p$ and observe that one of the following cases holds.
\begin{itemize}
\item[\;2.a.] there exists an unstable branch of $p$ contained in $V_p$,
\item[\;2.b.] $p$ belongs to an arc which is fixed for $f^k$, contained in $\overline{V_p}$ and not reduced to $p$,
\item[\;2.c.] $p$ is a saddle-node of $f^k$.
\end{itemize}
In the case 2.a, we can apply again Theorem~\ref{t.renormalize} for $f^k$ and the unstable branch of $p$ that is contained in $V_p$:
this gives a disc $D\subset V_p$ which is trapped by $f^k$; since the decorated regions of the iterates of $p$ are disjoint,
the disc $D$ is disjoint from its $k-1$ first iterates.
In the first cases 2.b and 2.c, it follows immediately that there is a compact disc disjoint from its $k-1$ first iterates and mapped into itself by $f^k$.

If $p$ is a stabilized fixed point, it is not a sink. Let $V_p$ be one of its decorated regions. Only the cases 2.a and 2.b can occur.
In case 2.a, $p$ has two unstable branches that are exchanged by $f$; hence there exists a disc $D\subset V_p$ which is trapped by $f^2$.
In case 2.b, $p$ is accumulated by points of period $2$: one can then find an arc $I\subset V_p$ which is fixed by $f^2$ and disjoint from $f(I)$
and then a disc $D\subset V_p$ which is mapped into itself by $f^2$.

To summarize, in the second case we have found a disc $D$, disjoint from its first $k-1$ iterates and mapped into itself by $f^k$:
the diffeomorphism is renormalizable. The Theorem~\ref{t.theoremA} is now proved.\qed

\subsection{Renormalization inside decorated regions}

The following proposition provides the renormalization inside each decorated region, refining the trapped domain inside a decorated region of a periodic point $p$ into finite disjoint periodic trapping  domains that capture only the periodic points of larger period that are  decreasing chain related to $p$.
\begin{proposition}\label{p.semi}
Let $f$ be a mildly dissipative diffeomorphism of the disc with zero entropy,
$p$ be a stabilized periodic point with a decorated region $V$ and $k$ be the period of $V$.
Then, there exists a finite number of disjoint topological disks $D_1,..., D_m$ such that
\begin{enumerate}
 \item[a.] $D_1\cup\dots\cup D_m\subset V,$
 \item[b.] each $D_i$ is trapped by $f^k,$
 \item[c.] $D_1\cup\dots\cup D_m$ contains all the periodic points $q\in V$ that are decreasing chain related to $p$ in $V$ with period larger than $k$,
 \item[d.] conversely any periodic point in $D_1\cup\dots\cup D_m$ is decreasing chain related to $p$.
\end{enumerate}
\end{proposition}

\begin{proof}
From Corollary~\ref{c.orientation}, the map $f^k$ preserves the orientation.

Let $\mathcal{P}$ be the set of $q\in V$ which are
decreasing chain-related to $p$ such that
\begin{itemize}
\item[--] either the period of $q$ is larger than $k$,
\item[--] or the period of $q$ equals $k$ and $Df^k(p)$ has an eigenvalue less or equal to $-1$.
\end{itemize}
For each $\tau>k$, $\mathcal{P}(\tau)$
will denote the set of $q\in \mathcal{P}$ with period less or equal to $\tau$.

\begin{lemma}\label{l.first-disc}
Any $q\in \mathcal{P}$ belongs to a disc $\Delta_q\subset V$ trapped by some iterate of $f^k$.
\end{lemma}
\begin{proof}
The case where $q$ is a sink is clear. One can thus assume that there exists $\tau > k$
such that $f^\tau(q)=q$ and $Df^\tau(q)$
has an eigenvalue larger or equal to $1$.

We consider a  finite collection $\cJ$ of disjoint isolated arcs fixed by $f^\tau$ which contains all the points that are fixed by $f^\tau$.
Since $p$ has an unstable branch in $\DD\setminus \overline V$,
we also may assume that each arc is either contained in $\overline V$ or disjoint from it.
As in the statement of Theorem~\ref{t.renormalize2}, we
write  $J>J'$ if $J$ has an unstable manifold fixed by $f^\tau$ which accumulates on $J'$
and we consider all the sequences $J^0>J^1>\dots>J^n$ in $\cJ$ such that $q\in J^0$.

\begin{claim}
The periodic points (different from $p$) in all the arcs $J^i$ are decreasing chain related to $p$.
\end{claim}
\begin{proof}
The property holds for $J^0$ since $J^0$ contains $q\in \mathcal{P}$ and is included in $\overline V$.
If the property holds for $J^0,\dots,J^i$,
then by Proposition~\ref{p.unstable} the unstable branches of $J^i$ are contained in $\overline V$.
By Definition~\ref{d.decreasing-chain},
their unstable sets only accumulate on periodic points that are decreasing chain related to $p$
or coincide with $p$. In particular $J^{i+1}$ is included in $\overline V$ and then
satisfies the inductive property.
\end{proof}

\begin{claim}\label{c.local}
The unstable branches of $J^n$ do not accumulate on $p$.
\end{claim}
\begin{proof}
The claim is proved inductively for all arcs $J^0,J^1,\dots,J^n$.
Hence one will assume that for any arc $J^0,\dots,J^{n-1}$
the unstable branches do not accumulate on $p$ and by contradiction that
one of the unstable branches of $J^n$ accumulates on $p$.

Up to removing some intervals $J^i$, $0< i<n$, from the sequence $J^0,\dots, J^n$,
one can also assume that $J^i>J^{i'}$ does not hold when $i+1<i'$.

We first prove that for each $i\in\{0,\dots,n-1\}$,
the unstable branches of $J^i$ avoid one of the components of $\DD\setminus W^s_{\DD}(J^{i+1})$:
otherwise one unstable branch $\Gamma$ of $J^i$ would accumulate on the unstable branches of $J^{i+1}$
and from Proposition~\ref{p.transitive}, the accumulation set of $\Gamma$ contains the accumulation set of the unstable branches of $J^{i+1}$.
When $i<n-1$, this implies $J^i>J^{i+2}$, contradicting our choice on the sequence $J^0,\dots,J^n$.
When $i=n-1$, our assumption on $J^n$ implies that $\Gamma$ accumulates on $p$, contradicting
our assumption that the unstable branches of $J^{n-1}$ do not accumulate on $p$.
\medskip

Recall that from Corollary~\ref{c.orientation}, the map $f^k$ preserves the orientation.
The property obtained in the last paragraph together with Proposition~\ref{p.heteroclinic2} imply that for each $i\in\{0,\dots,n-1\}$,
the period of the unstable branches of $J^i$ equals the period of the unstable branches of $J^{i+1}$.

By definition of $\mathcal{P}$,
either $q\in J^0$ has period larger than $k$, or has period $k$ and the unstable branches
of $J^0$ are exchanged by $f^k$. In any case the unstable branches of $J^0$ have period larger than $k$.
Consequently the same property holds for each arc $J^i$.

But by assumption an unstable branch of $J^n$ accumulate on $p$
and is contained in $\overline V$.
Since $f^k$ is orientable, Proposition~\ref{p.heteroclinic} implies that the unstable branches of $J^n$
have period $k$. This is a contradiction and the claim is proved.
\end{proof}

Theorem \ref{t.renormalize2} applied to $f^\tau$ provides discs that are trapped by $f^\tau$,
that are contained in $V$ (thanks to Claim~\ref{c.local}), and that contain the accumulation sets
of the unstable branches of $J^0$.
Consider a neighborhood of $J^0$. Iterating forward, it may be glued to the trapped discs.
This defines a disc contained in $V$ that is trapped by $f^\tau$.
The Lemma~\ref{l.first-disc} is proved.
\end{proof}

\begin{lemma}\label{l.trapping-tau}
For any $\tau>k$ there exists a finite number of disjoint $f^k-$trapped discs whose union $U_\tau$ is included in $V$ and contains $\mathcal{P}(\tau)$.
Moreover $U_{\tau}\subset U_{\tau'}$ when $\tau\leq \tau'$.
\end{lemma}
\begin{proof}
Observe that $\mathcal{P}(\tau)$ is compact.
From Lemma~\ref{l.first-disc}, there exist finitely many discs $\Delta_q\subset V$ that are trapped by some iterates of $f^k$
and contain all the points of $\mathcal{P}(\tau)$.
Up to {\color{black} modifying} slightly their boundaries if necessary, one can assume that they are transverse.
As a consequence the union $U_\tau$ of the discs is a finite disjoint union of submanifolds with boundary
which are trapped by an iterate of $f^k$.
Since $f$ is dissipative, the components of $U_\tau$ are topological discs.
The construction is performed inductively on $\tau$, so that $U_{\tau}\subset U_{\tau'}$ when $\tau\leq \tau'$.

It remains to prove that each component of $U_\tau$ is trapped by $f^k$ (instead of an iterate of $f^k$).
Let us assume by contradiction that this is not the case: there exists $k<l\leq \tau$  and a disc $\Delta_q\subset U_\tau$
that only contains points decreasing chain related to $p$ with period larger or equal to $l.$
As in the proof of Lemma~\ref{l.first-disc}, one considers a finite collection of disjoint isolated arcs fixed by $f^\tau$
and compute their contribution to the indices of $f^k$ and $f^l$ in the decorated region $V$.

The arc $I_0$ which contains $p$ is fixed by $f^k$ and $index(I_0,V,f^k)=index(I_0,V,f^k)$.
Similarly, the total contribution of the arcs contained in a disc trapped by $f^k$ equals $1$ (both for the maps $f^k$ and $f^l$).
But the total contribution of the arcs contained in a disc trapped by $f^l$ and not by $f^k$
equals $1$ for the map $f^l$ and $0$ for the map $f^k$.
Consequently the index of $f^l$ in $V$ is larger than the index of $f^k$ in $V$.
This is a contradiction, since from Proposition~\ref{p.local-lefschetz}, the indices of $f^k$ and $f^l$ in the decorated region $V$ coincide (and are equal to $1/2$).
\end{proof}

\begin{lemma}\label{c.U}
The set $\mathcal P$ is contained in one of the regions $U_\tau$.
\end{lemma}
\begin{proof}
If the conclusion of the lemma does not hold, one can find a sequence $\tau_k\to +\infty$
such that $U_{\tau_k+1}\setminus U_{\tau_k}$ contains a periodic $f^k$-orbit $\cO_k$ supported on $\mathcal{P}$.
Up to {\color{black} taking} a subsequence, $(\cO_k)$ converges towards a $f^k$-invariant compact set $K\subset \overline V$.

Note that $K$ is aperiodic. Indeed any $x\in K$ is accumulated by a sequence of points $(q_n)$ of $\mathcal{P}$.
If $x$ were periodic, then Corollary~\ref{c.isolated2} would imply that the periods of the points $q_n$ is bounded.
Since the $q_n$ are decreasing chain-related to $p$, this would imply that $x$ has the same property and belongs to $\mathcal P$.
This is a contradiction since $K$ is disjoint from the increasing union $\cup U_\mathcal \tau$ which contains $\mathcal{P}$.

Let $\mu$ be an ergodic $f^k$-invariant measure supported on $K$.
By construction, for $\mu$-almost every point $x$ there exist a component $D$ of $U_\tau$
which contains a point arbitrarily close to $x$ as $\tau\to +\infty$. Since $D$ is trapped by $f^k$,
the Proposition~\ref{p.trapping-aperiodic} implies that $\mu$ is supported on $D$. A contradiction
since $K$ is disjoint from $U_\tau$.
\end{proof}

We have shown that $\mathcal{P}$ is contained in the union $U_\tau$ of finitely many
disjoint disks (denoted by $D_1,\dots, D_m$) in $V$ that are trapped by $f^k$.
To conclude, we need to prove that for each disc $D_i$,
any periodic point in $D_i$ is decreasing chain related to $p.$

Note that the iterates of $D_i$ do not meet the stable manifold of the orbit of $p$ (otherwise the trapping property would imply that
$D_i$ contain $p$, a contradiction). In particular $D_i$ can not contain any fixed point.
Observe also that $D_i$ does not contain a stabilized periodic point (since one of its unstable branches
accumulates on a fixed point and has to be contained in $D_i$).
Hence by Proposition~\ref{p.decreasing-chain}, each periodic point in $D_i$ is decreasing chain related to some stabilized periodic point.
The next lemma asserts that they are necessarily decreasing chain related to $p$.

\begin{lemma}\label{c.disc2}
Any periodic point $q\in D_i$ is decreasing chain related to $p$.
\end{lemma}
\begin{proof}
The proof is done by contradiction: we assume that $D_i$ contains $q$ which is decreasing chain related to a stabilized periodic point $p'$
which is different from $p$.

Since $D$ does not contain the stabilized point $p'$ and is trapped by $f^k$, it is disjoint from $W^s_{\mathbb{D}}(p')$.
In particular, it is contained in a decorated region $V'$ of $p'$.
Note that either $V\subset V'$ or $V'\subset V$. We will assume that the first case occurs (the proof in the second case is similar).
By Definition~\ref{d.decreasing-chain},
there exists a chain $C$ for an iterate of $f$ that is contained in $\overline V'$  intersects $V$ and $p'$. Hence $C$ meets $W^s_\mathbb{D}(p)$:
there is an unstable branch in $C$ which intersects the stable manifold of $p$ implying that $p$ is decreasing chain-related to $p'$.
The Corollary~\ref{c.unique-stabilized} then gives the contradiction.
\end{proof}

The proof of Proposition~\ref{p.semi} is now complete.
\end{proof}

\section{Finiteness of the set of stabilized periodic points}
\label{ss.finitness}

In order to prove the global renormalization (Theorem~\ref{t.renormalize-prime} in the next section), one {\color{black} needs} to show that the number of stabilized orbits is finite. This section is devoted to  prove the following:

\begin{theorem}\label{t.finite}
Let $f$ be a mildly dissipative diffeomorphism of the disc with zero topological entropy.
Then the set of its stabilized periodic orbits is finite.
\end{theorem}

Before proving the theorem we associate to any stabilized orbit a filtrating region.
 
\begin{proposition}\label{p.separation} 
For any stabilized periodic orbit $\cO$ there exist two topological disks $U_\cO\subset \widehat U_\cO$ that are trapped by $f$ such that
\begin{enumerate}
\item[--]  $\widehat U_\cO\setminus U_\cO$  contains $\cO$ and any periodic orbit decreasing chain related to $\cO$,
\item[--] any periodic orbit in $\widehat U_\cO\setminus U_\cO$ is either $\cO$ or is decreasing chain related to $\cO$.
\end{enumerate}
Moreover if $\cO$ is contained in a trapping disc $\Delta$,
then $U_\cO$ can be chosen in $\Delta$.
\end{proposition}

\begin{proof}
We first assume that the period $k$ of $\cO$ is larger than $1$.
Let $p\in \cO$, let $V$ be the decorated region associated to $p$.
Theorem \ref{t.renormalize2} applied to the stabilized unstable branch $\Gamma$ of $p$
associates a trapping disc $U_\cO$ which contains the accumulation set of $\Gamma$.
Note that if $\cO$ is contained in a trapping disc $\Delta$,
then $U_\cO$ can be chosen to be included in $\Delta$.

Now, we deal with the decorated region. Let $D_1,\dots,D_n$ be the $f^k$-trapping discs given by Proposition \ref{p.semi}.
Let $I_1,\dots,I_\ell$ be isolated $f^k$-fixed arcs containing all the periodic points in $V$
that are decreasing chain-related to $p$ and not in $\cup D_i$.
Since any periodic point close to $p$ is contained in $V$ (by Corollary~\ref{c.isolated2}),
we can assume that the arcs $I_j$ are contained in $\overline V$.
By Definition~\ref{d.chain} of the chains and the invariance of the arcs,
each periodic point in $I_j$ either coincides with $p$ or is decreasing chain related to $p$.

Each $I_j$ admits a neighborhood $O_j$ which is a topological disc
such that if the $\omega$-limit set of a point $x$ by $f^k$
intersects $\overline O_j$, then it is contained in $I_j$ (by Corollary~\ref{c.isolated}).

Let $\widetilde U_\cO$ be the forward invariant set defined by the union of the forward iterates of $O_j$, of $D_i$ and of $U_\cO$.
It can be written as the union of finitely many connected sets:
\begin{itemize}
\item[--] the disc $U_\cO$,
\item[--] the trapping discs $f^m(D_i)$ for $0\leq m<k$ and $1\leq i\leq n$,
\item[--] the connected unions $f^m(T_j):=f^m(O_j)\cup f^{m+k}(O_j)\cup f^{m+2k}(O_j)\cup\dots$ for $1\leq j\leq \ell$
and $0\leq m<k$.
\end{itemize}
By definition of the decreasing chain relation,
for each $0\leq m<k$, the union of the interior of the sets $f^m(D_i)$ and $f^m(T_j)$ for $1\leq i\leq n$ and $1\leq j\leq \ell$
is connected. One set $O_{j_0}$ contains the point $p$ so that each $f^m(T_{j_0})$ for $0\leq m<k$ intersects $U_\cO$.
This proves that the interior of $\widetilde U_\cO$ is connected.

By construction, the set $\widetilde U_\cO$ is forward invariant.
Since $f$ contracts the volume, the interior of $\widetilde U_\cO$
is simply connected, hence homeomorphic to the open disc.

\begin{lemma}\label{l.trapped-finite}
If $O_1$,\dots, $O_\ell$ are sufficiently small neighborhoods of $I_1$,\dots, $I_\ell$, 
the $\omega$-limit set of any point in $\operatorname{Closure}(\widetilde U_\cO)$
is contained in $\operatorname{Interior}(\widetilde U_\cO)$.
\end{lemma}
\begin{proof}
Since there is no cycle, one can enumerate the arcs $I_j$ in a way that
the unstable branches of $I_j$ do not accumulate on $I_{j'}$ when $j\geq j'$.

\begin{claim}\label{c.unstable}
Let us consider an unstable branch $\Gamma$ of $I_j$.
Then the $\omega$-limit set by $f^k$ of any point of $\Gamma$ belongs to
$D_1\cup \dots\cup D_n\cup I_1\cup\dots\cup I_{j-1}$.
\end{claim}
\begin{proof}
Let us consider a $f^k$-invariant ergodic measure $\mu$ supported on the accumulation set $L$ of $\Gamma$.
If $\mu$ is supported on a $f^k$-periodic orbit, this orbit is decreasing chain-related to $\cO$,
hence is contained in $D_1\cup \dots\cup D_n\cup I_1\cup\dots\cup I_\ell$.
By our choice of the indices of the arcs $I_1,\dots,I_\ell$,
the periodic orbit is contained in $D_1\cup \dots\cup D_n\cup I_1\cup\dots\cup I_{j-1}$ and the claim follows
in this case.

If $\mu$ is aperiodic, the closing lemma (Theorem~\ref{t.measure local}) implies that there exist
periodic points in $L$ which accumulate on the support of $\mu$.
These periodic points are also decreasing chain-related to $\cO$
and are contained in $D_1\cup \dots\cup D_n\cup I_1\cup\dots\cup I_{j-1}$.
Passing to the limit one deduces that $\mu$ is also contained
in $D_1\cup \dots\cup D_n\cup I_1\cup\dots\cup I_{j-1}$.
\end{proof}

Since the discs $U_\cO$ and $D_i$ are trapped by some iterates of $f$,
it is enough to prove that the $\omega$-limit set under $f^k$ of any point in $\operatorname{Closure}(O_j)$ is contained in the union
$$U_\cO \cup (D_1\cup ,\dots, \cup  D_n) \cup (I_1\cup \dots \cup I_{j}).$$
This is proved inductively.
We assume that the property holds for any $j'<j$ and
consequently we can suppose that the closure of
$$\Delta_{j-1}:=U_\cO \cup (D_1\cup ,\dots, \cup  D_n) \cup (T_1\cup \dots \cup T_{j-1})$$
is mapped into its interior by some iterate of $f^k$.

Let us consider any point $x$ in $\operatorname{Closure}(O_j)$.
If its $\omega$-limit set belongs to $I_j$, the inductive property holds trivially.
Otherwise, $x$ has a forward iterate close to a neighborhood $W$ of a fundamental domain of the unstable branches of $I_j$.
By choosing the neighborhood $O_j$ of $I_j$ small enough, the neighborhood $W$ can be chosen arbitrarily small
and by the Claim~\ref{c.unstable}, any point in $W$ has a forward iterate by $f^k$ in the interior of $\Delta_{j-1}$.
We have thus proved that if the $\omega$-limit set of $x$ is not contained in $I_j$,
then a forward iterate of $x$ by $f^k$ belongs to $\Delta_{j-1}$ as required.
The inductive property is proved, which concludes the proof of Lemma~\ref{l.trapped-finite}.
\end{proof}

From the previous lemma,
one can thus modify $\widetilde U_\cO$ near its boundary
(as explained in the proof of Theorem~\ref{t.renormalize2})
and define a topological disc $\widehat U_\cO$
which is trapped by $f$.

The limit orbits in $\widehat U_\cO$ and $\widetilde U_\cO$
are the same:
for any point in $\widehat U_\cO$,
the $\omega$-limit set is contained in one of the trapping discs
$D_i$, or in $U_\cO$, or in one of the arcs $I_j$.
With Proposition \ref{p.semi}, this shows that any periodic orbit in $\widehat U_\cO\setminus U_\cO$
either coincide with $\cO$ or is decreasing chain related to $\cO$.
\medskip

When the period is $1$, the proof is almost the same.
The point $p$ is fixed and has two decorated regions $V,V'$, each one of period $2$.
Working with $V$,
one builds $f^2$-trapping discs $D_1,\dots,D_n$,
isolated $f^2$-fixed arcs $I_1,\dots,I_\ell$,
and neighborhoods $O_1,\dots, O_\ell$  as before.
\end{proof}

\begin{proof}[Proof of Theorem~\ref{t.finite}]
We distinguish two types of stabilized periodic orbits $\cO$.
\begin{itemize}
\item[--] First type: $\cO$ admits trapping discs $U_\cO\subset \widehat U_\cO$ as in Proposition~\ref{p.separation}
such that the set of stabilized periodic orbits in $U_\cO$ is finite.
\item[--] Second type: for any trapping discs $U_\cO\subset \widehat U_\cO$ associated to $\cO$ as in Proposition~\ref{p.separation}
there exist infinitely many stabilized periodic orbits in $U_\cO$.
\end{itemize}

The following shows that the set of stabilized periodic orbits of the first type is finite. Later we will prove that there are no stabilized periodic orbits of the second type.
\begin{claim}\label{c.subsequence}
Let $(\cO_n)$ be a sequence of distinct stabilized periodic orbits
and let $U_{\cO_n}\subset \widehat U_{\cO_n}$ be trapping discs associated to $\cO$ as in Proposition~\ref{p.separation}.
Up to {\color{black} considering} a subsequence, the following property holds:
$\cO_{m}\subset U_{\cO_n}$ for each $m>n$.
\end{claim}
\begin{proof}
Up to {\color{black} taking} a subsequence, one can assume that the sequence $(\cO_{n})$ converges for the Hausdorff topology towards an invariant
compact set $K$. From the fact that periodic points are isolated from periodic points of large period (Corollary~\ref{c.isolated2})
it follows that  $K$ does not contain any periodic point.
Let $\mu$ be an ergodic measure supported on $K$.
It is aperiodic and by Proposition~\ref{p.trapping-aperiodic} for any $n$ large the disc $\widehat U_{\cO_n}$
contains the support of $\mu$. In particular the stabilized orbits
$\cO_m$ for $m>n$ large intersect $\widehat U_{\cO_n}$.
All the stabilized periodic points (different from the points of $\cO_n$)
that are contained in $\widehat U_{\cO_n}$ are contained in $U_{\cO_n}$,
hence the orbits $\cO_m$ meet and (by the trapping property) are contained in $U_{\cO_n}$.
Up to {\color{black} extracting} the sequence $(\cO_n)$, one can assume that
for any $n$, the set $K$ is contained in $U_{\cO_n}$.

Let us fix the integer $n$.
We have obtained that for $m$ large $\cO_m$ is contained in $U_{\cO_n}$.
Up to {\color{black} extracting} the subsequence $(\cO_m)_{m>n}$,
one can assume that all the orbits $\cO_m$ for $m>n$ are contained in $U_{\cO_n}$.
By induction, one builds an extracted sequence which satisfies the required property for all integers $m>n$.
\end{proof}
To conclude, it is enough to show that there are not  stabilized points of the second type. 
Let us assume now by contradiction that this is not the case.
One builds inductively a sequence of stabilized periodic orbits of the second type $(\cO_n)$
and trapping discs $(U_n)$ satisfying for each $n\geq 1$:
\begin{itemize}
\item[--] $U_{n}\subset U_{n-1}$,
\item[--] $\cO_{n}\subset U_{n-1}\setminus U_{n}$,
\item[--] $U_n$ contains infinitely many stabilized periodic orbits,
\item[--] the period of $\cO_{n}$ is minimal among the periods of the stabilized periodic orbits of the second type
contained in $U_{n-1}$.
\end{itemize}
After  $\cO_n$ and $U_n$ have been built,
we choose $\cO_{n+1}$ as a stabilized periodic orbits of the second type
contained in $U_{n}$ which minimizes the period.
By Proposition~\ref{p.separation}, there exists trapping discs $U_{n+1}\subset \widehat U_{n+1}$
associated to $\cO_n$ and one can require that $U_{n+1}$ is contained in the trapping $U_n$.
In particular $\cO_{n+1}\subset U_{n}\setminus U_{n+1}$
Since $\cO_{n+1}$ is of the second type, $U_{n+1}$ contains infinitely many stabilized periodic orbits
as required.
\medskip

Once the sequences $(\cO_n)$ and $(U_n)$ have been built,
one considers (up to {\color{black} extracting} a subsequence) the Hausdorff limit $K$ of $(\cO_n)$.
As in the proof of Claim~\ref{c.subsequence},
it supports an ergodic measure $\mu$ which is aperiodic.
The intersection of the discs $U_n$ defines an invariant cellular set $\Lambda$ that contains
$K$. The closing Lemma~\ref{t.measure local} implies that $\Lambda$ contains periodic points $(p_k)$ with arbitrarily large period
which accumulate on a point $x$ of $K$.
{\color{black} Each point $p_k$ may either belong to a stabilized periodic orbit (but since its period is large, it will be a stabilized orbit of the second type),
or is decreasing chain related to a stabilized periodic orbit.
We thus have to consider the following cases:}

\begin{itemize}
\item[--] \emph{Some $p_k$ belongs to stabilized periodic orbits of the second type.}
{\color{black} On the one hand the point $p_k$ belongs to all the sets $U_n$.
On the other hand} the minimal period of stabilized periodic orbits of the second type contained in $U_n$ goes to $+\infty$ as $n\to +\infty$,
{\color{black} hence is larger than the period of $p_k$ for $n$ large.} This is a contradiction.

\item[--] \emph{Some $p_k$ is decreasing chain related to a stabilized periodic orbit $\cO$ of the second type.}
Each trapping disc $U_n$ contains a fixed point $q$ and there exists a decorated region $V$ of $\cO$
which does not contain $q$. Up to {\color{black} replacing} $p_k$ by one of its iterates, one can assume $p_k\in V$.
This shows that $U_n$ meets $V$ and its complement, hence intersects the stable set of $\cO$.
Since $U_n$ is a trapping disc, it contains $\cO$.
{\color{black} On the one hand we have shown that the stabilized periodic orbit of the second type $\cO$ belongs to all the sets $U_n$.
On the other hand} the minimal period of stabilized periodic orbits of the second type contained in $U_n$ goes to $+\infty$ as $n\to +\infty$.
{\color{black} As before} this is a contradiction.

\item[--]
\emph{ All the points $p_k$ are decreasing chain related to stabilized periodic orbits of the first type.}
Since the number of this type of stabilized periodic orbits is finite,
one can assume that the $p_k$ are all decreasing chain related to the same stabilized periodic orbit $\cO$.
Let us consider trapping discs $U_\cO\subset \widehat U_\cO$ as in Proposition~\ref{p.separation}.
All the $p_k$ are contained in the filtrating region $\widehat U_\cO\setminus U_\cO$.
Taking the limit, $K$ meets that region. In particular, the orbits $\cO_n$ for $n$ large
also meet that region. This is a contradiction since the orbits $\cO_n$
are stabilized and the region $\widehat U_\cO\setminus U_\cO$ contains only one stabilized periodic orbit
(the orbit $\cO$).
\end{itemize}

In all the cases we found a contradiction.
This ends the proof of Theorem~\ref{t.finite}.
\end{proof}

\section{Global renormalization}
\label{s.renormalize}

We now prove a strong version of Theorem~\ref{t.theoremA} and its
Corollary~\ref{c.dichotomy0}.
The proof of Corollary~\ref{c.infinitely-renormalizable} is then immediate and left to the reader.

\newcounter{theorembis}
\setcounter{theorembis}{\value{theorem}}
\setcounter{theorem}{0}
\renewcommand{\thetheorem}{\Alph{theorem}'}

\begin{theorem}\label{t.renormalize-prime}
For any mildly dissipative diffeomorphism $f$ of the disc with zero topological entropy,
there exist $\ell\geq 0$, some disjoint topological discs $D_1, \dots, D_\ell$,
some integers $k_1,\dots,k_\ell\geq 2$ such that:
\begin{enumerate}

\item[--] each $D_i$ is trapped by $f^{k_i}$,

\item[--] the discs $f^m(D_i)$ for $1\leq i\leq \ell$ and $0\leq m<k_i$ are pairwise disjoint,

\item[--] for each $D_i$ there is a stabilized orbit such that each iterate of $D_i$ is contained in a decorated region of an iterate of the stabilized orbit,

\item[--] $f$ is generalized Morse-Smale in the complement of the union of the iterates of the disks  $(D_i)$ with periodic points of period smaller and equal to  to $\max\{1, k_1,\dots, k_\ell\}$.

\end{enumerate}
In particular the interior $W$ of the union $\bigcap _i \bigcap _{m\geq 0} f^m(D_i)$ is a filtrating open set. 
\end{theorem}

\setcounter{theorem}{\value{theorembis}}
\renewcommand{\thetheorem}{\Alph{theorem}}

\subsection{Global renormalization: proof of Theorem~\ref{t.renormalize-prime}}
We apply Proposition \ref{p.semi} and associate to each stabilized periodic orbit $\cO_i$ some
discs $D_{i,1},\dots,D_{i,\ell_i}$ that are trapped by $f^{k_i}$ where $k_i$ is the
period of the decorated regions associated to $\cO_i$.
By construction each $D_{i,j}$ is contained in a decorated region of $\cO_i$
and all the periodic points decreasing chain-related to $\cO_i$
and with period larger than $k_i$ belong to the orbit of the $D_{i,j}$.

The discs $D_{i,j}$ and $D_{i',j'}$ associated to different orbits $\cO_i,\cO_{i'}$ are disjoint by Lemma \ref{c.disc2}, Proposition~\ref{p.decreasing-chain} and Corollary~\ref{c.unique-stabilized}.
Since the number of stabilized orbits is finite (Theorem~\ref{t.finite}), we get the two first items.

Note that any periodic point which does not belong to the $D_{i,j}$ is either fixed, or stabilized,
or decreasing chain-related to a stabilized point with the same period.
Hence its period is bounded by $\max\{1, k_1,\dots, k_\ell\}$.

Let $x$ be any point whose $\omega$-limit set does not belong to a trapped disc.
The limit set supports an ergodic measure $\mu$
This measure cannot be aperiodic since the closing lemma would produce a periodic orbit
with large period outside the discs $D_{i,j}$.
Hence the limit set contains a periodic orbit and by 
Corollary \ref{c.isolated} coincides with the periodic orbit.
The Theorem~\ref{t.renormalize-prime} is now proved.
\qed

\subsection{Infinite renormalization: proof of Corollary~\ref{c.dichotomy0}}
By Theorem~\ref{t.renormalize-prime},
the dynamics of $f$ reduces to a generalized Morse-Smale dynamics in a filtrating set
$\mathbb{D}\setminus \overline W$.
If $W=\emptyset$, the diffeomorphism $f$ is generalized Morse-Smale and Corollary~\ref{c.dichotomy0} holds.

Each connected component of $W$ is a topological disc $\Delta$ which is trapped by an iterate $f^k$ of $f$;
moreover the restriction of $f^k$ to $\Delta$ is a mildly dissipative diffeomorphism.
One may thus apply Theorem~\ref{t.renormalize-prime}
inside each of these discs. Arguing inductively, one gets a new decomposition of the dynamics
into a generalized Morse-Smale part and discs that are eventually trapped after
a return time which increases at each step of the induction.
If $f$ is not generalized Morse-Smale, the induction does not stop and $f$ is infinitely renormalizable.
Corollary~\ref{c.dichotomy0} follows.
\qed

\section{Chain-recurrent dynamics}
\label{s.odometers}

We now describe in detail the dynamics of a mildly dissipative diffeomorphism with zero entropy
and prove Corollary~\ref{c.structure}.

\subsection{Generalized odometers}

\begin{proposition}\label{p.odometer}
Let $f$ be a mildly dissipative diffeomorphism of the disc,
$(D_i)$ be a sequence of topological discs and $(k_i)$ be a sequence of integer such that
\begin{itemize}
\item[--] $D_i$ is trapped by $f^{k_i}$ and disjoint from its $k_i-1$ first iterates,
\item[--] $D_i\subset D_{i+1}$ and $k_i<k_{i+1}$ for each $i$.
\end{itemize}
Then the intersection of the sets $f^{k_1}(D_i)\cup f^{k_i+1}(D_i)\cup\dots \cup f^{2k_i-1}(D_i)$ is
a chain-recurrence class $\cC$ which is a generalized odometer.
In particular it supports a unique invariant measure $\mu$ and for $\mu$-almost every point $x$,
the connected component of $x$ in $\cC$ is {\color{black} reduced to $\{x\}$}.
\end{proposition}
\begin{proof}
Let us denote by $\cC$ the intersection of the sets $f^{k_1}(D_i)\cup f^{k_i+1}(D_i)\cup\dots \cup f^{2k_i-1}(D_i)$.
It is a compact invariant set.

For each $i$, let $(\cO_i,h_i)$ be the cyclic permutation on the set with $k_i$ elements.
The inverse limit of the systems $(\cO_i,h_i)$ defines an odometer $(\mathcal{K},h)$ on the Cantor set.
The sets $D_i,\dots, f^{k_i-1}(D_i)$ define a partition of $\cC$ and induce a factor map on
$(\cO_i,h_i)$, hence a semi-conjugacy $p\colon (\cC,f)\to (\mathcal{K},h)$.
Since the connected components of $\cC$ coincide with the decreasing intersections
of sequences of the form $f^{m_i+k_i}(D_i)$, the preimages $p^{-1}(x)$ coincide with the connected components
of $\cC$.

Let $\nu$ be the unique invariant probability measure on $(\mathcal{K},h)$ and let $\mu$
be an ergodic probability on $(\cC,f)$ such that $p_*(\mu)=\nu$.
The following claim shows that $\cC$ is a generalized odometer.

\begin{claim}
For $\nu$-almost every point $z\in\mathcal{K}$, the preimage $p^{-1}(z)$ is a singleton.
\end{claim}
\begin{proof}
Let us consider a set $X\subset \cC$ with positive $\mu$-measure which is a hyperbolic block,
such that $W^s_\mathbb{D}(x)$ varies continuously with $x\in X$.
One can also find a disc $D\subset \mathbb{D}$ which contains $\cC$ such that $f(D)\subset \interior(D)$
and whose boundary is transverse to the manifolds $W^s_\mathbb{D}(x)$, for $x\in X$.
Let $B\subset X$ be a subset with positive measure
of points having arbitrarily large backward iterates $f^{-n}(x)\in X$ that are accumulated by points of $X$ on both
components of $\mathbb{D}\setminus W^s_\mathbb{D}(f^{-n}(x))$.

Let us choose $\varepsilon>0$.
For each $x\in B$, there exist backward iterates $f^{-n}(x)\in X$ such that
$f^n(W^s_\mathbb{D}(f^{-n}(x)))$ has diameter smaller than $\varepsilon/2$.
As in the proof of Theorem~\ref{t.measure local}, one can thus find a rectangle $R$ with diameter smaller than $\varepsilon$, which contains $x$
and whose  boundary is contained in $\partial f^n(D) \cup W^s_\mathbb{D}(x')\cup W^s_\mathbb{D}(x'')$
for two forward iterates $x',x''$ of $x$.
For $i$ large enough, the disc $f^{m_i+k_i}(D_i)$ which contains $x$ is contained in $f^n(D)$,
and does not meet the iterates $x',x''$, nor their stable manifolds.
Consequently, $f^{m_i+k_i}(D_i)$ has diameter smaller than $\varepsilon$.
Since $\varepsilon>0$ has been chosen arbitrarily,
the connected component of $\cC$ containing $x\in B$ is reduced to $x$.

Since $x$ is arbitrary in $B$ which has positive measure and since $\mu$ is ergodic,
one deduces that for $\mu$-almost every $x$, the connected component of $x$
in $\cC$ (which coincides with $p^{-1}(p(x))$) {\color{black} is reduced to a unique point}.
Since $p_*(\mu)=\nu$, the claim follows.
\end{proof}

The claim and the characterization of the connected components of $\cC$ prove that for $\mu$-almost
every point $x$, the connected component of $x$ in $\cC$ {\color{black} is reduced to a unique point}.

Since the discs $D_i$ are trapped by $f^{k_i}$, any chain-recurrence class which meets $\cC$
is contained in $\cC$. For any $\varepsilon$, let us consider $i$ and an iterate $f^{m_i+k_i}(D_i)$
with diameter smaller than $\varepsilon$.
Any forward and backward orbit in $\cC$ intersects $f^{m_i+k_i}(D_i)$,
showing that $\cC$ is chain-transitive.
This implies that $\cC$ is a chain-recurrence class.
\end{proof}

\subsection{Dynamics on chain-recurrence classes: proof of Corollary~\ref{c.structure}}
Let us apply inductively Theorem~\ref{t.renormalize-prime}, as in the proof of Corollary~\ref{c.dichotomy0}.
We obtain a decreasing sequence $(W_n)$ of trapped open sets such that the dynamics in each $\DD\setminus \overline {W_n}$
is generalized Morse-Smale.
By Proposition~\ref{p.MS}, the chain-recurrent set in that region is the set of periodic points.
Since their period is bounded, the chain-recurrence classes $\cC$ of $f$ in the complement of any $W_k$
can be written as a disjoint union $\cC=C\cup\dots\cup f^{m-1}(C)$, where $C$ is a connected component of the set
of periodic points and $f^m(C)=C$.
Corollary~\ref{c.structure} is proved when $f$ is generalized Morse-Smale.

It remains to describe the dynamics when $f$ is infinitely renormalizable,
that is, when the sequence $(W_n)$ is infinite.
By construction the infimum of the periods of the periodic points in $W_n$
gets arbitrarily large as $n$ goes to infinity.
Up to {\color{black} replacing} $(W_n)$ by the sequence $(f^n(W_n))$, one can assume that the intersection
$\cap W_n$ is an invariant compact set $\Lambda$.

By construction, each connected component of $W_n$ is a topological disc $D$
which is trapped by an iterate $f^m$ and disjoint from its $m-1$ first iterates.
Moreover, $m$ goes uniformly to infinity as $n\to +\infty$ (since the periods in $D$
get large when $n$ increases).
One deduces from Proposition~\ref{p.odometer} that $\Lambda$ is a union of generalized odometers
that are chain-recurrences classes of $f$.
This ends the proof of Corollary~\ref{c.structure}.
\qed

\section{Set of periods}
\label{s.period}

In the present section we provide the proof of Theorem~\ref{t.period} and Corollary~\ref{c.period}.

\subsection{Proof of Theorem~\ref{t.period}}
We consider an infinitely renormalizable $C^r$ diffeomorphism,
since for generalized Morse-Smale systems the conclusion of Theorem~\ref{t.period} holds immediately by taking $W=\emptyset$.
From Theorem~\ref{t.renormalize-prime}, if one assumes by contradiction that Theorem~\ref{t.period} is not satisfied,
there exist a sequence of topological discs $(D_i)$, integers $k_i,\tau_i\to +\infty$, periodic orbits $(\cO_i)$ such that
\begin{itemize}
\item[--] $D_i$ is trapped by $f^{k_i}$ and disjoint from its $k_i-1$ first iterates,
\item[--] $\cO_{i}$ is contained in $D_i\cup f(D_i)\cup\dots\cup f^{k_i-1}(D_i)$ and has period $\tau_i$,
\item[--] $\cO_i\cap f^m(D_i)$ is a stabilized orbit of $f^{k_i}$ for $0\leq m<k_i$ and has period $\tau_i/k_i\geq 3$.
\end{itemize}

\begin{claim}
For each $i,m$, $\cO_{i}\cap f^{m}(D_i)$ is a decorated orbit of $f^{k_i}$ in $\DD$.
\end{claim}
\begin{proof}
The intersection $\cO_{i}\cap f^{m}(D_i)$ is a decorated orbit in $D=f^{m}(D_i)$: this means that for any points
$x,y$ in the orbit, there exists a path in $D$ which connect them and is disjoint from the local stable manifolds
$W^s_D(z)$ in $D$ of the other points of the orbit. In particular there exists a path in $\DD$ which connect them and is disjoint
from the local manifolds $W^s_\DD(z)$ in $\DD$, proving that $\cO_{i+1}\cap f^{m}(D_i)$ is decorated in $\DD$
\end{proof}

Let us choose $\alpha\in (0,\min(1, r-1))$ and $\varepsilon \in (0, 1/4)$.
Theorem~\ref{t.stable} associates $\gamma\in (0,1)$.
By Theorem~\ref{t.renormalize-prime}, there exists a nested sequence of topological discs
$\widehat D_i$ that are periodic and trapped with periods $\widehat k_i\to+\infty$
such that $D_i\subset \widehat D_i$.
By Proposition \ref{p.odometer} the intersection of the sets $\widehat D_i\cup f(\widehat D_i)\cup\dots \cup f^{\widehat k_i-1}(\widehat D_i)$
is a chain-recurrence class $\cC$ which is a generalized odometer. In particular it does not contain any periodic points
and it supports a unique invariant probability {\color{black} measure} $\mu$. Proposition~\ref{p.gamma-strong} implies that $f$ is $\gamma$-dissipative on $\cC$,
hence on the domains $\widehat D_i\cup f(\widehat D_i)\cup\dots \cup f^{\widehat k_i-1}(\widehat D_i)$ for $i$ large enough.
Theorem~\ref{t.stable} provides a compact set $A$ such that $W^s_\mathbb{D}(x)$ exists and varies continuously with $x\in A$ in the $C^1$ topology
and ${\color{black} \mu(A)}>3/4$ for any invariant probability measure supported on a neighborhood of $\cC$. In particular the orbits $\cO_i$ have at least $3\tau_i/4$ iterates in $A$.

By Proposition~\ref{p.odometer}, for $\mu$-almost every point $x$, the connected component of $x$ in $\cC$
is reduced to $\{x\}$.
This implies that for any $\delta>0$ and for $i$ large enough,
at least $3\widehat k_i/4$ discs in the family $\widehat D_i\cup f(\widehat D_i)\cup\dots \cup f^{\widehat k_i-1}(\widehat D_i)$
have diameter smaller than $\delta$.

The number of discs $f^{m}(\widehat D_i)$ ($0\leq m<\widehat k_i$) which contain at most $2$ points in $\cO_{i+1}\cap A$
is smaller than $(\tau_{i}/\widehat k_i-2)^{-1}\operatorname{Card}(\cO_{i}\setminus A)$, hence than $\tau_i/4$.
Consequently there exists a disc $f^{m+k_i}(\widehat D_i)$ with diameter smaller than $\delta$ which contains at least $3$
points $x,y,z$ of $A\cap \cO_i$.
Since the three points are close, the local stable manifolds $W^s_\DD(x)$, $W^s_\DD(y)$, $W^s_\DD(z)$
are close for the $C^1$-topology. In particular there are coordinates in the disc such that
the three curves are graphs over one of the coordinate axis. This implies that one of the stable manifolds
separates the two other ones in $\DD$. This is a contradiction since the orbit $\cO_{i}\cap f^{m}(D_i)$ of $f^{k_i}$ is decorated.
Theorem~\ref{t.period} is proved.
\qed

\subsection{Proof of Corollary~\ref{c.period}}
Theorem~\ref{t.period} implies that there exists $m\geq 1$, a finite number of topological discs
$D_1,\dots,D_\ell$ and integers $m_1,\dots,m_\ell$ such that
\begin{itemize}
\item[--] the discs $f^k(D_i)$ with $1\leq i\leq \ell$ and $0\leq k<m_i$ are pairwise disjoint,
\item[--] each disc $D_i$ is trapped by $f^{m_i}$,
\item[--] the set $F$ of periodic points in the complement of $\cup_{i,k} f^k(D_i)$ is finite,
\item[--] each $f^{k_i}|_{D_i}$ is infinitely renormalizable and each renormalization disc $\Delta\subset D_i$
is contained in a sequence of renormalization discs $\Delta_0=\Delta\subset \Delta_1\subset\dots\subset \Delta_s=D_i$
such that the period of $\Delta_{j}$ is the double of the period of $\Delta_{j+1}$.
\end{itemize}
Theorem~\ref{t.renormalize-prime} shows that the set of periods of each diffeomorphism $f^{k_i}|_{D_i}$ coincides with $\{2^n,n\geq 0\}$.
This shows that the set of periods of $f$ coincides with
$$F\cup \{m_i.2^n,\: 1\leq i\leq \ell \text{ and } n\in\NN\}.$$
Corollary~\ref{c.period} follows.\qed

\section{Dynamics close to one-dimensional endomorphisms}\label{ss.close-endo}
In this section, we prove Theorem~\ref{t.small jacobian}.

\subsection{Extension of one-dimensional endomorphisms}
From now on, and to keep the approach described in \cite{CP} we consider extensions of a one-dimensional endomorphisms which slightly differ from~\eqref{e.extension},
but  which work both for the interval and the circle: given a one-dimensional manifold $I$
(the circle $S^1$ or the interval $(0, 1)$), a $C^2$ map $h : I \to I$ isotopic to the identity (such that
$h(\partial I) \subset \interior(I)$ in the case of the interval), $\eps>0$ small and $b \in (-1, 1)$ even smaller, we
get a map $f_b$ on $\DD := I \times (-\eps, \eps)$ defined by
$$f_b : (x, y) \to (h(x) + y, b(h(x)-x + y)).$$
Indeed for any $y\in \RR$ close to $0$ and any $x \in  h(I)$, the sum $x + y$ is well defined and, since $h$
is isotopic to the identity, the difference $h(x) - x$ belongs to $\RR$.
Note that the Jacobian is constant and equal to $b$. When $|b|> 0,$ the map $f_b$ is a diffeomorphism onto its image. When $b = 0$ the image $f_0 (\DD)$ is contained in $I \times  \{0\}$ and the
restriction of $f_0$ coincides with $h \times \{0\}.$

Theorems 1 and 2 in~\cite{CP} assert that for $|b|>0$ small enough, the map $f_b$ is mildly dissipative and that the same property holds for any diffeomorphism.
The diffeomorphism~\eqref{e.extension} that is presented in the introduction can be handled in the same way.
Indeed for $b = 0$, the map $f_0$ is
an endomorphism which contracts the curves $h(x) + y = {\color{black} \operatorname{const}}$ to a point: these curves are
analogous to strong stable manifolds. One can check moreover that, for any ergodic measure
which is not supported on a sink, the points in a set with uniform measure are far from the
critical set, implying that these curves cross the domain $I \times (-\eps, \eps)$. For
$|b| > 0$, the control of the uniformity of the stable manifolds ensures that for points in a set
with uniform measure has local stable manifolds close to the curves $h(x) + y = {\color{black} \operatorname{const}}$.

\subsection{Parallel laminations}
The proof of Theorem~\ref{t.small jacobian} follows from the property that
for points on a large set, the stable manifolds are ``parallel", \emph{i.e.,} do not contain decorated configurations.

\begin{definition}\label{d.paralell} A family of $C^1-$curves $\gamma\colon [0,1]\to \DD$ is \emph{parallel}  if: 
\begin{enumerate}
 \item[--] every curve separates the disc: $\gamma(\{0,1\})\subset \partial \DD$ and $\gamma((0,1))\subset \interior(\DD)$;
 \item[--]  given three of them, there is one that separates the other two.
\end{enumerate}
\end{definition}

\begin{proposition}\label{p.improving-onedim} Given $\delta>0$ and a $C^2-$endomorphism $h$ of the interval, there exists $b_0>0$ such that for any
$b$ with $0<|b|<b_0$, and for any diffeomorphism $g$ in a $C^2$-neighborhood of
$f_b$, there exists a compact set $S$ such that:

\begin{enumerate}
\item[--] each $x\in S$ has a stable manifold and the family $\{W^s_\DD(x),x\in S\}$ is parallel;
\item[--] for any ergodic measure $\mu$ which is not supported on a sink,  $\mu(S)>1-\delta$.
\end{enumerate}
\end{proposition}
\begin{proof}
We follow and adapt the proof of Theorem~2 in~\cite{CP}.
Let $K>\|Dh\|$ and fix $L\gg K$.
We introduce four numbers, depending on $b$:
$$\sigma(b):=L.|b|,\quad \tilde \sigma(b)=|b|/5K,\quad
\tilde \rho(b):=|b|/25K^2,\quad \rho=L^2.|b|.$$
Consider the set $\mathcal{A}(f_b)$ of points $x$
having a direction $E\subset T_x\mathbb{D}$ satisfying
\begin{equation*}
\forall n\geq 0,\;\;\; {\tilde \sigma}^{n}\leq \|Df^n(x)_{|E}\|\leq  \sigma^{n},\;
\text{ and }\; {\tilde \rho}^n\leq \frac{\|Df^n(x)_{|E}\|^2}{|\det Df^n(x)|}\leq \rho^{n}.
\end{equation*}
The proof of Lemma 4.4 in~\cite{CP} shows that by taking $L$ large enough,
then $\mu(\mathcal{A}(f_b))>1-\delta/2$ for any invariant ergodic probability $\mu$ which is not supported on a sink. Let us choose a small neighborhood $U$ of
the critical set $\{x, Dh(x)=0\}$.
Then the measure $\mu(U\times (-\eps,\eps))$ is smaller than $\delta/2$
and on its complement, the angle between the stable manifolds
$W^s(x)$ for $x\in \mathcal{A}(f_b)\setminus U\times (-\eps,\eps)$ is bounded away from zero.

Having chosen $|b|$ small enough,
the leaves $W^s_\DD(x)$ for $f_b$ are $C^1$-close to affine segments
(Theorem~1 in~\cite{CP})
and are uniformly transverse to the horizontal
for points $x$ in the set $S:=\mathcal{A}(f_b)\setminus U\times (-\eps,\eps)$,
defining a parallel lamination for $f_b$.

When $|b|$ is small enough, the mild dissipation is robust
(see Theorem~1 of~\cite{CP}) and the property extends to diffeomorphisms
$g$ that are $C^2$-close to $f_b$.
\end{proof}

\subsection{Proof of Theorem~\ref{t.small jacobian}}
Let us choose a diffeomorphism $g$ as in the statement of Theorem~\ref{t.small jacobian}.
Having chosen $g$ in a small neighborhood of a diffeomorphism $f_b$, with $|b|$ small,
ensures that 
Proposition \ref{p.improving-onedim} holds for some $\delta\in(0,1/3).$

In particular at least $2/3$ of the iterates of any stabilized periodic orbit meets the set $S$: the parallel property then implies that the period of any stabilized periodic orbit is
$1$ or $2$.
The Theorem~\ref{t.renormalize-prime} gives disjoint renormalization discs with period $2$, such that any periodic orbit in the complement has period $1$ or $2$.

Let us consider any one of the obtained renormalization domains $D$
and the induced diffeomorphism $g^2|_D$.
Note that if $O\subset D$ is a stabilized orbit of $g^2$, then both $O$ and $g(O)$
are  stabilized periodic orbits of $g^2$ in $\DD$. By Proposition \ref{p.improving-onedim}, at least $2/3$ of the iterates of $O$, or $g(O)$ belong to $S$.
The parallel property then implies that the period of $O$ under $g^2$ is $1$ or $2$.
The Theorem~\ref{t.renormalize-prime} gives smaller
disjoint renormalization discs with period $4$, such that any periodic orbit in the complement is fixed by $g^4$.
Arguing inductively, one deduces that there exists
renormalization discs of period $2^n$ such that
the periodic orbits in the complement are fixed by $g^{2^n}$.
Consequently, any periodic orbit is fixed by some iterate $g^{2^n}$,
hence has a period which is a power of $2$.
\qed

\section{Dynamics of the H\'enon map}
In this section we prove Corollary~\ref{c.henon}.

\subsection{Reduction to a dissipative diffeomorphism of the disc}
The dynamics of a dissipative H\'enon map is the same as the dynamics of a dissipative diffeomorphism of the disc.

\begin{proposition}\label{p.reduction}
For any H\'enon map $f_{b,c}$ with $|b|\in(0,1)$, there exists:
\begin{itemize}
\item[--] a smooth dissipative diffeomorphism of the disc $g\colon \DD\to g(\DD)$,
\item[--] a topological disc $\Delta\subset \RR^2$,
\item[--] a homeomorphism $h\colon \Delta\to \DD$,
\item[--] a decomposition $\DD=\DD_1\cup\DD_2$ into {\color{black} two} half discs,
\item[--] a decomposition $\Delta=\Delta_1\cup \Delta_2$ with $\Delta_i=h(\DD_i)$,
\end{itemize}
such that:
\begin{enumerate}
\item $g(\DD_2)\subset \operatorname{interior}(\DD_2)$ and
any forward orbit of $g|_{\DD_2}$ converges to a fixed point $p_0$;
\item $f_{b,c}(\Delta_1)\subset \operatorname{interior}(\Delta)$ and
$f_{b,c}=h\circ g\circ h^{-1}$ on $\Delta_1$;
\item the forward orbit of any $x\in \Delta_2$ under $f_{b,c}$ escapes to infinity: $\|f_{b,c}^n(x)\|\underset{\tiny n\to+\infty}\longrightarrow \infty$;
\item the backward orbit of any $x\in \Delta\setminus f_{b,c}(\Delta)$  under $f_{b,c}$ escapes to infinity: $\|f_{b,c}^n(x)\|\underset{\tiny n\to-\infty}\longrightarrow \infty$;
\item any $f_{b,c}$-orbit which does not meet $\Delta$ escapes to infinity in the past and future.
\end{enumerate}
\end{proposition}
\begin{remark}\label{r.reduction}
When $|b|<1/4$, the diffeomorphism $g$ is mildly dissipative.
Indeed let us consider an ergodic measure $\mu$ of $g$ which is not supported on a sink.
From item (1), it is supported on $\DD_1$.
From items (2), $\nu:=h^{-1}_*(\mu)$ is an ergodic measure for $f_{b,c}$ which is not supported on a sink.
From Wiman theorem (see~\cite{CP} Theorem 2 and Section 4.2), for $\nu$-almost every point $x$,
each stable curve of $x$ is unbounded in $\RR^2$, hence intersects the boundary of $\Delta_1$.
From item (2) again, one deduces that for $\mu$-almost every point $y$, each stable curve intersects the boundary of $\DD_1$
and cannot meet $\DD_2$ since its forward orbit is not attracted by $p_0$. One deduces that each stable curve of $y$
meets the boundary of $\DD$, proving that $g$ is mildly dissipative.
\end{remark}

\begin{proof}[Proof of Proposition~\ref{p.reduction}]
Let us fix $R>0$ and define $D=D_1\cup D_2$ where
$$D_1=[-R,R]\times\left[-\sqrt{|b|}R,\sqrt{|b|}R\right],\quad
D_2=[R, 2R^2]\times\left[-\sqrt{|b|}R,\sqrt{|b|}R\right].$$
One checks easily that if $R$ is large enough,
$$f_{b,c}(D_1)\subset \operatorname{interior}(D)\quad \text{ and } \quad
f_{b,c}(D_1\cap D_2)\subset \operatorname{interior}(D_2).$$
One then defines an embedding $\tilde f\colon D\to \operatorname{interior}(D)$ such that
\begin{itemize}
\item[--] $\tilde f|_{D_1}=f_{b,c}|_{D_1}$,
\item[--] $\tilde f(D_2)\subset \operatorname{interior}(D_2)$,
\item[--] any forward orbit of $\tilde f|_{D_2}$ converges to a fixed sink.
\end{itemize}
One approximates $D$ by a disc $\Delta\subset D$ with a smooth boundary:
$D\setminus \Delta$ is contained in a small neighborhood of $\partial\Delta$ and such that $\{R\}\times \RR$ decomposes $\Delta$ in two half discs $\Delta_1$, $\Delta_2$. One then chooses a diffeomorphism $h\colon \Delta\to \DD$ and set $g=h\circ \tilde f\circ h^{-1}$.
The items (1) and (2) of the proposition are then satisfied.

Note that the domain $U^+=\{(x,y),\; |x|\geq R \text{ and } |y|\leq \sqrt{|b|}|x|\}$ is mapped into itself and that if $R$ has been chosen large enough then
the image $(x',y')$ of $(x,y)\in U^+$ satisfies $x'>2|x|$. Consequently the forward orbit of any point in $U^+$ escapes to infinity.

The inverse map is $f_{b,c}^{-1}\colon (x,y)\mapsto (-y/b, x-y^2/b^2-c).$
As before, the domain $U^-=\{(x,y),\; |x|\geq R \text{ and } |y|\geq \sqrt{|b|}|x|\}$ is mapped into itself by $f_{b,c}^{-1}$
and that if $R$ has been chosen large enough then
the preimage $(x',y')$ of $(x,y)\in U^-$ satisfies $|y'|>2|y|$. Consequently the backward orbit by $f_{b,c}$ of any point in $U^-$ escapes to infinity.
This concludes the proof of items (3), (4), (5).
\end{proof}

\subsection{Proof of Corollary~\ref{c.henon}}
Let $g$ be the diffeomorphism given by Proposition~\ref{p.reduction}.
Since the topological entropy of $f_{b,c}$ vanishes, the same holds for $g$.
Moreover by Remark~\ref{r.reduction}, $g$ is mildly dissipative.

From the items (3) and (5) of Proposition~\ref{p.reduction}, any forward orbit by $f_{b,c}$ which does not escape to infinity
accumulates in a subset $K$ of $\Delta_1$. The image $h(K)$ by the conjugacy is the limit set of a forward orbit of $g$.
It is contained in a chain-recurrence class of $g$. With Corollary~\ref{c.structure},
one deduces that the forward orbit of $f_{b,c}$ converges to a periodic orbit or to a subset of a generalized odometer.
From items (4) and (5), a similar conclusion holds for backward orbits.

The periodic set of $f_{b,c}$ is included in $\Delta_1$ and is conjugated by $h$ to the periodic set of
$g$, once the fixed point $p_0$ has been excluded. Hence the set of periods of $f_{b,c}$
can be described from the set of periods of $g$. By Corollary~\ref{c.period}, it has the structure~\eqref{e.period}.
\qed

\subsection{Final remark: trapping discs for the H\'enon map}
We propose an alternative proof to Corollary~\ref{c.henon}
in the case where the H\'enon map $f=f_{b,c}$ is orientation-preserving (\emph{i.e.,} $b\in (0,1/4)$).
Indeed the following proposition holds.
When the dynamics of $f$ is not trivial, we can find a trapping disc for $f$
(whose boundary is the union of two compact graphs contained in the stable and unstable manifolds of a fixed point)
and apply Corollaries~\ref{c.structure} and~\ref{c.period} directly to $f_{b,c}$.

\begin{proposition}\label{p.quadritomie}
Any H\'enon map $f=f_{b,c}\colon (x,y)\mapsto (x^2+c+y,-bx)$ with jacobian $b\in (0,1)$ satisfies one of the following properties:
\begin{itemize}
\item[a)] There is no fixed point. All the orbits escape to infinity in the past and the future.
\item[b)] The fixed points belong to a simple curve $\gamma\colon [0,+\infty)\to \RR^2$   whose image is invariant and which satisfies $\gamma(t)\underset{t\to+\infty}\longrightarrow \infty$. Any forward (resp. backward) orbit either converges to a fixed point or escape to infinity.
\item[c)] There exist a topological disc $D\subset \RR^2$ trapped by $f$ and a simple curve $\gamma\colon \RR\to \RR^2$
whose image is invariant and which satisfies $\gamma(t)\underset{t\to+\infty}\longrightarrow \infty$ and $\gamma((-\infty,0])\subset D$
such that all the fixed points are contained in $D\cup \gamma(\RR)$.
Any backward (resp. forward) orbit either converges to a fixed point, escapes to infinity or is (eventually) contained in $D$. See Figure~\ref{f.construction-trapped}.
\item[d)] There is a fixed point with a homoclinic orbit. The topological entropy is positive.
\end{itemize}
\end{proposition}
\begin{proof}[Sketch of the proof]
If there is no fixed point, by Brouwer theorem, any orbit $(f^n(x))_{n\in\ZZ}$ converges to infinity when $n\to \pm \infty$
and the proposition holds.
If there exists a fixed point with a homoclinic orbit, then the topological entropy is positive by Lemma~\ref{l.heteroclinic-cycle}.
In the following we will thus assume that there exists at least one fixed point and that there is no homoclinic orbit.

If $(x,y)$ is fixed, then $x^2+c-(1+b)x=0$. Hence there exists at most two fixed points.
We denote by $q=(x_q,y_q)$ the point whose first coordinate satisfies
$x_q=\frac{1+b}{2}+ \sqrt{\frac{(1+b)^2}{4}-c}.$ Note that it is a saddle-node fixed point when $4c=(1+b)^2$,
and a hyperbolic saddle with positive eigenvalues otherwise.

\begin{claim}\label{c.unbounded-graph}
The right unstable branch of $q$ is a graph over the interval $(x_q,+\infty)$ in the first coordinate.
It is contained in the wandering set and the forward orbit of any point in a neighborhood escapes to infinity.
\end{claim}
\begin{proof}
At a point $z=(x,y)$, one considers the direction $(1,v_z):=(1,-x+\sqrt{x^2-b})$.
Let us assume $x\geq 0$ and that the image $z'=(x',y')$ satisfies $x'\geq x$. Note that $v_{z}\leq v_{z'}\leq 0$.
If the direction $(1,v)$ at $z$ satisfies $-x\leq v\leq v_z$, then the image $(1,v')$ at $z'$ satisfies
$$-x'\leq -x\leq v\leq v'\leq v_z\leq v_{z'}.$$
One can thus obtain the unstable manifold of $q$ by iterating forwardly a local half graph at $q$ whose tangent directions $(1,v)$ satisfy
$-x< v<v_z$ at any of its points $z$ with $x\geq x_q$. The iterates still satisfy these inequalities.
The sequence of iterated graphs converges towards the right unstable branch of $q$, which is a graph.
Since there is no fixed point satisfying $x>x_q$, it is a graph over $(x_q,+\infty)$.

Any point $x$ in a neighborhood of the unstable branch of $q$ has a forward iterate in the domain $U^+$
introduced in the proof of Proposition~\ref{p.reduction}, hence escape to infinity in the future.
In particular the first coordinate is strictly increasing and $x$ admits a wandering neighborhood.
\end{proof}

\begin{claim}\label{l.concavity}
On the domain where it is a graph, the local stable manifold of $q$ is convex.
On the domain where it is a graph, the local unstable manifold of $q$ is concave.
\end{claim}
In the following, we denote by $(x_s,x_q)$ and $(x_u,x_q)$ the maximal open domains where the left stable and left unstable branches of $q$ are graphs.
\begin{proof}
Let us consider a graph $\Gamma=\{(t,\varphi(t))\}_{t\in I}$
whose image  is a graph $\Gamma'=\{(s,\varphi'(s))\}$  such that $f|_\Gamma$ preserves the orientation on the first projection.
The slope $v(t)=D\varphi(t)$ has an image with slope $v'(t)=-b/(2t+v(t))$.
The derivative of the slope of the image is $Dv'(t)=\frac{b(2+Dv(t))}{(2t+v(t))^2}$.
If $Dv(t)$ is positive, the same holds for $Dv'(t)$. If $Dv'(t)$ is negative, the same holds for $2+Dv(t)$, hence for $Dv(t)$.

One deduces that: if $\Gamma$ is concave, then $\Gamma'$ is also concave;
if $\Gamma'$ is a convex, then $\Gamma$ is also convex.
\end{proof}

\paragraph{\it First case: the left unstable or the left stable branch of $q$ is a graph.}
If the left unstable branch of $q$ is a graph, it is bounded by the second fixed point $p$. Hence $W^u(q)\cup\{p\}$ is an invariant closed half line
containing the two fixed points. The domain $U=\RR^2\setminus (W^u(q)\cup\{p\})$ is homeomorphic to a plane.
By Brouwer theorem, any orbit which does not belong to that line escape to infinity in the domain $U$ when $n\to \pm \infty$.
Together with the Claim~\ref{c.unbounded-graph}, this implies that the forward (resp. backward) orbit either belongs to the stable (resp. unstable) manifold of a fixed point, or converges to infinity in $\RR^2$. If the left stable branch of $q$ is a graph, the union of the left stable branch and of the right stable branch is an invariant closed half line and a similar argument holds.

\paragraph{\it Second case: the unstable manifold of $q$ is not a graph and $x_s\leq x_u$.}
The left unstable branch is not a graph: there exists a point $z_u\in W^u(q)$ with a vertical tangent space and (by Lemma~\ref{l.concavity}) the unstable arc connecting the points $z_u$ and $q$ is a concave graph $\gamma^u$ over an interval $(x_u,x_q)$. The (local) left stable branch of $q$ is a graph over a maximal interval $(x_s,x_q)$.
The tangent map at $q$ is $Df(q)=\begin{pmatrix} 2x_q & 1 \\ -b & 0 \end{pmatrix}$, hence the stable graph is above the
unstable graph. Since there is no homoclinic point, the two graphs are disjoint.

One can build a Jordan domain $\Delta$ by considering the union of
the unstable arc $\gamma^u$, a vertical segment $\gamma^v$ and a stable arc $\gamma^s$ above $(x_u,x_q)$,
see Figure~\ref{f.construction-trapped}.

We claim that $f(\Delta)\subset \Delta$. Indeed $f(\gamma^u)$ does not intersect $\gamma^s$, as explained above.
It does not crosses $\gamma^v$ either, since $f^{-1}(\gamma^v)$ is a subset of a convex graph $\{(t, {\color{black} \operatorname{const}}-t^2),t\in \RR\}$
which is tangent to the concave graph $\gamma^u$ at the point $f^{-1}(z_u)$.
Similarly the horizontal segment $f(\gamma^v)$ does not cross the convex graph $\gamma^s$ since one of its endpoints is below the convex graph $\gamma^s$ and the other one is on the graph.

One considers a domain $D$ which is bounded by curves close to (but disjoint from) $\gamma^u, \gamma^v,\gamma^s$.
The inclination lemma implies that $D$ is a trapped disc, see Figure~\ref{f.construction-trapped}. By construction it contains a fundamental domain of the left unstable branch of $q$.
One deduces that $\RR^{2}\setminus (W^u(q) \cup (\cap_n f^n(D)))$ is homeomorphic to the plane and does not contain any fixed point.
Brouwer theorem implies that any forward (resp. backward) orbit either is contained in the stable (resp. unstable) manifold of $q$,
or intersects $D$ (resp. is contained in $\cap_n f^n(D)$), or escapes to infinity.
\begin{figure}
\includegraphics[width=5cm,angle=0]{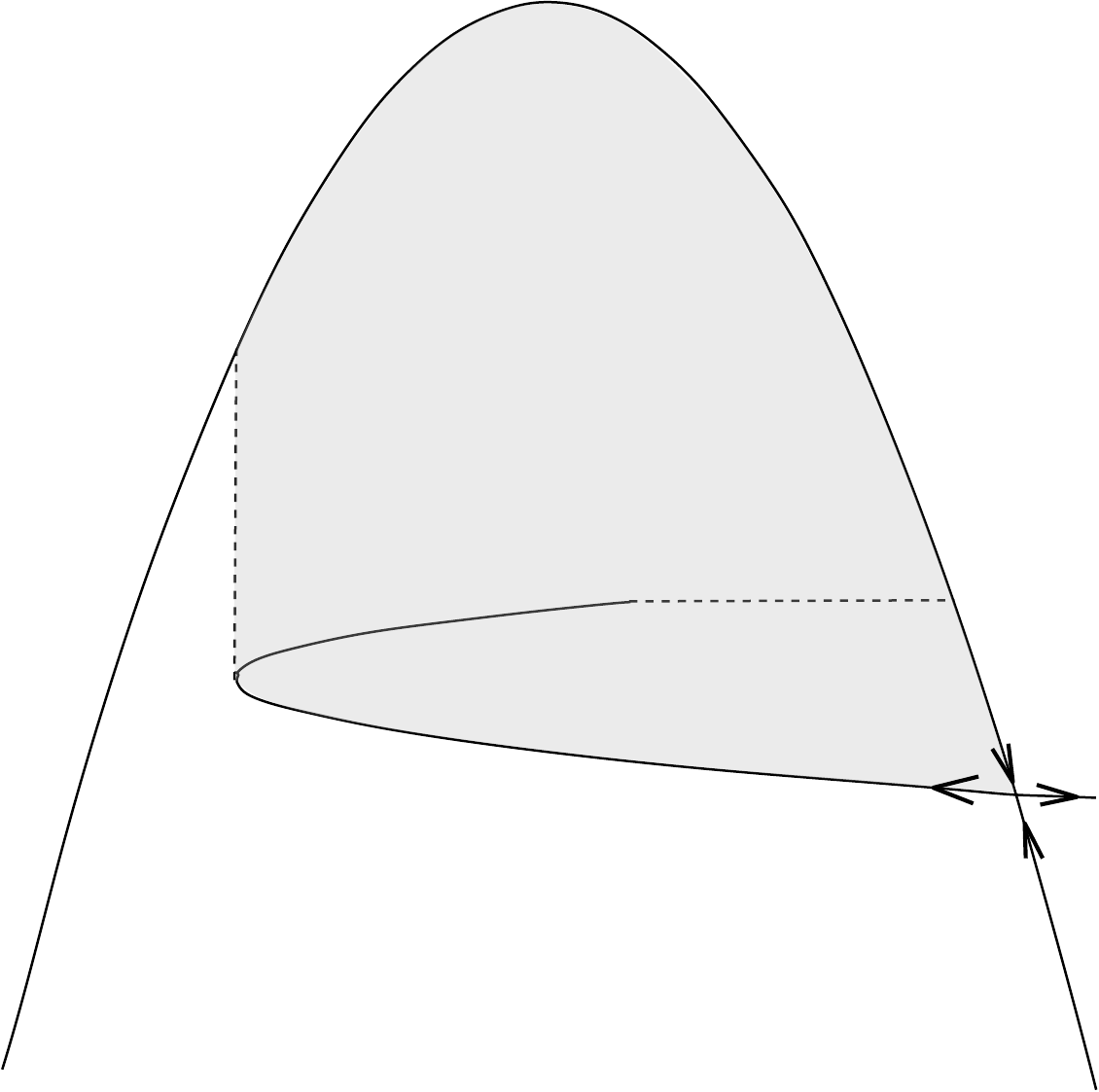}
\hspace{2cm}
\includegraphics[width=5cm,angle=0]{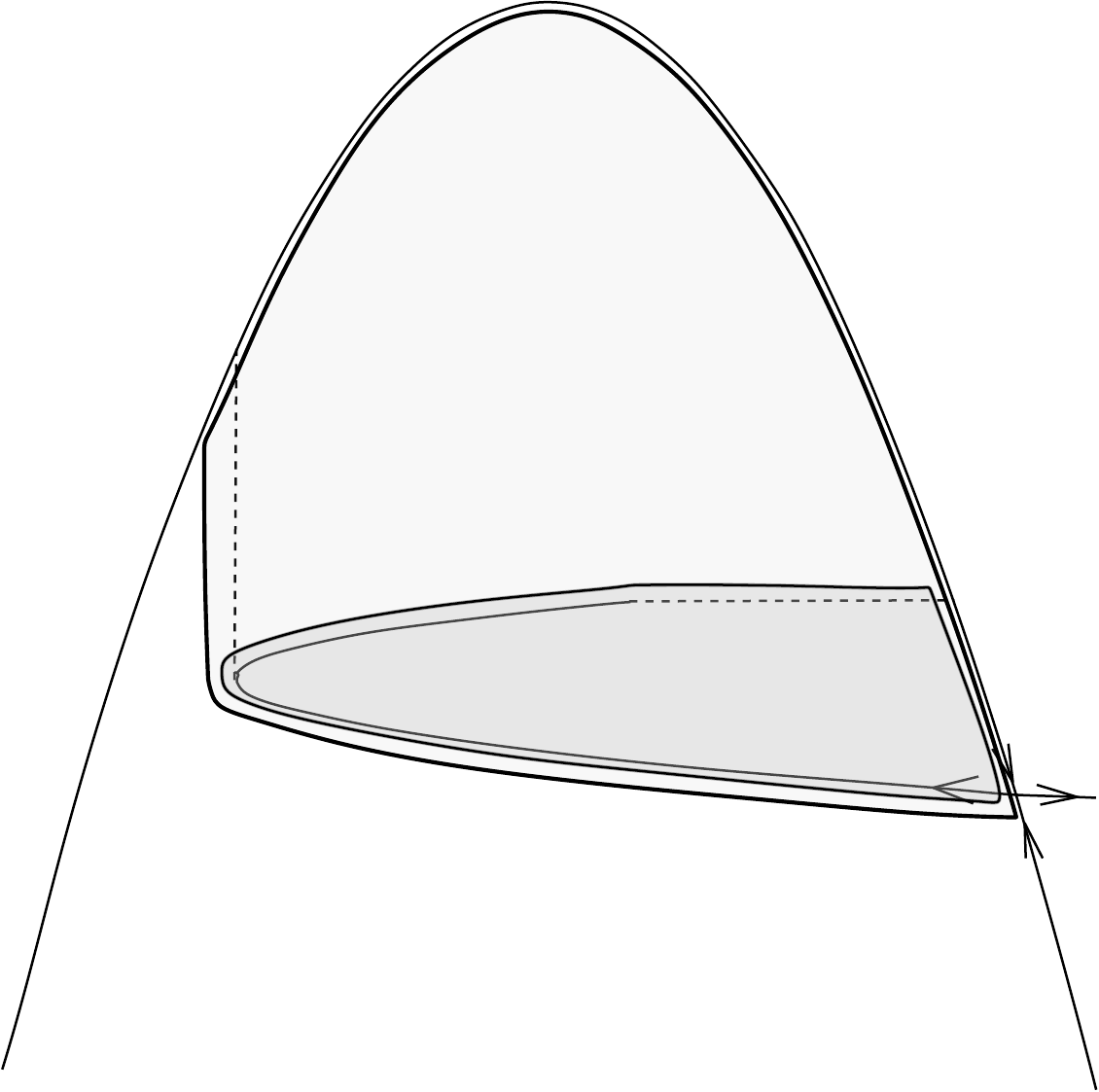}
\put(-212,31){\small $q$}
\put(-266,48){\small $f(\Delta)$}
\put(-325,43){\small $z_u$}
\put(-315,70){\small $\gamma^v$}
\put(-280,30){\small $\gamma^u$}
\put(-235,90){\small $\gamma^s$}
\put(-275,84){\small $\Delta$}
\put(-6,31){\small $q$}
\put(-49,48){\small $f(D)$}
\put(-73,84){\small $D$}
\caption{Construction of a trapped domain (when $b\in (0,1)$).\label{f.construction-trapped}}
\end{figure}

\paragraph{\it Third case: the left stable branch is not a graph and $x_s> x_u$.}
We perform a similar construction.
The local stable graph is bounded by a point $z_s$ with a vertical tangent space.
As before, the two local graphs are disjoint and one builds a Jordan domain $\Delta$ by considering the union of
a stable arc $\gamma^s$, a vertical segment $\gamma^v$ and an unstable arc $\gamma^u$ above $(x_s,x_q)$.
For the same reasons as before, the boundary of $\Delta$ does not cross its image.
In this case $f(\gamma^v)$ is an horizontal graph tangent to the convex graph $\gamma^s$ and hence above it.
This implies $f(\Delta)\supset \Delta$, contradicting the volume contraction of~$f$.
\end{proof}

{\color{black}
\printindex
}

\bigskip

\hspace{-2.5cm}
\begin{tabular}{l l l l l}
\emph{Sylvain Crovisier}
& &
\emph{Enrique Pujals}
& &
\emph{Charles Tresser}
\\

Laboratoire de Math\'ematiques d'Orsay
&& Graduate Center-CUNY
&& IMPA\\
CNRS - UMR 8628, Univ. Paris-Saclay
&& New York, USA
&& Estrada Dona Castorina, 110\\
Orsay 91405, France
&&
&& 22460-320 Rio de Janeiro, Brazil
\end{tabular}
 
\end{document}